\documentclass[lefttitle, 3p]{elsarticle}
\usepackage[english]{babel}
\usepackage[utf8]{inputenc}
\usepackage[T1]{fontenc}
\usepackage{subcaption}
\usepackage{enumerate}
\usepackage{hyperref}
\usepackage{xcolor}
\usepackage{algorithm}
\usepackage{tabularx}
\usepackage{multirow}
\usepackage{algorithmic}

\usepackage{mathtools, amsmath, amssymb, amsthm, mathrsfs}
\usepackage{bm} \usepackage{dsfont} 

\newcommand{\Prob}{\mathbb{P}}
\newcommand{\Ept}{\mathbb{E}}
\newcommand{\transp}{^{\mathsf T}}
\newcommand{\herm}{^{\ast}}

\newcommand{\Id}{\mathrm{Id}}

\newcommand{\field}[1]{\ensuremath{\mathds{#1}}}

\newcommand{\PhiYt}{\widetilde{\mathbf \Phi}_{M,m}}

\newcommand{\ba}{\bm {a}}

\newcommand{\by}{\bm {y}}
\newcommand{\bu}{\bm {u}}
\newcommand{\bc}{\bm {c}}
\newcommand{\bv}{\bm {v}}
\newcommand{\bw}{\bm {w}}
\newcommand{\bA}{\bm {A}}

\newcommand{\bX}{\bm {X}}
\newcommand{\bY}{\bm {Y}}

\newcommand{\bM}{\bm {M}}
\newcommand{\bH}{\bm {H}}
\newcommand{\bL}{\bm {L}}
\newcommand{\bI}{\bm {I}}
\newcommand{\bD}{\bm {D}}
\newcommand{\bU}{\bm {U}}
\newcommand{\bW}{\bm {W}}

\newcommand{\cN}{\mathcal{N}}

\newcommand{\signum}{{\ensuremath{\mathrm{sgn}} \,}}

\numberwithin{equation}{section}
\numberwithin{table}{section}
\numberwithin{figure}{section}
\newcommand{\numberthis}{\addtocounter{equation}{1}\tag{\theequation}}
\DeclareMathOperator*{\spn}{span}

\DeclareMathOperator*{\diag}{diag}
\DeclareMathOperator*{\tr}{tr}
\DeclareMathOperator*{\range}{R}
\DeclareMathOperator*{\rank}{rank}
\DeclareMathOperator*{\argmin}{arg\,min}
\DeclareMathOperator{\lamax}{\lambda_{\mathrm max}}
\DeclareMathOperator{\lamin}{\lambda_{\mathrm min}}

\newtheorem{manuallemmainner}{Lemma}
\newenvironment{manuallemma}[1]{\renewcommand\themanuallemmainner{#1}\manuallemmainner
}{\endmanuallemmainner}

\newtheorem{manualpropinner}{Proposition}
\newenvironment{manualprop}[1]{\renewcommand\themanualpropinner{#1}\manualpropinner
}{\endmanualpropinner}

\newtheorem{manualtheoreminner}{Theorem}
\newenvironment{manualtheorem}[1]{\renewcommand\themanualtheoreminner{#1}\manualtheoreminner
}{\endmanualtheoreminner}

\newtheorem{theorem}{Theorem}[section]
\newtheorem{lemma}[theorem]{Lemma}
\newtheorem{remark}[theorem]{Remark}
\newtheorem{definition}[theorem]{Definition}
\newtheorem{corollary}[theorem]{Corollary}

\begin{document} \title{Constructive subsampling of finite frames with applications in optimal function recovery}

\author{Felix~Bartel}
\ead{felix.bartel@mathematik.tu-chemnitz.de}
\author{Martin~Schäfer}
\ead{martin.schaefer@mathematik.tu-chemnitz.de}
\author{Tino~Ullrich\corref{cor1}}
\ead{tino.ullrich@mathematik.tu-chemnitz.de}

\cortext[cor1]{Corresponding author}
\address{Chemnitz University of Technology, Faculty of Mathematics, 09107 Chemnitz, Germany}

\begin{frontmatter}

\begin{abstract}
In this paper we present new constructive methods, random and deterministic, for the efficient subsampling of finite frames in $\field{C}^m$. Based on a suitable random subsampling strategy, we are able to extract from any given frame  with bounds $0<A\le B<\infty$ (and condition $B/A$) a similarly conditioned reweighted subframe consisting of merely $\mathcal{O}(m\log m)$ elements.
Further, utilizing a deterministic subsampling method based on principles developed by Batson, Spielman, and Srivastava
to control the spectrum of sums of Hermitian rank-1 matrices, we are able to reduce the number of elements to $\mathcal{O}(m)$
(with a constant close to one). By controlling the weights via a preconditioning step, we can, in addition,
preserve the lower frame bound in the unweighted case. This permits the derivation of new quasi-optimal unweighted (left) Marcinkiewicz-Zygmund inequalities for $L_2(D,\nu)$ with constructible node sets of size $\mathcal{O}(m)$ for $m$-dimensional subspaces of bounded functions. Those can be applied e.g.\ for (plain) least-squares sampling reconstruction of functions, where we obtain new quasi-optimal results avoiding the Kadison-Singer theorem.
Numerical experiments indicate the applicability of our results.
\end{abstract}


\end{frontmatter}

\section{Introduction}

The paper mainly deals with the question of how to choose a well-conditioned subframe out of a given frame. The notion of a frame goes back to Duffin and Schaeffer~\cite{DuScha52}. Let $H$ be a complex Hilbert space with scalar product $\langle\cdot,\cdot\rangle$ and norm $\|\cdot\|$. A countable subset $(\bm y^i)_i$ in $H$ is said to be a \emph{frame} if there are constants $0<A\le B<\infty$ such that
\begin{align}\label{frame}
    A\|\bm x\|^2 \leq \sum_{i} |\langle \bm x, \bm y^i\rangle|^2 \leq B\|\bm x\|^2
    \quad\text{for all}\quad
    \bm x\in H\,.
\end{align}
We are mostly interested in frames consisting of finitely many elements in some finite dimensional Hilbert space $H$, see e.g.\ Casazza and Kutyniok~\cite{CK13} for an introduction to finite frame theory.
Systems of this kind may be represented by $(\bm y^i)_{i=1}^M \subset \mathds C^{m}$ with $M\ge m$. 
The question of finding good subframes in such a system is rather fundamental and important for many applications ranging from graph sparsifiers \cite{BaSpSr09,SpSr11}, the Kadison-Singer problem \cite{MaSpSr15,Wea2004}, to optimal discretization and sampling recovery of multivariate functions \cite{DKU22,KrUl19,KrUl20,NaSchUl20,LiTe20,Temlyakov20,PU21}. In this context, let us also mention the possibility of generating `approximations' of Hadamard matrices, a problem which has been considered in~\cite{DoRu22} for example. Subsampling of a tight Hadamard frame $(\bm y^i)_{i=1}^M$, where all entries of the $\bm y^i$ are $\pm 1$ (or in a complex setting of modulus $1$), may lead to an almost square Hadamard-type matrix with good condition. 

As a first step, we put forward a simple random subsampling strategy in \textbf{Section~\ref{sec:weighted_random}}.
For the index set $\{1,\ldots,M\}$ of the intial system let us subsequently use the short-hand notation $[M]$. 
Theorem~\ref{thm:frame_sub} shows that drawing elements $\by^i$, $i\in[M]$, according to the probabilities $\varrho_i:=\|\by^i\|^2_2/(\sum_j \|\by^j\|^2_2)$ yields a
reweighted subframe $(\varrho_i^{-1/2} \by^i)_{i\in J}$ with similar frame bounds, with high probability provided $|J|=\mathcal{O}(m\log m)$.
In terms of computational complexity this strategy is very efficient.
It is not optimal with respect to the number of frame elements, however.

This shortcoming is dealt with in \textbf{Section~\ref{sec:weighted_bss}}.
Here we formulate a deterministic algorithm (\texttt{BSS}) which yields reweighted subframes with an optimal order of $\mathcal{O}(m)$ frame elements. It is an extension of a subsampling method due to Batson, Spielmann, Srivastava \cite{BaSpSr09} to the complex and non-tight case, which in the original version only applies to tight frames in $\field{R}^m$. Let us remark that an extension to tight frames in $\field{C}^m$ was already considered in~\cite[Cor.~2.1]{LiTe20}. We further modify the subsampling method 
in~\cite{BaSpSr09} by allowing variable barrier shifts.
As proved in Theorem~\ref{BSS}, the \texttt{BSS} algorithm constructs in polynomial time
for any frame $(\bm y^i)_{i=1}^M \subset \mathds C^{m}$ with constants $0<A\le B<\infty$ and $b>\kappa^2\ge1$, $\kappa=\kappa(A,B)$ being the value in \eqref{eqdef:kappa}, a subset $J \subset [M]$ and weights $s_i\geq 0$ such that 
\begin{align}\label{int:weightedframe}
    A\|\bm x\|_2^2 \leq \sum_{i\in J} s_i|\langle \bm x, \bm y^i\rangle|^2 \leq \gamma\cdot B\|\bm x\|_2^2
    \quad\text{for all}\quad \bm x \in \field{C}^m
\end{align}
with  $|J| \leq \lceil bm \rceil$ and $\gamma = \gamma(b,\kappa)$ as in \eqref{eq:gamma}. 
Combined with a suitable `preconditioning' step, it
can even achieve~\eqref{int:weightedframe} for any $b>1$ and with the better constant $\gamma=\gamma(b,1)$. In this variant (\texttt{BSS$^\perp$}), at the core, \texttt{BSS} is only applied to a tight frame. Since exact tightness can usually not be guaranteed in practice, due to numerical inaccuracies, it is nevertheless important for applications that \texttt{BSS} works stably also for non-tight frames, as ensured by Theorem~\ref{BSS}.

Sections~\ref{sec:weighted_random} and~\ref{sec:weighted_bss} lay the groundwork for the main \textbf{Section~\ref{sec:unweighted}}, where we are concerned with the extraction of unweighted subframes.
This is a much more difficult task.  The existence of similarly conditioned subframes consisting of order $\mathcal{O}(m)$ elements is guaranteed by the solution of the famous Kadison-Singer problem by Markus, Spielmann, Srivastava~\cite{MaSpSr15}. To the knowledge of the authors, there are no general 
constructive polynomial-time methods available, however. Our approach to tackle this problem is
to use the obtained results on weighted subframes (e.g.\ \eqref{int:weightedframe}) and try to control 
the weights~$s_i$. By bounding those from above, we are able to preserve the lower frames bounds,
which for many applications are the relevant ones. 

One of the main results is Theorem~\ref{thm:unweightes_bss}.
In Corollary~\ref{frame:smallb}, we obtain the non-weighted sparsification inequality~\eqref{f17}, which holds true for arbitrary sets of vectors $(\bm y^i)_{i=1}^M \subset \mathds C^{m}$. 
For any $1+\frac{1}{m}<b\le \frac Mm$ we can extract a subset $J \subset [M]$ of cardinality $|J|\le\lceil bm \rceil$ such that 
\begin{equation}\label{f17}
    \frac{1}{M} \sum\limits_{i = 1}^M |\langle {\bm x},{\bm y}^i \rangle|^2 \leq C_0\frac{b^3}{(b-1)^{3}}\frac{1}{|J|}\sum\limits_{i \in J} |\langle {\bm x},{\bm y}^i \rangle|^2
    \quad\text{for all}\quad\bm x \in \field{C}^m \,,
\end{equation}
where $C_0>0$ is an absolute constant.
The lower frame bound of an initial frame is thus preserved up to a constant $(b/(b-1))^{3}$. Earlier results in this direction, see Harvey, Olver \cite{HaOl14}, Nitzan, Olevskii, Ulanovski \cite{NiOlUl16}, Limonova, Temlyakov \cite{LiTe20}, and Nagel, Sch\"afer, T.~Ullrich \cite{NaSchUl20}, are all non-constructive and have been initiated by the solution of the Kadison-Singer problem~\cite{MaSpSr15} in the form of Weaver's conjecture \cite{Wea2004}. These results need additional restrictions on the size of the frame elements. They provide the upper frame bounds as well, however. 

A \textbf{central message} in this paper is the fact that, at least for the lower bounds, 
we can argue much more elementary and do not need such deep results as Kadison-Singer.
Our approach is semi-constructive (with an at times probabilistic component) and yields polynomial-time algorithms (e.g.\ \texttt{plainBSS}) that work with high probability and can be efficiently implemented.
We developed a corresponding \textsc{Julia}-package, availabe at \url{www.github.com/felixbartel/BSSsubsampling.jl}, 
and conducted numerical experiments with this code.
A few results are presented in~\textbf{Section~\ref{sec:numerics}}.

Some applications of our obtained subsampling results are discussed in 
\textbf{Section~\ref{Discussion}}.  A first interesting consequence of \eqref{f17} is the non-weighted left Marcinkiewicz-Zygmund inequality given by Theorem~\ref{theorem:finite}. Assume we are given an
$m$-dimensional space
\begin{align}\label{int:subspaceVm}
	V_m = \text{span}\{\eta_1(\cdot),\cdots,\eta_m(\cdot)\}
\end{align}
of complex-valued functions on some non-empty set $D$. 
In case that $V_m\subset L_2(D,\nu)$ for a finite measure $\nu$, \eqref{f17} allows us to construct a set of nodes ${\mathbf{X}}_n = ({\bm x}^1,\ldots,\bm x^n) \in D^n$ with $n \le \lceil bm \rceil$ for any $1+\frac{1}{m}<b \le \frac{M}{m}$ in polynomial time (in $m$) such that 
\begin{equation}\label{f20}
	\|f\|_{L_2(D,\nu)}^2 \leq C_1 \frac{b^3}{(b-1)^{3}}\frac{1}{n}\sum\limits_{i = 1}^n |f({\bm x}^i)|^2
	\quad\text{for all}\quad f\in V_m 
	\,.
\end{equation}
Here $C_1>0$ is again an absolute constant.
Inequalities like~\eqref{f20} have direct consequences for the sampling recovery of functions. 
With $\nu$ and $V_m$ as before,  
let us consider a simple plain least squares recovery operator $S_{V_m}^{\bX_n}$ for 
nodes $\bX_n$ satisfying~\eqref{f20}. It reconstructs 
any $\nu$-measurable function $f:D\to \field{C}$ in $V_m$ via a plain least squares minimization at the nodes $\bX_n$ and, according to Theorem~\ref{thm:sampl}, with  
\begin{equation}\label{intro:f21}
	\|f-S_{V_m}^{{\mathbf{X}}_n}f\|_{L_2(D,\nu)}^2 \leq C_2 \frac{b^3}{(b-1)^{3}} e(f,V_m)_{\ell_\infty(D)}^2\,,
\end{equation}
where $e(f,V_m)_{\ell_\infty(D)} = \inf_{g\in V_m}\|f-g\|_{\ell_\infty(D)}$ is the error of best approximation of $f$ from $V_m$. 

Inequalities of this type have been first established by Cohen and Migliorati~\cite{CoMi16}, but with a larger number of samples, namely $n = \mathcal{O}(m\log(m))$. This has been improved by Temlyakov \cite{Temlyakov20} to $n = \mathcal{O}(m)$ samples with unspecified constants. The mentioned results rely on weighted least squares algorithms, however, and Temlyakov posed the 
question in~\cite{Temlyakov20}, if also classical plain least squares methods could be used.
Theorem~\ref{thm:sampl} gives an affirmative answer and even displays the dependence of the constant on the oversampling factor $b$.
As a consequence we obtain from~\eqref{intro:f21} for classes of bounded complex-valued functions $F\subset \ell_\infty(D)$, when optimizing over all $m$-dimensional reconstruction spaces, the new relation
\begin{equation}\label{intro:newrelation}
	g_{\lceil bm \rceil,m}^{\mathrm{ls}}(F,L_2(D,\nu)) \leq C_3 \frac{b^{3/2}}{(b-1)^{3/2}} d_m(F,\ell_{\infty}(D))
\end{equation}
between the $m$-th Kolmogorov number $d_m$ of the class $F$ (see~\eqref{f100} for a definition) and the plain sampling numbers $g_{n,m}^{\mathrm{ls}}(F,L_2(D,\nu))$ defined in \eqref{sampling_numbers}. Here recovery is restricted to canonical plain least squares operators using $n$ samples acting on subspaces of dimension $m$. Since the quantities on the left-hand side of \eqref{intro:newrelation} 
are in general larger than the standard sampling numbers \cite[(5.0.1)]{DuTeUl19}, where there are no such restrictions on the recovery, this slightly improves on recent results by Temlyakov, Theorems~1.1 and 1.2 in~\cite{Temlyakov20} 
as well as~\cite[Thm.~3.4]{LiTe20}, the latter joint work with Limonova. Interestingly, the $b$-dependent constant may be improved to $(b-1)^{-1}$ when allowing weighted least squares algorithms, cf.~\cite[Thm.~1.7]{LiTe20} (or the original~\cite[Thm.~6.3]{DPSTT21}) for the case of real functions (an extension to the complex case has been given in \cite[Rem.~3.2]{LiTe20} but only for $b>2$). In our case a distinction between real and complex $L_2(D,\nu)$ in~\eqref{intro:newrelation} as in \cite{LiTe20} is unnecessary due to the validity of Theorem~\ref{BSS} in the complex setting. Note that the right-hand side in~\eqref{intro:newrelation} is of particular importance if the linear widths in $L_2$ are not square-summable \cite{TeUl21_1,TeUl21_2}.

A related scenario is investigated in the recent papers \cite{DKU22,KrUl19,KaVoUl21,MoUl20,NaSchUl20}. Here one is interested in the recovery of functions from a reproducing kernel Hilbert space (RKHS) $H(K)$ in $L_2(D,\nu)$, where $\nu$ is allowed to be infinite.
If we assume $H(K)$ to fulfill some natural assumptions (such as a finite trace of the kernel $K$ and a compact embedding $\Id_{K,\nu}:H(K)\hookrightarrow L_2(D,\nu)$), our polynomial-time
subsampling schemes allow to construct node sets $\bX_n$ with $n\le\lceil bm \rceil$,  $1+\frac{1}{m}<b\le 2$, and a weight function $w_m$ such that 
\begin{equation}\label{eq:perfRKHS}
\|f-S^{\bX_n}_{V_m,w_m}f\|^2_{L_2(D,\nu)} \leq C_4 \frac{1}{(b-1)^3} \log\Big(\frac{m}{p}\Big) \Big(\sigma_{m+1}^2 + \frac{7}{m}\sum\limits_{k=m+1}^{\infty} \sigma_k^2 \Big) \|f\|^2_{H(K)}
	\end{equation}
for every $f\in H(K)$, with a probability exceeding $1-\frac{3}{2}p$ for each $p\in(0,\frac{2}{3})$ (see Theorem~\ref{thm:approx}). Here, $C_4>0$ is an absolute constant, $\sigma_1 \geq \sigma_2 \geq \cdots \geq 0$ denote the singular numbers of the embedding $\Id_{K,\nu}$, and the recovery operator $S^{\bX_n}_{V_m,w_m}$ is a weighted 
least squares operator for the subspace $V_m$, as in~\eqref{int:subspaceVm}, spanned by the left singular functions corresponding to the $m$ largest singular numbers. The performance~\eqref{eq:perfRKHS} is near-optimal as in~\cite{NaSchUl20}. The latter reference is the first which used the Weaver subsampling technique for the sampling recovery problem. However, it does not achieve the optimal rate. By a further refinement of the technique, established very recently in~\cite{DKU22} by Dolbeault, Krieg, and M.~Ullrich, the optimal rate (without additional $\log$-term) has been found. In contrast to~\cite{DKU22,NaSchUl20} we have a semi-constructive method to generate the sampling nodes (offline step) that does not need the Kadison-Singer theorem in terms of the Weaver subsampling. In addition, the dependence on the oversampling factor $b$ is displayed. 
An \textbf{open question} remains.
Although in many relevant cases (like periodic Sobolev spaces with mixed smoothness) the recovery operator turns out to be a canonical plain least squares operator with equal weights (acting on the hyperbolic cross frequency subspace with nodes displayed in Figure \ref{fig:d2_sparse_grid}) we do not know whether this is possible in general.  

Throughout the paper, we will use the following \textbf{notation}. 
As usual, $\field{N}$, $\field{Z}$,  $\field{R}$, $\field{C}$  denote the natural (without $0$), integer, real, and complex numbers. If not indicated otherwise $\log(\cdot)$ denotes the natural logarithm. For $m\in\field{N}$ we further set $[m]:=\{1,\ldots,m\}$ and $\field{N}_{\ge m}:=\{m,m+1,\ldots\}$. $\field{C}^n$ shall denote the complex $n$-space and $\field{C}^{m\times n}$ the set of complex $m\times n$-matrices. Vectors and matrices are usually typesetted boldface.
For a vector $\by\in \field{C}^n$ we introduce the tensor notation $\by \otimes \by$ for the 
matrix $\by\cdot \by^\ast \in\field{C}^{n\times n}$, where $\by^\ast:=\overline{\by}^\top$. 
More general, the adjoint of a matrix $\bL\in\field{C}^{m\times n}$ is denoted by $\bL^{\ast}$. For the spectral norm we use $\|\bL\|$ or $\|\bL\|_{2\to 2}$ and we use $A \preceq B$ to denote that $B-A$ is positive semi-definite.
Finally, we will write $\Ept(X)$ for the expectation of a random variable $X$ and $\Prob(E)$ for the probability of an event $E$. In our case the probability measure $\Prob$ is a product measure $\mu^{\otimes n}$ (resp.~$\varrho_m^{\otimes n}$) on $D^n$ with a certain probability measure $\mu$ (resp.~$\varrho_m$) on $D$, for both discrete and continuous domains $D$. The abbreviation i.i.d.\ refers to `independent and identically distributed'. For a set $D$ we denote with $\ell_\infty(D)$ the set of all bounded complex-valued functions on $D$. If $D$ is $\nu$-measurable we denote with $L_2(D,\nu)$ the space of all $\nu$-measurable square-integrable functions (equivalence classes) on $D$.

\section{Random weighted subsampling of finite frames}\label{sec:weighted_random}

We begin with a random subsampling strategy that allows to extract `good' subframes of $\mathcal{O}(m\log m)$ elements out of any given frame in $\field{C}^m$.
This goes back to Rudelson and Vershynin \cite{RuVer07}, see also Spielman and Srivastava \cite{SpSr11}, where the goal was to efficiently find a low rank approximation of a given matrix such that the error with respect to the spectral norm remains small.
The method is rather simple since it relies on a random subselection where the discrete probability mass $\varrho_i$ for selecting one particular frame element $\by^i$ (see Theorem~\ref{thm:frame_sub}) is directly linked to its contribution to the sum of the norms, i.e., the Frobenius norm $\|\bm Y\|_F^2$ of the matrix
\begin{align}\label{eq:analysisoperator}
  \bm Y 
:= \begin{bmatrix}
  (\bm y^1)\herm\\[-1ex] \hrulefill \\[-1ex] \vdots\\[-1ex]\hrulefill \\ (\bm y^M)\herm
  \end{bmatrix}
  \in\mathds C^{M\times m} \,.
\end{align}
Note that for a given frame $(\bm y^i)_{i = 1}^M \subset \field{C}^m$, with $M\ge m$ and $m\in\field{N}$, this matrix represents the analysis operator
of the frame and that 
\begin{align}\label{chain}
    {m}A  \leq \tr(\bm Y^* \bm Y) 
    = \|\bm Y\|_F^2 
    \le {m}\|\bm Y^* \bm Y\|_{2\to 2} 
    = m \lambda_{\max}(\bm Y^* \bm Y) 
    \le {m}B\,.  
\end{align}

Our main result of this section
relies on a matrix Chernoff bound proven by Tropp~\cite[Thm.~1.1]{Tr11} (see Theorem~\ref{matrixchernoff} in the Appendix). 
It shows how one can randomly subsample a finite frame of arbitrary size in $\field{C}^m$ to a weighted subframe with $\mathcal{O}({m}\log {m})$ elements while essentially keeping its stability properties. 

\begin{theorem}\label{thm:frame_sub} Let $(\by^i)_{i=1}^M\subset\field{C}^{m}$ be a frame with constants $0<A\le B<\infty$ (see~\eqref{frame}).
Let further $p,t \in (0,1)$ and $n\in\field{N}$ be such that 
\begin{align*}
  n \ge \frac{3B}{At^2}m\log\left(\frac{2m}{p}\right) \,.
\end{align*}
Drawing $n$ indices $J\subset[M]$ (with duplicates) i.i.d.\ according to the discrete probability density $\varrho_i = \|\bm y^i\|_2^2 / \|\bm Y\|_F^2$, $i\in[M]$, then gives a rescaled random subframe $(\varrho_i^{-1/2}\bm y^i)_{i\in J}$ such that 
\begin{align*}
    (1-t)A\|\bm a\|_2^2 
    \leq \frac{1}{n}\sum_{i\in J} 
    \left|\left\langle \bm a, \varrho_i^{-1/2} \bm y^i \right\rangle\right|^2 \leq (1+t)B\|\bm a\|_2^2
    \quad\text{for all}\quad
    \bm a\in \field{C}^{m}
\end{align*}
with probability exceeding $1-p$.
\end{theorem} 

\begin{proof} The result is a direct consequence of Tropp's result in Lemma~\ref{matrixchernoff}. For a randomly chosen index $i \in [M]$ we define the rank-one random matrix 
$ \bA_i := \frac{1}{n}\varrho^{-1}_i (\bm y^i \otimes \bm y^i)$.
Clearly, it holds
\begin{align*}
    \lamax(\bA_i) = \lambda_{\max}\left(\frac{1}{n}\varrho^{-1}_i (\bm y^i \otimes \bm y^i)\right) 
    = \frac{1}{n}\varrho_i^{-1}\|\bm y^i\|_2^2 
    = \frac{1}{n}\|\bm Y\|_F^2\,.
\end{align*}
Furthermore, having $n$ independent copies $(\bA_i)_{i\in J}$, we obtain
\begin{align*}
  \sum_{i\in J} \mathds E \bA_i  
  = \sum_{i\in J} \mathds E \left(\frac{1}{n}\varrho^{-1}_{i} (\bm y^{i} \otimes \bm y^{i})\right) 
  = \sum_{i\in J} \frac{1}{n} \bm Y^* \bm Y
  = \bm Y^* \bm Y\,.
\end{align*}
This gives for $\mu_{\min}:=\lamin(\sum_{i\in J} \mathds E \bA_i)$ and $\mu_{\max}:=\lamax(\sum_{i\in J} \mathds E \bA_i)$ that 
\begin{align*}
   \mu_{\min} = \lambda_{\min}(\bm Y^* \bm Y) \geq A
    \quad\text{and}\quad
     \mu_{\max}= \lambda_{\max}(\bm Y^* \bm Y) 
    = \|\bm Y\herm\bm Y\|_{2\to 2}^2
    \,.
\end{align*}
Since $\|\bm Y\|^2_F= \tr(\bm Y^* \bm Y)$, Lemma~\ref{matrixchernoff} and \eqref{chain} gives  
\begin{align*}
    \mathds P \left(\lambda_{\max}\left(\frac{1}{n}\sum_{i\in J} \varrho_{i}^{-1} \bm y^{i} \otimes \bm y^{i}\right) \geq (1+t)B\right)
    &\leq {m}\exp\left(-\frac{n\|\bm Y^*\bm Y\|_{2\to 2}}{\tr(\bm Y^* \bm Y)}\frac{t^2}{3}\right)\\
    &\leq m\exp\left(-\frac{n}{{m}}\frac{t^2}{3}\right)\,.
\end{align*}
For the smallest eigenvalue things are a bit different. Here we obtain
\begin{align*}
    \mathds P \left(\lambda_{\min}\left(\frac{1}{n}\sum_{i\in J} \varrho_{i}^{-1} \bm y^{i} \otimes \bm y^{i}\right) \leq (1-t)A\right)
    &\leq m\exp\left(-\frac{nA}{\tr(\bm Y^* \bm Y)}\frac{t^2}{2}\right)\\
    &\leq m\exp\left(-\frac{An}{Bm}\frac{t^2}{2}\right)\,.
\end{align*}
For the probability of our assertion we need the complement of the two events above:
\begin{align*}
  1-m\exp\left(-\frac{n}{{m}}\frac{t^2}{3}\right)
  - m\exp\left(-\frac{An}{Bm}\frac{t^2}{2}\right)
  \ge 1-2m\exp\left(-\frac{An}{{Bm}}\frac{t^2}{3}\right) \ge  1-p \,,
\end{align*}
which follows from the assumption on $n$.
\end{proof} 

\begin{remark} The rescaled random subframe $(\varrho_i^{-1/2}\bm y^i)_{i\in J}$ in Theorem~\ref{thm:frame_sub} is an equal-norm frame.
Thus, starting with a tight frame, we are able to construct an `almost tight' frame with unit-norm (UNTF).
These are important in robust data transmission and have proven notoriously difficult to construct, cf.~\cite{CFM12, CK03}.
\end{remark}

\section{Deterministic weighted subsampling of finite frames}\label{sec:weighted_bss}

We next present a deterministic subsampling algorithm for finite frames in $\field{C}^m$ which we subsequently call (generalized) \texttt{BSS} algorithm.
A version for real-valued tight frames in $\field{R}^m$ was originally introduced by Batson, Spielman, and Srivastava in the context of graph sparsification~\cite{BaSpSr09}.
It allows to extract from any given finite frame in $\field{C}^m$ a comparably well-conditioned
re-weighted subframe of cardinality $\mathcal{O}(m)$. This is the statement of Theorem~\ref{BSS} below which generalizes~\cite[Thm.~3.1]{BaSpSr09}.

In contrast to related non-weighted subsampling results, such as e.g.~\cite[Thm.~2.3]{NaSchUl20} which are all based on Weaver's theorem, a deep result equivalent
to the famous Kadison-Singer theorem~\cite{MaSpSr15}, the proof of Theorem~\ref{BSS} is elementary and constructive. The underlying \texttt{BSS} algorithm lends itself to practical polynomial time implementation. 

\begin{theorem}\label{BSS} Let $(\bm y^i)_{i=1}^M \subset \mathds C^{m}$ be a frame \eqref{frame} with frame constants $0<A\le B<\infty$ and let $b>\kappa^2\ge1$ with
\begin{align}\label{eqdef:kappa}
    \kappa:= \left(\frac{B}{2A} +\frac 12\right)+\sqrt{\left(\frac{B}{2A}+\frac 12\right)^2-1}\,.
\end{align}
Then the \texttt{BSS} algorithm in Subsection~\ref{ssec:implementation} computes a subset $J \subset [M]$ with $|J| \leq \lceil bm \rceil$ and nonnegative weights $s_i$, $i\in J$, such that
\begin{flalign*}
  &&  A\|\bm a\|_2^2 \leq \sum_{i\in J} s_i|\langle \bm a, \bm y^i\rangle|^2 \leq \gamma\cdot B\|\bm a\|_2^2 \quad\text{for all}\quad
	\bm a\in \field{C}^{m} && \numberthis\label{framecondsub}\\
\text{with} &&
    \gamma := \frac{(\sqrt{b}+1)^2}{(\sqrt{b}-1)(\sqrt{b}-\kappa)}\,. \qquad\qquad\qquad  && \numberthis\label{eq:gamma}
\end{flalign*}
\end{theorem}

\begin{remark} \begin{enumerate}[(i)]
\item The \texttt{BSS} algorithm computes the index subset $J$ and the corresponding weights $s_i$ in $\mathcal{O}(bMm^3)$. An implementation and runtime analysis is given in
Subsection~\ref{ssec:implementation}, see also \cite[Sec.3]{BaSpSr09}. Better guarantees on the bound
can be obtained by a `preconditioning' of the frame, given by Lemma~\ref{discreteconstruction1}. The resulting algorithm is called \texttt{BSS$^\perp$.} In \texttt{BSS$^\perp$} also the restriction $b>\kappa^2$ can be evaded. Some empirical results are presented in Section~\ref{sec:numerics}.
\item The theorem neither gives control over the weights $s_i$ nor provides an unweighted version of itself. The latter would actually be useful for applications. We refer to \cite{NaSchUl20} for an unweighted result which is called `Weaver subsampling' and relies on the Kadison-Singer theorem~\cite{MaSpSr15}. In Section~\ref{sec:unweighted} below we will use a special construction from Lemma~\ref{discreteconstruction} to deduce an unweighted version that preserves the left frame bound, cf.\ Corollary~\ref{frame:smallb}.
\end{enumerate}
\end{remark}

\subsection{The principal structure of the \texttt{BSS} algorithm}
\label{ssec:SketchProof}

The frame property of the vectors $(\bm y^i)_{i=1}^M$ can be formulated as
\begin{align*}A\bI \preceq \sum_{i=1}^{M} \by^i(\by^i)^\ast \preceq B \bI \,,
\end{align*}
where $\bI$ denotes the identity matrix in $\field{C}^{m\times m}$ (see the notation paragraph for the meaning of $\preceq$).
Furthermore, condition~\eqref{framecondsub} of the subsampled frame can be rewritten as
\begin{align}\label{tgamma}
	A \bI \preceq \sum_{i\in J} s_i\by^i(\by^i)^\ast  \preceq \gamma\cdot B \bI \,.
\end{align}

The idea of the \texttt{BSS} algorithm is to build the sum $\sum_{i\in J} s_i\by^i(\by^i)^\ast$ iteratively
in $n:=\lceil bm\rceil$ steps.
Starting with the zero-matrix $\bA^{(0)}:=0$, a sequence of Hermitian matrices
\begin{align}\label{iteration_matrices}
\bA^{(0)},\, \bA^{(1)},\,\bA^{(2)},\,\ldots,\,\bA^{(n)}
\end{align}
is computed via rank-1 updates of the form
\begin{align*}\bA^{(k)} = \bA^{(k-1)} + t^{(k)}\by^{i^{(k)}}(\by^{i^{(k)}})^\ast \,,\quad k\in [n] \,,
\end{align*}
with suitably selected indices $i^{(k)}\in [M]$ and weights $t^{(k)}>0$.
After $n$ iterations we have thus constructed a matrix $\bA^{(n)}$ of the form
\begin{align}\label{matrixform}
\bA^{(n)} = \sum_{k=1}^n t^{(k)}\by^{i^{(k)}}(\by^{i^{(k)}})^\ast =  \sum_{i\in J} \tilde{s}_i\by^i(\by^i)^\ast
\end{align}
where \[
\tilde{s}_i:=\sum\limits_{k:i^{(k)}=i} t^{(k)} \quad\text{and}\quad J:=\Big\{ i^{(k)} : k=1,\ldots,n \Big\} \,.  \]
Clearly, $|J|\le n = \lceil bm\rceil$.
During the whole process the spectra of the constructed matrices $\bA^{(k)}$ are controlled by means of so-called spectral barriers,
i.e., numbers $l^{(k)},u^{(k)}\in\field{R}$ such that
\begin{align}\label{reg}
\sigma(\bA^{(k)})\subset(l^{(k)},u^{(k)}) \,,\quad k\in \{0,\dots, n\} \,.
\end{align}

Whereas the precise location of the eigenvalues of $\bA^{(k)}$ may not be known, in this way we have enclosed their location in
open intervals $(l^{(k)},u^{(k)})$, in particular it holds
\begin{align*}l^{(k)} \bI \preceq \bA^{(k)} \preceq u^{(k)} \bI  \,.
\end{align*}

The algorithm starts with initial barriers $l^{(0)}<0$ and $0<u^{(0)}$ for $\bA^{(0)}=0$. From each step to the next, the barriers
are then shifted to the right, most simply by certain fixed lengths $\delta_L>0$ and $\delta_U>0$.
In the $k$th iteration we thus have $l^{(k)}=l^{(0)} + k\delta_L$ and $u^{(k)}=u^{(0)} + k\delta_U$
(see Figure~\ref{fig:barr_shifts}).
For each $\bA^{(k)}$ the indices and weights $i^{(k+1)}$ and $t^{(k+1)}$ are further chosen such that \eqref{reg} remains valid
for the updated matrix $\bA^{(k+1)}$. This is the main technical challenge and requires some preparation, carried out in Subsections~\ref{ssec:PotBar} and~\ref{ssec:Lem3.5}.
Under these conditions, the final matrix $\bA^{(n)}$ then has property~\eqref{reg} for
\[
l^{(n)}= l^{(0)} + n\delta_L \quad\text{and}\quad u^{(n)}= u^{(0)} + n\delta_U \,.
\]

As shown in Subsection~\ref{ssec:Thm2.1}, for the `right' choice of $l^{(0)}$, $u^{(0)}$, $\delta_L$, and $\delta_U$
we end up with final barriers satisfying 
\begin{align*}l^{(n)}>0 \quad\text{and}\quad \frac{u^{(n)}}{l^{(n)}} \le \gamma\cdot \frac{B}{A} \,.
\end{align*}
This finally allows to rescale the weights $\tilde{s}_i$ in~\eqref{matrixform} appropriately, giving the desired weights $s_i$ such that \eqref{tgamma} is fulfilled. 

\begin{figure}[ht]
 \centering
 \begin{subfigure}{0.8\textwidth}
 \includegraphics[width = \textwidth]{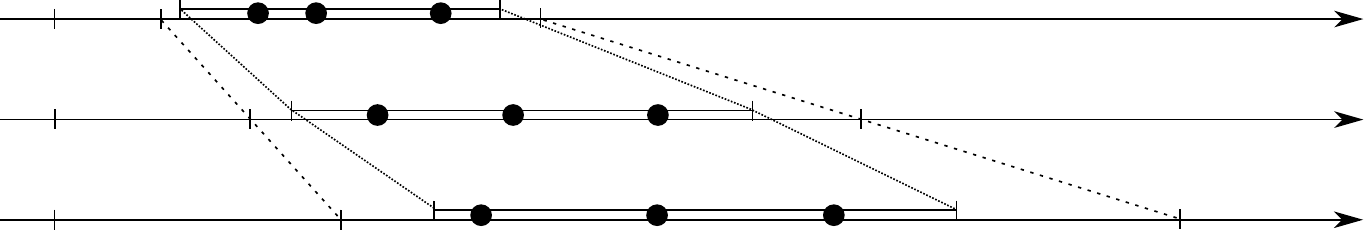}
 \put(-361,-6.6){$\scriptstyle{0}$}\put(-361,+20){$\scriptstyle{0}$}\put(-361,+47){$\scriptstyle{0}$}
 \put(-17,-5){$\scriptstyle{\field{R}}$}\put(-17,+22){$\scriptstyle{\field{R}}$}\put(-17,+50){$\scriptstyle{\field{R}}$}
 \put(-333,+48){$\scriptstyle{\textrm{\textbf l}^{(k-1)}}$}\put(-228,+48){$\scriptstyle{\textrm{\textbf u}^{(k-1)}}$}
 \put(-327,+64){$\scriptstyle{\mathit{l}^{(k-1)}}$}\put(-240,+64){$\scriptstyle{\mathit{u}^{(k-1)}}$}
 \put(-307,+21){$\scriptstyle{\textrm{\textbf l}^{(k)}}$}\put(-140,+21){$\scriptstyle{\textrm{\textbf u}^{(k)}}$}
 \put(-297,+37){$\scriptstyle{\mathit{l}^{(k)}}$}\put(-171,+37){$\scriptstyle{\mathit{u}^{(k)}}$}
 \put(-283,-7){$\scriptstyle{\textrm{\textbf l}^{(k+1)}}$}\put(-53,-7){$\scriptstyle{\textrm{\textbf u}^{(k+1)}}$}
\put(-258,10){$\scriptstyle{\mathit{l}^{(k+1)}}$}\put(-114,10){$\scriptstyle{\mathit{u}^{(k+1)}}$}
 \end{subfigure}
 \caption{Spectral shifting via \texttt{constant} and $\textit{variable}$ barrier shifts.}
 \label{fig:barr_shifts}
\end{figure}

\subsection{A concrete implementation and runtime analysis}
\label{ssec:implementation}

Before providing a profound theoretical basis for the \texttt{BSS} algorithm (and with that 
also a rigorous proof of Theorem~\ref{BSS}), starting in Subsection~\ref{ssec:PotBar}, let us first present 
a concrete numerical implementation. The subsequent version, Algorithm~\ref{algo1:BSS}, was implemented for
the purpose of empirical analysis (see Section~\ref{sec:numerics}). Instead of fixed barrier shifts 
$\delta_L$ and $\delta_U$ it uses variable shifts $\delta_L^{(k)}$ and $\delta_U^{(k)}$ depending on the iteration step $k$. 

A \textsc{Julia} code is available at~\url{www.github.com/felixbartel/BSSsubsampling.jl}.

\noindent
\begin{algorithm}[H]
\caption{\texttt{BSS}}\label{algo1:BSS}
  \begin{tabularx}{\textwidth}{lXr}
    \textbf{Input:} & Frame $\bm y^1,\dots,\bm y^M\in\field{C}^m$ with frame bounds $0<A\le B<\infty$; \\
    & Oversampling factor $b>\kappa^2$ with $\kappa$ as in~\eqref{eqdef:kappa}; Stability factor $\Delta\ge 0$. &  \\\hline
    \textbf{Output:} & 
    Nonnegative weights $s_i$ such that $\sqrt{s_1}\bm y^1, \dots \sqrt{s_M}\bm y^M$ is a frame &   \\
          &   with  $|\{i:s_i>0 \}|\le\lceil bm\rceil $ and
          bounds $0<A\le B\gamma(1+\Delta)<\infty$. & \\
          & ($\gamma:=\gamma(b,\kappa)$ is the value from~\eqref{eq:gamma}.) &
  \end{tabularx}
\end{algorithm}
\vspace*{-14pt}
\begin{algorithmic}[1]
\STATE{
    Put $n:=\lceil bm \rceil$ and $\kappa:=\kappa(A,B)$ as in \eqref{eqdef:kappa}.
    Further, set $\bA^{(0)}:= 0$,
    \begin{align*}
    l^{(0)}:= - m\frac{\sqrt{b}\kappa}{1+\Delta} \,,\quad
    u^{(0)}:= m \frac{b+\sqrt{b}}{\sqrt{b}-1} \frac{B}{A}  \,,\quad \delta_L^{(0)}:= \frac{1}{1+\Delta}
     \,,\quad \delta_U^{(0)}:= \frac{\sqrt{b}+1}{\sqrt{b}-1}\frac{B}{A} \,.
    \end{align*}
$\triangleright$\; $\bA^{(0)}\in\field{C}^{m\times m}$ is the zero matrix, $l^{(0)}$, $u^{(0)}$ associated lower and upper spectral \\\quad barriers. The initial barrier shifts are given by $\delta_L^{(0)}$, $\delta_U^{(0)}$.}

\FOR{$k=1$ \TO $n$}
\STATE{
    Compute the eigenvalues $\lambda^{(k-1)}_1,\ldots, \lambda^{(k-1)}_m$ of $\bA^{(k-1)}$.
}
\STATE{
    Compute the so-called lower and upper potentials (see~Definition~\ref{def:uplowpotentials}) \\[1ex]
    \quad$\epsilon^{(k-1)}_L :=\Phi_{l^{(k-1)}}(\bA^{(k-1)}) = \sum_{j=1}^{m} \big(\lambda^{(k-1)}_j - l^{(k-1)}\big)^{-1}$, \\[1ex]
    \quad$\epsilon^{(k-1)}_U :=\Phi^{u^{(k-1)}}(\bA^{(k-1)}) = \sum_{j=1}^{m} \big(u^{(k-1)} - \lambda^{(k-1)}_j\big)^{-1}$. \\[1ex]
}
\STATE{
    Put\,
    $\delta_L^{(k-1)}:= \Big( \frac{1}{\delta_L^{(0)}} - \kappa \epsilon^{(0)}_L + \kappa  \epsilon^{(k-1)}_L \Big)^{-1} $
    and\quad
    $\delta_U^{(k-1)}:= \Big( \frac{1}{\delta_U^{(0)}} + \epsilon^{(0)}_U - \epsilon^{(k-1)}_U \Big)^{-1} $.
}\label{li:varshifts}
\STATE{
    Increment $l^{(k-1)}$  and $u^{(k-1)}$:
    $l^{(k)} := l^{(k-1)} + \delta^{(k-1)}_L$, $u^{(k)} := u^{(k-1)} + \delta^{(k-1)}_U$.
}\label{li:increments}
\STATE{
    Compute the factors\\[1ex]
    \quad$ f^{(k-1)}_L := \Phi_{l^{(k)}}(\bA^{(k-1)}) = \sum_{j=1}^{m} \big(\lambda^{(k-1)}_j - l^{(k)}\big)^{-1} $, \\[1ex]
    \quad$ f^{(k-1)}_U := \Phi^{u^{(k)}}(\bA^{(k-1)}) = \sum_{j=1}^{m} \big(u^{(k)} - \lambda^{(k-1)}_j\big)^{-1} $. \\[1ex]
}
\FOR{$j=1$ \TO $M$}
\STATE{
    Compute
    \begin{flalign*}
    \quad L^{(k-1)}(\by^{j})&:= \frac{(\by^{j})^\ast(\bA^{(k-1)}- l^{(k)}\bI)^{-2}\by^{j}}{f^{(k-1)}_L - \epsilon^{(k-1)}_L} - (\by^{j})^\ast(\bA^{(k-1)}- l^{(k)}\bI)^{-1}\by^{j} \,, && \\
    U^{(k-1)}(\by^{j})&:= \frac{(\by^{j})^\ast(\bA^{(k-1)} - u^{(k)}\bI)^{-2}\by^{j}}{\epsilon^{(k-1)}_U-f^{(k-1)}_U} - (\by^{j})^\ast(\bA^{(k-1)} - u^{(k)}\bI)^{-1}\by^{j}  \,.
    \end{flalign*}
}\label{li:LU}
\IF 
{ $L^{(k-1)}(\by^{j}) - U^{(k-1)}(\by^{j}) \ge \frac{\Delta}{2M} \Big( 1 - \frac{1}{\sqrt{b}} \Big)\label{selectCond}$
}\label{testlabel}
\STATE{ denote this index by $i^{(k)}$.}
\STATE{\textbf{break}}
\ENDIF
\ENDFOR
\STATE{
    Compute\\
    \quad$t^{(k)} \coloneqq 2\big(L^{(k-1)}(\by^{i^{(k)}})+U^{(k-1)}(\by^{i^{(k)}})\big)^{-1}$,\\
    \quad$\tilde{s}_{i^{(k)}} \coloneqq \tilde{s}_{i^{(k)}} + t^{(k)}$,\\
    \quad$\bA^{(k)} \coloneqq \bA^{(k-1)} + t^{(k)} \by^{i^{(k)}}(\by^{i^{(k)}})^\ast$.
}\label{li:update}
\ENDFOR
\RETURN{
    rescaled weights $\displaystyle{s_{i}:= \frac{1}{2} \Big( \frac{A}{l^{(n)}} + \frac{B \gamma(1+\Delta)}{u^{(n)}} \Big) \tilde{s}_{i}}$ for $i=1,\dots,M$.
}\label{li:return}
\end{algorithmic}
\vspace*{-4pt}
\noindent
\rule{\textwidth}{0.8pt}\\[-12pt]
\rule{\textwidth}{0.8pt}
\vspace{1ex}

\noindent
As explained in Subsection~\ref{ssec:SketchProof}, we want to produce a sequence~\eqref{iteration_matrices} of matrices 
$\bA^{(k)}$ which fulfill the spectral condition~\eqref{reg} for the respective spectral barriers $l^{(k)}$ and $u^{(k)}$.
Algorithm~\ref{algo1:BSS} accomplishes this. For the details we refer to Subsection~\ref{ssec:Thm2.1}.
A crucial step is the index selection in line~\ref{testlabel}. The condition there guarantees that the chosen $i^{(k)}$ and the subsequently computed $t^{(k)}$ lead to a new updated matrix $\bA^{(k)}$ (in line~\ref{li:update}) which fulfills~\eqref{reg} as the matrices $\bA^{(0)},\ldots,\bA^{(k-1)}$ did before. According to Corollary~\ref{cor:bar_shift}, proved below, essential for this
is $L^{(k-1)}(\by^{i^{(k)}}) \ge U^{(k-1)}(\by^{i^{(k)}})$ and $t^{(k)}\in[(L^{(k-1)}(\by^{i^{(k)}}))^{-1},  (U^{(k-1)}(\by^{i^{(k)}}))^{-1}]$.
To avoid numerical issues in the selection, which might occur due to calculation inaccuracies,
the stability parameter $\Delta\ge0$ comes into play.
It ensures that $L^{(k-1)}(\by^{i^{(k)}}) \ge U^{(k-1)}(\by^{i^{(k)}})$ can be verified, via the condition in line~\ref{testlabel}, in a numerically stable manner.

\begin{remark}
Algorithm~\ref{algo1:BSS} also works for fixed barrier shifts. We can
skip the update in line~\ref{li:varshifts} and always use $\delta^{(0)}_L$ and $\delta^{(0)}_U$
in the subsequent incrementation step in line~\ref{li:increments}. The advantage of variable shifts is a 
sharper containment of the spectrum (see illustration in Fig.~\ref{fig:barr_shifts}). 
\end{remark}

By including a preceding orthogonalization procedure, 
it is possible to allow arbitrarily small oversampling factors $b>1$.
Further the guarantees on the bounds improve. The modified algorithm is called \texttt{BSS$^\perp$.} It is based on the following simple observation. 

\begin{lemma}\label{discreteconstruction1} For every matrix $\bm Y\in\mathds C^{M\times m}$ with $M\ge m$ 
there is a matrix $\bm{\tilde Y}\in\mathds C^{M\times m}$ 
such that
  \begin{align*}
    \range(\bm{\tilde Y}) \supset \range(\bm Y)
    ,\quad
\bm{\tilde Y}\herm    \bm{\tilde Y} = \bI
    ,\quad\text{and}\quad
    \|\bm{\tilde Y}\|_F^2 = m \,,
  \end{align*}
  where $\range(\bm{\tilde Y})$ and $\range(\bm Y)$ denote the range in $\field{C}^M$ of the respective operators.
\end{lemma} \begin{proof} The matrix $\bm{\tilde Y}$ is constructed by applying the Gram-Schmidt algorithm to the columns of $\bm Y$. If we end up with less than $m$ vectors, which happens if $\rank(\bm{Y}) <m$, we orthogonally extend them, which is possible since $M\geq m$.  
\end{proof}

\noindent
\begin{algorithm}[H]
\caption{\texttt{BSS$^\perp$}}\label{alg2:BSS}
  \begin{tabularx}{\textwidth}{lXr}
     \textbf{Input:} & Frame $\bm y^1,\dots,\bm y^M\in\field{C}^m$ with frame bounds $0<A\le B<\infty$; \\
    & Oversampling factor $b>1$; Stability factor $\Delta\ge 0$. &  \\\hline
    \textbf{Output:} & 
    Nonnegative weights $s_i$ such that $\sqrt{s_1}\bm y^1, \dots \sqrt{s_M}\bm y^M$ is a frame &   \\
          &   with  $|\{i:s_i>0 \}|\le\lceil bm\rceil $ and
          bounds $0<A\le B\gamma(1+\Delta)<\infty$. & \\
          & ($\gamma$ is the value from~\eqref{eq:gamma} for $\kappa=1$.) &
  \end{tabularx}
\end{algorithm}
\vspace*{-14pt}
\begin{algorithmic}[1]
\STATE{
    Let $\bm Y\in\mathds C^{M\times m}$ be the matrix with rows $\bm y^1,\dots,\bm y^M$
    and construct $\bm{\tilde Y}\in\mathds C^{M\times m}$ as in Lemma~\ref{discreteconstruction1} via Gram-Schmidt orthogonalization of the columns of $\bm Y$.
}
\RETURN{
    weights $s_1,\dots,s_M$, calculated by applying \texttt{BSS} (Algorithm~\ref{algo1:BSS}) to the rows of $\bm{\tilde Y}$. 
}
\end{algorithmic}
\vspace*{-10pt}
\noindent
\rule{\textwidth}{0.8pt}\\[-12pt]
\rule{\textwidth}{0.8pt}
\vspace{1ex}

\noindent
Note that the rows $\bm{\tilde y}^1, \dots, \bm{\tilde y}^M$ of $\bm{\tilde Y}$, constructed in line~1
of Algorithm~\ref{alg2:BSS}, form a tight frame. 
Hence, in lines~2 Algorithm~\ref{algo1:BSS} can be applied for arbitrarily small $b>1$. 
In fact, the frame property of the initial system $(\bm y^i)_{i=1}^M$ is not needed for this. 
It is possible to run \texttt{BSS$^\perp$} for any input vector sequence $(\bm y^i)_{i=1}^M$ in $\mathds C^{m}$, satisfying $M\ge m$. The returned weights $s_i$ always fulfill $|\{i:s_i\neq 0 \}|\le\lceil bm\rceil $ and it 
always holds
    \begin{align}\label{auxweightedestimate1}
      \sum_{i=1}^M \left|\left\langle \bm a, \bm y^i \right\rangle\right|^2
      \le 
      \sum_{i=1}^{M}
      s_i \left|\left\langle \bm a, \bm y^i \right\rangle\right|^2
      \le \frac{(\sqrt b + 1)^2}{(\sqrt b - 1)^2}(1+\Delta) \sum_{i=1}^M \left|\left\langle \bm a, \bm y^i \right\rangle\right|^2 
    \end{align}
for all $ a\in \field{C}^{m}$. Assuming the input sequence was a frame, we then further deduce
\begin{align*}
      A \| \bm a \|_2^2 
      \le 
      \sum_{i=1}^{M}
      s_i \left|\left\langle \bm a, \bm y^i \right\rangle\right|^2
      \le \frac{(\sqrt b + 1)^2}{(\sqrt b - 1)^2}(1+\Delta) B \| \bm a \|_2^2 \quad\text{for all}\quad
      \bm a\in \field{C}^{m}\,.
    \end{align*}
To verify~\eqref{auxweightedestimate1}, let us first reformulate this inequality as
\begin{align}\label{auxweightedestimate}
     \|\bm Y\bm a\|_2^2
      \le 
      \|\bm S^{\frac{1}{2}}(\bm Y\bm a)|_{J}\|_2^2
      \le \frac{(\sqrt b+1)^2}{(\sqrt b-1)^2}(1+\Delta) \|\bm Y\bm a\|_2^2 \,,
\end{align}
where $J \coloneqq \{i : s_i \neq 0\}$, $(\bm Y\bm a)|_{J}$ stands for the restriced vector $([\bm Y\bm a]_i)_{i\in J} \in \mathds C^{|J|}$, and $\bm S := \diag(s_i)_{i\in J}\in\mathds C^{|J|\times|J|}$. 
By Theorem~\ref{BSS}, applying \texttt{BSS} (Algorithm~\ref{algo1:BSS}) to $(\bm{\tilde{y}}^i)_{i=1}^M$ yields
$s_i$ such that $|J|=|\{i:s_i\neq0 \}|\le\lceil bm\rceil $ and
    \begin{align*}
        \|\bm a\|_2^2
        \le \sum_{i=1}^{M} s_i |\langle\bm a,\bm{\tilde y}^i\rangle|^2
        \le \frac{(\sqrt{b}+1)^2}{(\sqrt{b}-1)^2}(1+\Delta) \|\bm a\|_2^2 \quad\text{for all}\quad
      \bm a\in \field{C}^{m}\,.
    \end{align*}
    Therefore, by the orthogonality of $\bm{\tilde Y}$, for all $\bm a\in \field{C}^{m}$
    \begin{align*}
      \|\bm{\tilde Y}\bm a\|_2^2
      \le 
      \|\bm S^{\frac{1}{2}} (\bm{\tilde Y}\bm a)|_{J}\|_2^2
      \le \frac{(\sqrt b+1)^2}{(\sqrt b-1)^2}(1+\Delta) \|\bm{\tilde Y}\bm a\|_2^2 \,.
    \end{align*}
    Using $\range(\bm{\tilde Y}) \supset \range(\bm Y)$, as guaranteed by Lemma~\ref{discreteconstruction1}, we may finally replace $\bm{\tilde Y}$ in this last inequality with the original $\bm Y$, leading to~\eqref{auxweightedestimate}.

\paragraph{Runtime analysis of Algorithm~\ref{algo1:BSS}} The singular value decomposition (SVD) of $\bA^{k-1}$ in the $k$th iteration step 
has a complexity of $\mathcal{O}(m^3)$. Having the SVD decomposition at hand, matrix-vector products with
$(\bA^{(k-1)}- l^{(k)}\bI)^{-1}$, $(\bA^{(k-1)}- l^{(k)}\bI)^{-2}$, $(\bA^{(k-1)}- u^{(k)}\bI)^{-1}$, and $(\bA^{(k-1)}- u^{(k)}\bI)^{-2}$ are computable in $\mathcal O(m^2)$. 
To eventually decide, which index is selected in line~\ref{selectCond}, 
$L^{(k-1)}(\by^{i})$ and $U^{(k-1)}(\by^{i})$ in the worst case need to be computed for all $i\in [M]$, .
This thus may require $\mathcal{O}(Mm^2)$ multiplication steps. All in all, taking into account $M\ge m$, each iteration can be performed in $\mathcal{O}(Mm^2)$ time.
Since the number of iterations is $\lceil bm\rceil$, the total time of the algorithm is $\mathcal{O}(bMm^3)$.
In our implementation we used a random procedure to traverse the indices $i\in [M]$ and 
noticed that this speeds up the algorithm, see Section~\ref{sec:numerics}, Experiment~3.

Instead of computing the singular value decomposition from scratch every iteration it is possible to update it continuously, cf.\ \cite{BN78, MVV92}, which we have not implemented.

\subsection{Spectral analysis of rank-1 updates}
\label{ssec:PotBar}

In this subsection we analyze from a general perspective, how the spectrum $\sigma(\bA)$ of a Hermitian matrix $\bA\in\field{C}^{m\times m}$ changes
under a rank-1 update of the form
\begin{align}\label{eq:rank1up}
\bA \rightsquigarrow \bA^\prime:=\bA + t\bv\bv^\ast
\end{align}
with $\bv\in\field{C}^m$ and $t\in\field{R}$. This question has already been discussed for the real setting in~\cite[Sec.~3.1~\&~3.2]{BaSpSr09}.
Our analysis here is analogous, however, we go a bit more into detail. In the end, we can derive precise conditions in Lemmas~\ref{lem:lowbar_shift}, \ref{lem:upbar_shift},
and Corollary~\ref{cor:bar_shift}.

With the matrix determinant lemma (Lemma~\ref{APPlem:MatDet} in the Appendix)
the characteristic polynomial $p_{\bA^\prime}(\lambda)=\det(\lambda \bI - \bA^\prime) $ of $\bA^\prime$ in \eqref{eq:rank1up} can be calculated explicitly. For $\lambda\notin\sigma(\bA)$
\begin{align*}
p_{\bA^\prime}(\lambda) = \det(\lambda \bI-\bA)\left( 1 - t\bv^\ast (\lambda \bI-\bA)^{-1}\bv\right)   
= p_{\bA}(\lambda)\bigg( 1 - t \sum_{j=1}^m \frac{|\langle \bv,\bu^j\rangle|^2}{\lambda-\lambda_j} \bigg) \,,
\end{align*}
where $\{\lambda_1, \dots, \lambda_m\}$ are the, not necessarily distinct, eigenvalues of $\bA$ and $\{\bu^j\}_{j=1}^m$ is a corresponding orthonormal basis of eigenvectors.
The eigenvalues of $\bA^\prime$ can be obtained from the associated characteristic equation. In case $t=0$,
we obtain as solutions the eigenvalues of $\bA$, i.e., the roots of $p_{\bA}(\lambda)$. In case $t\neq0$, we obtain the roots of $p_{\bA}(\lambda)$ of at least second order together
with the solutions of the so-called secular equation
\[
\sum_{j=1}^m \frac{|\langle \bv,\bu^j\rangle|^2}{\lambda-\lambda_j} = \frac{1}{t} \,.
\]
A discussion of this equation yields the following insight:
If $t\neq0$ the eigenvalues of $\bA^\prime$ interlace the eigenvalues of $\bA$, shifted to the left when $t<0$ and shifted to the right when $t>0$.
Moreover, the shifts occur continuously in $t$.
In the limit $t\to\infty$, the largest eigenvalue tends to $\infty$ and the corresponding eigenvector to $\bv$.
For $t\to-\infty$, the smallest eigenvalue tends to $-\infty$, while
the corresponding eigenvector again tends to $\bv$.
In case of algebraic multiplicities, always merely one of the respective eigenvalues moves, the rest remain at their old position.

With this, we already have a good qualitative picture of what happens to $\sigma(\bA)$ when applying a rank-1 update~\eqref{eq:rank1up}.
Next, we want to quantify how far the eigenvalues are shifted depending on the size of $t$.
Again we follow \cite{BaSpSr09} and utilize so-called potential functions.

\begin{definition}[{cf.~\cite[Def.~3.2]{BaSpSr09}}]\label{def:uplowpotentials}
Let $\bI$ denote the identity matrix in $\field{C}^{m\times m}$.
For $l,u\in\field{R}$ and a Hermitian matrix $\bA\in\field{C}^{m\times m}$ with eigenvalues $\{\lambda_1,\ldots,\lambda_m\}\subset\field{R}$
the \emph{lower and upper potential functions} $\Phi_l(\bA)$ and $\Phi^u(\bA)$ are given by
\begin{align*}
\Phi_l(\bA)&:= \tr([\bA-l\bI]^{-1}) = \sum_{i=1}^{m} \frac{1}{\lambda_i-l} \,, \\
\Phi^u(\bA)&:= \tr([u\bI-\bA]^{-1}) =  \sum_{i=1}^{m} \frac{1}{u-\lambda_i} \,.
\end{align*}
\end{definition}

When $l$ and $u$ are lower, respectively upper, barriers for $\sigma(\bA)$, i.e., when $l<\lamin(\bA)$, respectively $\lamax(\bA)<u$,
these potential functions serve well as measures for the distance of $\sigma(\bA)$ to the respective barriers.
Note that the so-measured distance counts in the whole spectrum of $\bA$, i.e., the location of all eigenvalues matters.
Further, it holds: the larger the distance, the lower the potentials, and the smaller the distance, the larger the potentials.

From the qualitative discussion above it is clear that
the upper potential $\Phi^u(\bA^\prime)$
becomes smaller when $t$ decreases and larger when $t$ increases.
The lower potential $\Phi_l(\bA^\prime)$ behaves the other way round,
it becomes smaller when $t$ increases and larger when $t$ decreases.
Based on the Sherman-Morrison formula, we now precisely quantify the change of the potentials.

\begin{lemma}\label{lem:pot_shift}
Suppose $\bA\in\field{C}^{m\times m}$ is Hermitian and let $\bv\in\field{C}^{m}$ be a vector.
\begin{enumerate}[(i)]
\item
For $l\notin\sigma(\bA)$ and $t\neq -(\bv^\ast[\bA-l \bI]^{-1}\bv)^{-1}$
\begin{align*}
\Phi_{l}(\bA + t\bv\bv^\ast) = \Phi_{l}(\bA) -  \frac{t \bv^\ast [\bA-l \bI]^{-2}\bv}{1+t\bv^\ast[\bA-l \bI]^{-1}\bv} \,.
\end{align*}
\item
For $u\notin\sigma(\bA)$ and $t\neq (\bv^\ast(u \bI - \bA)^{-1}\bv)^{-1}$
\begin{align*}
\Phi^{u}(\bA + t\bv\bv^\ast) = \Phi^{u}(\bA) +   \frac{t \bv^\ast (u \bI - \bA)^{-2}\bv}{1-t\bv^\ast(u \bI - \bA)^{-1}\bv}   \,.
\end{align*}
\end{enumerate}
\end{lemma}
\begin{proof}
(i):\quad The lower potential has the form
\begin{align*}
\Phi_{l}(\bA + t\bv\bv^\ast) &= \tr\Big( \big[\bA + t\bv\bv^\ast - l \bI \big]^{-1} \Big) \,.
\end{align*}
Using the Sherman-Morrison formula and properties of the trace, we obtain
\begin{align*}
\Phi_{l}(\bA + t\bv\bv^\ast)
&= \tr\bigg( [\bA-l \bI]^{-1} - \frac{t[\bA-l \bI]^{-1}\bv\bv^\ast[\bA-l \bI]^{-1}}{1+t\bv^\ast[\bA-l \bI]^{-1}\bv}  \bigg) \\
&= \tr\big( [\bA-l \bI]^{-1} \big) -   \frac{t \tr\big([\bA-l \bI]^{-1}\bv\bv^\ast[\bA-l \bI]^{-1}\big)}{1+t\bv^\ast[\bA-l \bI]^{-1}\bv}  \\
&= \Phi_{l}(\bA) -   \frac{t \bv^\ast [\bA-l \bI]^{-2}\bv}{1+t\bv^\ast[\bA-l \bI]^{-1}\bv}  \,.
\end{align*}

\noindent
(ii):\quad Analogous to (i).
\end{proof}

Next, we turn our attention to barrier shifts, i.e., modifications
of the barriers of the form $l\rightsquigarrow l^\prime:=l+\delta_L$ and $u\rightsquigarrow u^\prime:=u+\delta_U$.
Our concrete goal is to specify shifts
for which the spectrum of the updated matrix $\bA^\prime$ in~\eqref{eq:rank1up}
is enclosed in $(l^\prime,u^\prime)$ and the distance of $\sigma(\bA^\prime)$ to the barriers has not decreased.
In other words, we want to find $l^\prime,u^\prime\in\field{R}$ such that
\begin{align}\label{eq:prec_cond}
\sigma(\bA^\prime) \subset (l^\prime,u^\prime)\,,\quad
\Phi_{l^\prime}(\bA^\prime) \le \Phi_{l}(\bA)\,,\quad\text{and}\quad
\Phi^{u^\prime}(\bA^\prime) \le \Phi^{u}(\bA) \,.
\end{align}

In Lemma~\ref{lem:lowbar_shift} below, which corresponds to \cite[Lem.~3.3]{BaSpSr09}, we first handle the lower potential.

\begin{lemma}[{cf.~\cite[Lem.~3.3]{BaSpSr09}}]
\label{lem:lowbar_shift}
Let $\bA\in\field{C}^{m\times m}$ be Hermitian and assume $\Delta_L:=\lamin(\bA)-l>0$ for $l\in\field{R}$, $\bv\in\field{C}^{m}$ be any vector.
Further let $\delta_L\in(-\infty,\Delta_L)$ and additionally assume $L_{\bA}(\bv;l,\delta_L)>0$ in case $\delta_L>0$,
where $L_{\bA}(\bv;l,\delta_L):=\infty$ for $\delta_L=0$ and otherwise
\begin{align}\label{LA}
L_{\bA}(\bv;l,\delta_L) := \frac{\bv^\ast(\bA-(l+\delta_L)\bI)^{-2}\bv}{\Phi_{l+\delta_L}(\bA) - \Phi_{l}(\bA)} - \bv^\ast(\bA-(l+\delta_L)\bI)^{-1}\bv  \,.
\end{align}
Then precisely for $t\ge L_{\bA}(\bv;l,\delta_L)^{-1}$ (with $\infty^{-1}=0$)
\begin{align}\label{stateLemlow}
\Phi_{l+\delta_L}(\bA+t\bv\bv^\ast) \le \Phi_{l}(\bA) \quad\text{and}\quad \lamin(\bA+t\bv\bv^\ast) > l + \delta_L \,.
\end{align}
For $0<\delta_L<\Delta_L$ with $L_{\bA}(\bv;l,\delta_L)\le0$ there are no $t\in\field{R}$ satisfying \eqref{stateLemlow}.
\end{lemma}
\begin{proof}
Let $l^\prime=l+\delta_L$. Due to $\delta_L<\Delta_L$, we have $\lamin(\bA)>l^\prime$ and in particular $l^\prime\notin\sigma(\bA)$. Hence, from Lemma~\ref{lem:pot_shift}, we  directly derive
\begin{align*}
\Phi_{l^\prime}(\bA + t\bv\bv^\ast) - \Phi_{l}(\bA) = (\Phi_{l^\prime}(\bA) - \Phi_{l}(\bA)) -   \frac{t \bv^\ast [\bA-l^\prime \bI]^{-2}\bv}{1+t\bv^\ast[\bA-l^\prime \bI]^{-1}\bv} \,.
\end{align*}
Note further that always $\bv^\ast(\bA-l^\prime \bI)^{-2}\bv>0$ and $\bv^\ast(\bA-l^\prime \bI)^{-1}\bv>0$.

The case $\delta_L=0$ is clear. Here $\Phi_{l}(\bA+t\bv\bv^\ast) \le \Phi_{l}(\bA)$ and $\lamin(\bA+t\bv\bv^\ast) > l$ precisely for $t\ge 0= L_{\bA}(\bv;l,0)^{-1}$.

Next we turn to $\delta_L>0$. Here we have the additional assumption $L_{\bA}:=L_{\bA}(\bv;l,\delta_L)>0$. We hence have $t\ge L_{\bA}^{-1}>0$ which implies
$0<1/t\le L_{\bA}$. We conclude
\begin{align}\label{auxauxaux}
 \frac{t \bv^\ast [\bA-l^\prime \bI]^{-2}\bv}{1+t\bv^\ast[\bA-l^\prime \bI]^{-1}\bv}  \ge  \frac{\bv^\ast [\bA-l^\prime \bI]^{-2}\bv}{L_{\bA}+\bv^\ast[\bA-l^\prime \bI]^{-1}\bv}
   = \Phi_{l^\prime}(\bA) - \Phi_{l}(\bA) \,.
\end{align}
In addition, $\lamin(\bA+t\bv\bv^\ast)>\lamin(\bA)>l^\prime$ due to $t>0$.

Finally, if $\delta_L<0$ then $\Phi_{l^\prime}(\bA) < \Phi_{l}(\bA)$ and thus $L_{\bA}<-\bv^\ast(\bA-l^\prime \bI)^{-1}\bv<0$.
Hence, in the range $0>t\ge L_{\bA}^{-1}$ we have $1/t\le L_{\bA}<-\bv^\ast(\bA-l^\prime \bI)^{-1}\bv$ and \eqref{auxauxaux} is valid.
For $t\ge0$ \eqref{auxauxaux} is also valid since then we have a negative right-hand side and a positive left-hand side there.
Also, clearly $\lamin(\bA+t\bv\bv^\ast)\ge\lamin(\bA)>l^\prime$ in case $t\ge0$.
If $t<0$ we argue by contradiction. Assuming $\lamin(\bA+t\bv\bv^\ast)\le l^\prime$, by continuity since $\lamin(\bA)>l^\prime$,
there would be $t^\prime$ with $t\le t^\prime<0$ and $\lamin(\bA+t^\prime\bv\bv^\ast) = l^\prime$.
But $\Phi_{l^\prime}(\bA+t^\prime\bv\bv^\ast)\le\Phi_{l^\prime}(\bA+t\bv\bv^\ast)\le\Phi_{l}(\bA)<\infty$, which is a contradiction.

For the last statement, assume that $0<\delta_L<\Delta_L$ and $L_{\bA}\le 0$.
Then necessarily $t>0$ for \eqref{stateLemlow}. Further $\Phi_{l^\prime}(\bA) > \Phi_{l}(\bA)$
and as a consequence $L_{\bA}>-\bv^\ast[\bA-l^\prime \bI]^{-1}\bv$. However $1/t > 0 \ge L_{\bA}$ now
contradicts~\eqref{auxauxaux} since $\frac{1}{t}+\bv^\ast[\bA-l^\prime \bI]^{-1}\bv> L_{\bA} + \bv^\ast[\bA-l^\prime \bI]^{-1}\bv>0$.
\end{proof}

In the next lemma, which corresponds to \cite[Lem.~3.4]{BaSpSr09}, we handle the upper potential, where the situation is dual to the one before.

\begin{lemma}[{cf.~\cite[Lem.~3.4]{BaSpSr09}}]
\label{lem:upbar_shift}
Let $\bA\in\field{C}^{m\times m}$ be Hermitian and assume $\Delta_U:=u-\lamax(\bA)>0$ for $u\in\field{R}$, $\bv\in\field{C}^{m}$ \mbox{be any vector}.
Further let $\delta_U\in(-\Delta_U,\infty)$ and additionally assume $U_{\bA}(\bv;u,\delta_U)<0$ in case $\delta_U<0$,
where $U_{\bA}(\bv;u,\delta_U):=\infty$ for $\delta_U=0$ and otherwise
\begin{align}\label{UA}
U_{\bA}(\bv;u,\delta_U) := \frac{\bv^\ast[(u+\delta_U)\bI-\bA]^{-2}\bv}{\Phi^{u}(\bA) - \Phi^{u+\delta_U}(\bA)} + \bv^\ast[(u+\delta_U)\bI - \bA]^{-1}\bv \,.
\end{align}
Then precisely for $t\le U_{\bA}(\bv;u,\delta_U)^{-1}$ (with $\infty^{-1}=0$)
\begin{align}\label{stateLemup}
\Phi^{u+\delta_U}(\bA+t\bv\bv^\ast) \le \Phi^{u}(\bA) \quad\text{and}\quad \lamax(\bA+t\bv\bv^\ast) < u + \delta_U \,.
\end{align}
For $-\Delta_U<\delta_U<0$ with $U_{\bA}(\bv;u,\delta_U)\ge0$ there are no $t\in\field{R}$ satisfying \eqref{stateLemup}.
\end{lemma}
\begin{proof}
Using duality, the proof carries over from Lemma~\ref{lem:lowbar_shift} in a strictly analogous manner.
\end{proof}

Let us now assume that $\bA\in\field{C}^{m\times m}$ is Hermitian with $\sigma(\bA)\subset(l,u)$.
Then, for any fixed vector $\bv\in\field{C}^m$ and any given $\delta_U$ fulfilling the assumptions of Lemma~\ref{lem:upbar_shift},
condition~\eqref{stateLemup} can be ensured by choosing $t$ small enough.
Similarly, condition~\eqref{stateLemlow} in Lemma~\ref{lem:lowbar_shift} can always be ensured for large $t$ if $\delta_L$ satisfies the assumptions of this lemma.
It is unclear, however, if there exist $t\in\field{R}$ which fulfill both conditions simultaneously, i.e., values for $t$ which
ensure~\eqref{eq:prec_cond}.

Combining Lemma~\ref{lem:lowbar_shift} and Lemma~\ref{lem:upbar_shift} the subsequent corollary arises.
It provides a precise condition when such $t$ exist and determines their precise range.

\begin{corollary}
\label{cor:bar_shift}
Let $\bv\in\field{C}^m$ be a vector and $\bA\in\field{C}^{m\times m}$ a Hermitian matrix with lower and upper barriers $l,u\in\field{R}$.
Let further $\delta_L,\delta_U\in\field{R}$ with $\delta_L<\Delta_L$ and $\delta_U>-\Delta_U$, where $\Delta_L:=\lamin(\bA)-l$ and $\Delta_U:=u-\lamax(\bA)$.
Then the following conditions are equivalent.
\begin{enumerate}[(i)]
\item With $L_{\bA}:=L_{\bA}(\bv;u,\delta_U)$ and $U_{\bA}:=U_{\bA}(\bv;u,\delta_U)$ as in \eqref{LA} and \eqref{UA}, respectively,
\begin{align}\label{corcond}
\signum \delta_L = \signum L_{\bA} \,,\quad \signum \delta_U = \signum U_{\bA} \,,\quad L_{\bA}^{-1} \le U_{\bA}^{-1} \,.
\end{align}
\item There exist $t\in\field{R}$ such that \eqref{eq:prec_cond} is fulfilled for $\bA^\prime:=\bA+t\bv\bv^\ast$, $l^\prime:=l+\delta_L$, and $u^\prime:=u+\delta_U$.
\end{enumerate}
In case that (i) and (ii) hold true, \eqref{eq:prec_cond} is fulfilled precisely for $t\in[L_{\bA}^{-1},U_{\bA}^{-1}]$.
\end{corollary}
\begin{proof}
The corollary is a consequence of Lemma~\ref{lem:lowbar_shift} and Lemma~\ref{lem:upbar_shift}.
To see this, note that $\delta_U>0$ implies $U_{\bA}>0$ and that, similarly, $\delta_L<0$ implies $L_{\bA}<0$.
\end{proof}

\subsection{The well-determined update step}
\label{ssec:Lem3.5}

Let us now focus back on the \texttt{BSS} algorithm.
The main result of this subsection, Lemma~\ref{lem3.5}, is the key tool to keep control of the spectra of the matrices $\bA^{(k)}$ in~\eqref{iteration_matrices}.
It provides a condition on the barrier shifts $\delta_L$ and $\delta_U$
to guarantee that, when the spectral window $(l^{(k)},u^{(k)})$ of $\bA^{(k)}$ is shifted to $(l^{(k+1)},u^{(k+1)})$,
there is at least one frame vector $\by^i$
such that condition \eqref{corcond} in Corollary~\ref{cor:bar_shift} is fulfilled. Hence, there exists
$t>0$ such that~\eqref{eq:prec_cond} is fulfilled for the update $\bA^{(k+1)}=\bA^{(k)}+ t\by^{i}(\by^{i})^\ast$, in particular $\sigma(\bA^{(k+1)})\subset(l^{(k+1)},u^{(k+1)})$.

For the proof of Lemma~\ref{lem3.5} we need two auxiliary results.
The first result is Lemma~\ref{lem:auxBSS} from the Appendix.
The second auxiliary result is~\cite[Claim~3.6]{BaSpSr09},
which we recall for convenience here.

\begin{lemma}\label{lem:aux2BSS}
Let $\delta_L,\,\epsilon_L>0$, $l\in\field{R}$, $\{\lambda_1,\ldots,\lambda_m\}\subset\field{R}$.
If $\lambda_i>l$ for $i\in [m]$, $0\le\sum_{i}(\lambda_i-l)^{-1}\le\epsilon_L$, and $1/\delta_L-\epsilon_L\ge0$,
then
\begin{align}\label{complicatedest}
\frac{\sum_{i}(\lambda_i-l-\delta_L)^{-2}}{\sum_{i}(\lambda_i-l-\delta_L)^{-1}-\sum_{i}(\lambda_i-l)^{-1}}  - \sum_{i} \frac{1}{\lambda_i-l-\delta_L}
\ge \frac{1}{\delta_L} - \sum_i \frac{1}{\lambda_i-l} \,.
\end{align}
\end{lemma}

With this, we are ready to prove Lemma~\ref{lem3.5},
which plays the same role in the proof of Theorem~\ref{BSS} as~\cite[Lem.~3.5]{BaSpSr09} in the proof of~\cite[Thm.~3.1]{BaSpSr09}.

\begin{lemma}\label{lem3.5}
Let $(\bm y^i)_{i=1}^M \subset \field{C}^{m}$ be a frame with frame bounds $0<A\le B<\infty$.
Let further $\bA\in\field{C}^{m\times m}$ be Hermitian with $\sigma(\bA)\subset(l,u)$ for $l,u\in\field{R}$ and with corresponding potentials $\Phi_l(\bA)$, $\Phi^u(\bA)$.
If the quantities
$\delta_L,\,\delta_U,\,\epsilon_L,\,\epsilon_U>0$
satisfy the condition
\begin{align}\label{cond:lem3.5}
0< \frac{B}{A} \Big( \frac{1}{\delta_U} + \epsilon_U \Big) \le \frac{1}{\delta_L} - \kappa\epsilon_L \,,\quad \epsilon_L\ge\Phi_l(\bA) \,,\quad \epsilon_U\ge\Phi^u(\bA) \,,
\end{align}
where $\kappa=\kappa(A,B)$ is as in \eqref{eqdef:kappa},
then there exists an index $i\in [M]$ such that
condition~\eqref{corcond} in Corollary~\ref{cor:bar_shift} is fulfilled for $\by^i$,
and the indices with this property are precisely those where $L_{\bA}(\by^i)\ge U_{\bA}(\by^i)$
for $L_{\bA}(\by^i)=L_{\bA}(\by^i;l,\delta_L)$ and $U_{\bA}(\by^i)=U_{\bA}(\by^i;u,\delta_U)$ as in \eqref{LA} and \eqref{UA}.
The corresponding rank-1 updates $\bA^\prime:=\bA+t\by^i(\by^i)^\ast$ fulfill~\eqref{eq:prec_cond} with $l^\prime=l+\delta_L$ and $u^\prime=u+\delta_U$
for each $t>0$ with
\[
L_{\bA}(\by^i) \ge 1/t \ge U_{\bA}(\by^i) \,.
\]

\end{lemma}
\begin{proof}
First note that, according to our assumptions, we have $\epsilon_L\ge \Phi_l(\bA)$ and $0< \delta_L^{-1} - \kappa\epsilon_L$, which
implies
\begin{align}\label{est:deltaL}
\delta_L < (\kappa\epsilon_L)^{-1}  \le \frac{1}{\kappa} \Phi_l(\bA)^{-1}\le\frac{1}{\kappa}\Delta_L \,,
\end{align}
where $\Delta_L=\lamin(\bA)- l$ and the last estimate is due to
\[
\Phi_l(\bA) = \sum_{i=1}^{m} \frac{1}{\lambda_i-l} \ge \frac{1}{\lamin(\bA)- l} = \Delta_L^{-1} \,.
\]
In particular, $\delta_L<\Delta_L$ since $\kappa\ge1$. Further $\delta_U>0>-\Delta_U$ for $\Delta_U=u-\lamax(\bA)$,
wherefore the quantities $L_{\bA}(\by^i)$ and $U_{\bA}(\by^i)$ are well-defined (see \eqref{LA} and \eqref{UA}).

Now, analogous to the proof of~\cite[Lem.~3.5]{BaSpSr09}, we aim to show
\begin{align}\label{aux:estimate}
\sum_i L_{\bA}(\by^i) \ge \sum_i U_{\bA}(\by^i) \,.
\end{align}
\noindent
On the one hand, using Lemma~\ref{lem:auxBSS}, we have
\begin{align*}
\sum_i U_{\bA}(\by^i) &= \frac{\sum_i (\by^i)^\ast((u+\delta_U)\bI-\bA)^{-2}\by^i}{\Phi^u(\bA)-\Phi^{u+\delta_U}(\bA)} + \sum_i (\by^i)^\ast ((u+\delta_U)\bI-\bA)^{-1} \by^i \\
&\le B \bigg[ \frac{\tr((u+\delta_U)\bI-\bA)^{-2}}{\Phi^u(\bA)-\Phi^{u+\delta_U}(\bA)} + \tr((u+\delta_U)\bI-\bA)^{-1} \bigg] \\
&= B \bigg[ \frac{\sum_{i}(u+\delta_U-\lambda_i)^{-2}}{\sum_{i}(u-\lambda_i)^{-1}-\sum_{i}(u+\delta_U-\lambda_i)^{-1}} + \Phi^{u+\delta_U}(\bA) \bigg]  \,.
\end{align*}
The denominator of the first term in the brackets can be estimated as follows,
\begin{align*}
\sum_{i}\frac{1}{u-\lambda_i}-\sum_{i} \frac{1}{u+\delta_U-\lambda_i}
= \sum_{i}\frac{\delta_U}{(u-\lambda_i)(u+\delta_U-\lambda_i)}
> \sum_{i} \frac{\delta_U}{(u+\delta_U-\lambda_i)^{2}} \,.
\end{align*}
Taking into account $\epsilon_U\ge\Phi^{u}(\bA)>\Phi^{u+\delta_U}(\bA)$, we obtain altogether
\begin{align}
\begin{aligned}\label{eq:result1}
\sum_i U_{\bA}(\by^i) < B \Big( \frac{1}{\delta_U} + \epsilon_U \Big) \,.
\end{aligned}
\end{align}

\noindent
On the other hand, again using Lemma~\ref{lem:auxBSS}, we have
\begin{align*}
\sum_i L_{\bA}&(\by^i) = \frac{\sum_i (\by^i)^\ast(\bA-(l+\delta_L)\bI)^{-2}\by^i}{\Phi_{l+\delta_L}(\bA)-\Phi_{l}(\bA)} - \sum_i (\by^i)^\ast (\bA-(l+\delta_L)\bI)^{-1} \by^i \\
&\ge A \bigg[ \frac{\tr(\bA-(l+\delta_L)\bI)^{-2}}{\Phi_{l+\delta_L}(\bA)-\Phi_l(\bA)} \bigg]
- B \bigg[ \tr(\bA-(l+\delta_L)\bI)^{-1} \bigg] \\
&= A \bigg[ \frac{\sum_{i}(\lambda_i-l-\delta_L)^{-2}}{\sum_{i}(\lambda_i-l-\delta_L)^{-1}-\sum_{i}(\lambda_i-l)^{-1}}  - \sum_{i} \frac{1}{\lambda_i-l-\delta_L}  \bigg]
- (B - A)\Phi_{l+\delta_L}(\bA) \,,
\end{align*}
and the term in brackets can further be estimated by Lemma~\ref{lem:aux2BSS}.
The assumptions of this lemma are satisfied, in particular $\delta_L^{-1}-\epsilon_L\ge\delta_L^{-1}-\kappa\epsilon_L>0$ due to $\kappa\ge1$.
Using~\eqref{complicatedest} leads to
\begin{align}
\begin{aligned}\label{eq:intermediate}
\sum_i L_{\bA}(\by^i)
\ge A\Big(\frac{1}{\delta_L} - \Phi_l(\bA)\Big) - (B - A) \Phi_{l+\delta_L}(\bA) \,.
\end{aligned}
\end{align}

We distinguish two cases, $A=B$ and $A<B$, to derive
\begin{align}\label{eq:result2}
\sum_i L_{\bA}(\by^i)\ge A\Big(\frac{1}{\delta_L} - \kappa\epsilon_L\Big) \,.
\end{align}

If $A=B$ then $\kappa=1$ and~\eqref{eq:result2} follows directly from~\eqref{eq:intermediate}, since $\epsilon_L\ge\Phi_l(\bA)$.
If $A<B$ the argument is a bit more involved. We then have $\kappa>1$.
From~\eqref{est:deltaL} we further deduce $\delta_L<\kappa^{-1}(\lambda_i-l)$ for $i= 1,\ldots,m$.
This allows for the estimate
\begin{align*}
\Phi_{l+\delta_L}(\bA) &= \sum_i \frac{1}{\lambda_i - l -\delta_L}
\le \sum_i \frac{1}{\lambda_i - l} \Big(1 - \frac{1}{\kappa}\Big)^{-1} = \Phi_l(\bA) \frac{\kappa}{\kappa-1} \,.
\end{align*}
Plugging this relation into~\eqref{eq:intermediate} then also implies~\eqref{eq:result2}, namely
\begin{align*}
\sum_i L_{\bA}(\by^i)
\ge  \frac{A}{\delta_L} - \Phi_l(\bA)\Big[A + (B - A)\frac{\kappa}{\kappa-1} \Big]
\ge A\Big(\frac{1}{\delta_L} - \kappa\epsilon_L\Big) \,,
\end{align*}
since $\kappa$ as in~\eqref{eqdef:kappa} fulfills
\begin{align*}
A + (B - A) \frac{\kappa}{\kappa - 1}  = A\kappa  \,.
\end{align*}
Hence, we have finally established~\eqref{eq:result1} and~\eqref{eq:result2} and, in view of assumption~\eqref{cond:lem3.5},
these two results yield~\eqref{aux:estimate}.

As a consequence, there exists
at least one $i\in [M]$ so that $L_{\bA}(\by^i)\ge U_{\bA}(\by^i)$. Furthermore, $U_{\bA}(\by^i)>0$
due to $\delta_U>0$, which in turn implies $L_{\bA}(\by^i)>0$. Since also $\delta_L>0$ by assumption,
condition~\eqref{corcond} in Corollary~\ref{cor:bar_shift} is satisfied for $\by^i$.
The rest follows from Corollary~\ref{cor:bar_shift}, which can indeed be applied due to $\delta_L<\Delta_L$ and $\delta_U>-\Delta_U$
as established at the beginning of this proof.
\end{proof}

With Lemma~\ref{lem3.5} we are now prepared to give a proof for Theorem~\ref{BSS}.

\subsection{Discussion of Algorithm~\ref{algo1:BSS} and proof of Theorem~\ref{BSS}}
\label{ssec:Thm2.1}

We use the \texttt{BSS} algorithm (Algorithm~\ref{algo1:BSS}) to construct the matrices $\bA^{(k)}$ in~\eqref{iteration_matrices}.
According to the algorithm, $\bA^{(0)}=0$ as it should be. Further, due to $l^{(0)}<0<u^{(0)}$, these initial values are valid spectral barriers for $\bA^{(0)}$.
The initial choice of barrier shifts fulfill $\delta^{(0)}_L>0$ and $\delta^{(0)}_U>0$. 
For the initial potentials $\epsilon^{(0)}_L$,
$\epsilon^{(0)}_U$ we have
\begin{align*}\epsilon^{(0)}_L=\Phi_{l^{(0)}}(\bA^{(0)})=-\frac{m}{l^{(0)}} 
\quad\text{and}\quad
\epsilon^{(0)}_U:=\Phi^{u^{(0)}}(\bA^{(0)})=\frac{m}{u^{(0)}} \,.
\end{align*}
Hence, also $\epsilon^{(0)}_L>0$, $\epsilon^{(0)}_U>0$,
and the quantities $\delta^{(0)}_L$, $\delta^{(0)}_U$, $\epsilon^{(0)}_L$, $\epsilon^{(0)}_U$ 
satisfy condition~\eqref{cond:lem3.5} of Lemma~\ref{lem3.5} with respect to $\bA^{(0)}$. 
But even more holds true. For the further analysis of the algorithm, it is useful
to note that for all $k\in\{0,\ldots,n\}$
\begin{align}\label{alg1:condition}
\frac{1}{\delta_L^{(k)}} - \kappa\epsilon^{(k)}_L - \frac{B}{A}\Big( \frac{1}{\delta_U^{(k)}} + \epsilon^{(k)}_U \Big) = \Delta \Big( 1 - \frac{1}{\sqrt{b}} \Big) \ge 0 \,.
\end{align}
This can be easily checked for $k=0$ by a direct calculation. The definition of the variable barrier shifts $\delta_L^{(k)}$ and $\delta_U^{(k)}$ in line~\ref{li:varshifts} of the algorithm further ensures that this expression does not change with $k$.

Let us now assume that for some arbitrary $k\in\{0,\ldots,n-1\}$
\begin{align}\label{induction_hypo}
\sigma(\bA^{(k)})\subset (l^{(k)},u^{(k)}) \quad,\quad \delta^{(k)}_L ,\, \delta^{(k)}_U>0 \,, 
\end{align}
which was already checked for $k=0$.
Then, due to~\eqref{alg1:condition}, condition~\eqref{cond:lem3.5} of Lemma~\ref{lem3.5} is fulfilled with respect to $\bA^{(k)}$.
Hence there is an index $i$ such that condition~\eqref{corcond} in Corollary~\ref{cor:bar_shift} is satisfied for $\by^{i}$ and $L^{(k)}(\by^{i}) \ge U^{(k)}(\by^{i})$, where (see line~\ref{li:LU} of the \texttt{BSS} algorithm)
\[
L^{(k)}(\by^{i}) = L_{\bA^{(k)}}\big(\by^{i};l^{(k)},\delta_L^{(k)}\big) \quad\text{and}\quad   U^{(k)}(\by^{i})= U_{\bA^{(k)}}\big(\by^{i};u^{(k)},\delta_U^{(k)}\big) \,.
\]
Further, based on~\eqref{alg1:condition} and looking into the proof of Lemma~\ref{lem3.5}, it holds
\begin{align*}
\frac{1}{A} \sum_{j=1}^{M}\big[ L^{(k)}(\by^j) - U^{(k)}(\by^j) \big] =
\frac{1}{\delta_L^{(k)}} - \kappa\epsilon^{(k)}_L - \frac{B}{A}\Big( \frac{1}{\delta_U^{(k)}} + \epsilon^{(k)}_U \Big) \ge \Delta \Big( 1 - \frac{1}{\sqrt{b}} \Big)  \,,
\end{align*}
which implies that there is always even an index $i\in [M]$ such that $\by^{i}$ satisfies
\begin{align*}L^{(k)}(\by^i) - U^{(k)}(\by^i) \ge \frac{\Delta}{M} \Big( 1 - \frac{1}{\sqrt{b}} \Big) \,.
\end{align*}
Hence, an index $i$ which satisfies the selection condition in line~\ref{selectCond} is found stably in each iteration of the \texttt{BSS} algorithm. According to Lemma~\ref{lem3.5}, 
such an index in particular satisfies condition~\eqref{corcond} of Corollary~\ref{cor:bar_shift}.

An update of $\bA^{(k)}$ to $\bA^{(k+1)}$ as in line~\ref{li:update} of the \texttt{BSS} algorithm thus yields a matrix with 
\begin{align*}
\sigma(\bA^{(k+1)})\subset (l^{(k+1)},u^{(k+1)}) \quad\text{and}\quad  0<\epsilon_L^{(k+1)}\le \epsilon_L^{(k)} \quad\text{and}\quad  0<\epsilon_U^{(k+1)}\le \epsilon_U^{(k)}  \,.
\end{align*}
From this and $\delta_L^{(k)},\,\delta_U^{(k)}>0$, we deduce
 \begin{align}\label{delta_decr}
 \delta_L^{(k+1)}\ge \delta_L^{(k)}>0 \quad\text{and}\quad  0<\delta_U^{(k+1)}\le \delta_U^{(k)} \,.
 \end{align}
Hence, \eqref{induction_hypo} is fulfilled for $k+1$. Inductively, this proves that \eqref{induction_hypo}
is fulfilled for $k=0,\ldots,n$.
In particular, $\sigma(\bA^{(n)})\subset (l^{(n)},u^{(n)})$ and due to~\eqref{delta_decr}
\[
l^{(n)} = l^{(0)} + \sum_{k=0}^{\lceil bm\rceil-1} \delta^{(k)}_L \ge l^{(0)} +  bm\delta^{(0)}_L = - m\kappa\sqrt{b} + bm >0 \,,
\]
where $b>\kappa^2$ was used in the last estimation step. Hence,
\begin{align*}
0<l^{(n)} \quad\text{and}\quad l^{(n)}\bI \preceq \sum_{i=1}^M \tilde{s}^{(i)} \by^i(\by^i)^\ast \preceq u^{(n)}\bI \,.
\end{align*}
The system
$(\sqrt{\tilde{s}^{(i)}}\by^{i})_{i=1}^M$, where the $\tilde{s}^{(i)}$ are the weights from 
the output line~\ref{li:return} before the rescaling, is thus a frame with bounds $0<l^{(n)}\le u^{(n)} <\infty$. 
Finally, we can estimate with~\eqref{delta_decr}
\begin{align*}
\frac{u^{(n)}}{l^{(n)}} &= 
\frac{u^{(0)} + \sum_{i=0}^{n-1}\delta^{(i)}_U}{l^{(0)}+\sum_{i=0}^{n-1}\delta^{(i)}_L}
\le \frac{u^{(0)} + \lceil bm\rceil\delta^{(0)}_U}{l^{(0)} + \lceil bm\rceil\delta^{(0)}_L}
= \frac{\delta^{(0)}_U}{\delta^{(0)}_L} + \frac{u^{(0)}-l^{(0)}\delta^{(0)}_U/\delta^{(0)}_L}{l^{(0)} + \lceil bm\rceil \delta^{(0)}_L} \notag\\
&\le \frac{\delta^{(0)}_U}{\delta^{(0)}_L} + \frac{u^{(0)}-l^{(0)}\delta^{(0)}_U/\delta^{(0)}_L}{l^{(0)} + bm\delta^{(0)}_L}
=  \frac{u^{(0)} + bm\delta^{(0)}_U}{l^{(0)} + bm\delta^{(0)}_L}
=  \frac{B}{A}  \gamma (1+\Delta) \,.
\end{align*}
The rescaled system $(\sqrt{s^{(i)}}\by^{i})_{i=1}^M$ with the actual output weights $s^{(i)}$ from Algorithm~\ref{algo1:BSS}
has thus frame bounds in the range $[A,B\gamma(1+\Delta)]$ due to
\begin{align*}
l^{(n)} \frac{1}{2} \Big( \frac{A}{l^{(n)}} + \frac{B \gamma (1+\Delta)}{u^{(n)}} \Big) &= \frac{1}{2} \Big( A + \frac{l^{(n)}}{u^{(n)}}B \gamma (1+\Delta) \Big) \ge A \,, \\
u^{(n)} \frac{1}{2} \Big( \frac{A}{l^{(n)}} + \frac{B \gamma (1+\Delta)}{u^{(n)}} \Big) &= \frac{1}{2} \Big( \frac{u^{(n)}}{l^{(n)}}A + B \gamma (1+\Delta) \Big) \le B\gamma(1+\Delta) \,.
\end{align*}
This finishes the proof of Theorem~\ref{BSS}, choosing $\Delta=0$.
\hfill$\blacksquare$

\section{Non-weighted subsampling of finite frames}\label{sec:unweighted}
We now turn to non-weighted versions of the subsampling strategies in 
Sections~\ref{sec:weighted_random}~and~\ref{sec:weighted_bss}.
Our approach is to give estimates on the occurring weights. In this way, we are able to save the lower frame bounds. For many applications those are the important ones as they ensure the stable reconstruction of any vector $\bm a\in \mathds C^m$ from its frame coefficients $\langle\bm a, \bm y^i\rangle$.
Results in this section will be of the following form:

Given vectors $\bm y^1, \dots, \bm y^M$, we seek inequalities of the type
\begin{align}\label{eq:subsample_mean}
  \frac 1M \sum_{i=1}^{M} |\langle \bm a, \bm y^i \rangle |^2
  \le \frac{C}{|J|} \sum_{i\in J} |\langle \bm a, \bm y^i \rangle |^2
  \quad\text{for all}\quad
  \bm a\in \mathds C^m 
\end{align}
for $J \subset [M]$ and some fixed constant $C>0$.
If the initial $\bm y^1, \dots, \bm y^M$ satisfy a lower frame bound, \eqref{eq:subsample_mean} gives that the vectors $\bm y^i$, $i\in J$, satisfy a lower frame bound as well.

For the non-weighted version of the random subsampling in Theorem~\ref{thm:frame_sub} the construction of $\bm{\tilde Y}$ is covered by Lemma~\ref{discreteconstruction1}. We obtain the following result with $|J|=\mathcal{O}(m\log m)$.

\begin{theorem}\label{theorem:unweighted_randomy} Let $(\by^i)_{i=1}^M\subset\field{C}^{m}$ be a frame and $c, p, t \in (0,1)$ and $n\in\field{N}$ be such that
\begin{align*}
  n \ge \frac{3}{ct^2}m\log\left(\frac{m}{p}\right) \,.
\end{align*}
Drawing $n$ indices $J\subset[M]$ (with duplicates) i.i.d.\ according to the discrete probability density $\varrho_i = (1-c)/M + c \cdot \|\bm{\tilde y^i}\|_2^2 / m$ gives
\begin{align*}
    \frac 1M
    \sum_{i=1}^M \left|\left\langle \bm a, \bm y^i \right\rangle\right|^2
    \le \frac{1}{(1-c)(1-t)} \frac{1}{|J|}\sum_{i\in J}
    \left|\left\langle \bm a, \bm y^i \right\rangle\right|^2
    \quad\text{for all}\quad
    \bm a\in \field{C}^{m}
\end{align*}
with probability exceeding $1-p$.
\end{theorem}

\begin{proof} Similar to Theorem~\ref{thm:frame_sub}, the result follows from Tropp's result in Lemma~\ref{matrixchernoff}. 
Here it is applied to the random rank-1 matrices
$\bA_i := \frac{1}{n}\varrho^{-1}_i \bm{\tilde y}^i \otimes \bm{\tilde y}^i$, where $\bm{\tilde y}^1, \dots, \bm{\tilde y}^M \in \mathds C^m$ are the rows of the matrix $\bm{\tilde Y}$ obtained according to Lemma~\ref{discreteconstruction1} from $\bm Y$, the analysis operator~\eqref{eq:analysisoperator} of $(\by^i)_{i=1}^M$. The matrices $\bA_i$ satisfy
\begin{align*}
    \lamax(\bA_i) = \frac{1}{n}\varrho_i^{-1}\|\bm{\tilde y}^i\|_2^2
    \le \frac{m}{c n} \,.
\end{align*}
For $n=|J|$ independent copies $(\bA_i)_{i\in J}$, due to the orthogonality of $\bm{\tilde Y}$, we further have
\begin{align*}
  \sum_{i\in J} \Ept\bA_i
  = \sum_{i\in J} \mathds E \left(\frac{1}{n}\varrho^{-1}_{i} \bm{\tilde y}^{i} \otimes \bm{\tilde y}^{i}\right)
  = \sum_{i\in J} \frac{1}{n} \bm{\tilde Y}^* \bm{\tilde Y}
  = \bI \,,
\end{align*}
where $\bm I$ is the $m\times m$ dimensional identity matrix.
Thus, $\mu_{\min} = \lamin(\sum_{i\in J}\Ept\bA_i) = 1$ and Lemma~\ref{matrixchernoff} gives
\begin{align*}
    \mathds P \left(\lambda_{\min}\left(\frac{1}{n}\sum_{i\in J} \varrho_{i}^{-1} \bm{\tilde y}^{i} \otimes \bm{\tilde y}^{i}\right) \leq 1-t\right)
    &\leq m\exp\left(-\frac{c n}{m}\frac{t^2}{3}\right)\,,
\end{align*}
which is smaller than $p$ by the assumption on $n$.
Using $\varrho_i \ge (1-c)/M$, we obtain
\begin{align*}
  \|\bm{\tilde Y} \bm a\|_2^2
  = \|\bm a\|_2^2
  \le \frac{1}{1-t} \frac 1n \sum_{i\in J} \varrho_i^{-1} |\langle\bm a, \bm{\tilde y}^i\rangle|^2
  \le \frac{M}{(1-c)(1-t)} \frac 1n \|(\bm{\tilde Y}\bm a)|_J\|_2^2
\end{align*}
for all $\bm a\in \field{C}^{m}$ with probability exceeding $1-p$.
By the arguments in \eqref{auxweightedestimate} and after we may replace $\bm{\tilde Y}$ with the original $\bm Y$ to obtain the assertion.
\end{proof} 

Next, we assume that we have a Bessel sequence in $\field{C}^m$ with elements that are norm-bounded from
below. 
Applying Algorithm~\ref{alg2:BSS} (\texttt{BSS$^\perp$}) then yields a non-weighted inequality of type~\eqref{eq:subsample_mean}
with $|J|=\mathcal{O}(m)$.
In Section~\ref{sec:numerics} this algorithm is used in the experiments 1-3.

\begin{lemma}\label{l:nonweighted_bss} Let $(\bm y^i)_{i=1}^M$ be a Bessel sequence in $ \mathds C^{m}$, i.e., a set of vectors satisfying the upper bound in~\eqref{frame} for some $B>0$. Further assume $M\ge m$ and 
    $\|\bm y^i\|_2^2 \ge \beta m/M$ for some $\beta>0$ and all $i\in[M]$. 
    Then, for any $b > 1$, there exists a subset $J \subset [M]$ with $|J| \le \lceil bm \rceil$ (without duplicates) such that
    \begin{align*}
      \frac 1M \sum_{i=1}^M \left|\left\langle \bm a, \bm y^i \right\rangle\right|^2
      \le \frac{B}{\beta} \frac{(\sqrt b + 1)^2}{(\sqrt b - 1)^2} \frac 1m
      \sum_{i\in J}
      \left|\left\langle \bm a, \bm y^i \right\rangle\right|^2
      \quad\text{for all}\quad
      \bm a\in \field{C}^{m}\,.
    \end{align*}
\end{lemma} 

\begin{proof} Applying Algorithm~\ref{alg2:BSS} (\texttt{BSS$^\perp$}) to the sequence $(\bm y^i)_{i=1}^M$ yields weights $s_i\ge0$, where $|\{i:s_i\neq 0 \}|\le\lceil bm\rceil$. Recall that, by the discussion
    of Algorithm~\ref{alg2:BSS}, its application to any input sequence
    $(\bm y^i)_{i=1}^M$ is possible provided $M\ge m$. We obtain~\eqref{auxweightedestimate1}. Taking into account the
    Bessel property of $(\bm y^i)_{i=1}^M$ and choosing $\Delta=0$ in Algorithm~\ref{alg2:BSS} yields
\begin{align}\label{uplowestaux}
      \sum_{i=1}^M \left|\left\langle \bm a, \bm y^i \right\rangle\right|^2
      \le 
      \sum_{i\in J}
      s_i \left|\left\langle \bm a, \bm y^i \right\rangle\right|^2
\le \frac{(\sqrt b+1)^2}{(\sqrt b-1)^2} B \|\bm a\|_2^2 
    \end{align}
    for $J:=\{ i:s_i\neq 0 \}$ and all $\bm a\in \field{C}^{m}$. 
Setting $\bm a = \bm y^j$ for $j\in J$, we obtain by the assumption $\|\bm y^j\|_2^2 \ge \beta m/M$ and the upper estimate in~\eqref{uplowestaux}
    \begin{align*}
      s_j
      \le \frac{(\sqrt b + 1)^2 B}{(\sqrt b-1)^2  \|\bm{y}^j\|_2^{2}}
      \le \frac{B}{\beta} \frac{(\sqrt b + 1)^2}{(\sqrt b - 1)^2} \frac{M}{m} \,.
    \end{align*}
    Thus, by the lower estimate in~\eqref{uplowestaux}, we obtain the assertion.
\end{proof} 

The condition on the norms $\|\bm y^i\|_2$ in Lemma~\ref{l:nonweighted_bss} can be dropped with a more elaborate subsampling strategy, \texttt{PlainBSS} (see below) instead of \texttt{BSS$^\perp$.} The `preconditioning' in \texttt{PlainBSS} is based on Lemma~\ref{discreteconstruction} rather than Lemma~\ref{discreteconstruction1}. The final result is stated in Corollary \ref{frame:smallb}.
The price we pay for this is the dependence of the constant in terms of the oversampling factor $b$. It deteriorates to $(b-1)^{-3}$ while in the previous result it is $(b-1)^{-2}$. 

\begin{lemma}\label{discreteconstruction} Let $\bm Y\in\mathds C^{M\times m}$ be a matrix and $K\in \{0, \dots, M\}$.
  Then there is a matrix $\bm{\tilde Y}\in\mathds C^{M\times m'}$ with 
  $m'\in\{K,\dots, K+m\}$ and rows $\bm{\tilde y}^1, \dots, \bm{\tilde y}^M \in \mathds C^{m'}$ such that
  \begin{align*}
    \range(\bm{\tilde Y}) \supset \range(\bm Y)
    \,,\quad
    \bm{\tilde Y}\herm\bm{\tilde Y} = \bI
    \,,\quad\text{and}\quad
    \|\bm{\tilde y}^i\|_2^2 \ge \frac{K}{M} \,,
  \end{align*}
  where $\bI$ is the $m'\times m'$ dimensional identity matrix.
\end{lemma} 

\begin{proof} Let us denote the columns of $\bm Y$ with $\bm c^1, \dots, \bm c^m$. Further define columns in $\mathds C^M$ by
\begin{align*}
    d^k = \frac{1}{\sqrt M}\Big[ \exp\Big(2\pi\mathrm i k \frac{j}{M}\Big) \Big]_{j=1}^{M}
\end{align*}
for $k = 1, \dots, K$, which are the first $K$ columns of a Fourier matrix.
By construction the system $(\bm d^k)_{k=1}^{K}$ is orthonormal.
It can hence be extended by appropriate vectors $\bm{\tilde c}^1,\ldots,\bm{\tilde c}^l$ to an orthonormal basis of 
\[
\spn\{\bm d^1, \dots, \bm d^{K}, \bm c^1, \dots, \bm c^m\} \,.
\]
Those can be constructed e.g.\ via the Gram-Schmidt algorithm.  
Finally, we set up
\begin{align*}
  \bm{\tilde Y} :=
  \left[\,\bm d^1 \,\middle|\, \cdots \,\middle|\, \bm d^{K} \,\middle|\, \bm{\tilde c}^1 \,\middle|\, \cdots \,\middle|\, \bm{\tilde c}^{l}\,\right]
  = \begin{bmatrix}
  (\bm{\tilde y}^1)\herm\\[-1ex] \hrulefill \\[-1ex] \vdots\\[-1.9ex]\hrulefill \\(\bm{\tilde y}^M)\herm
  \end{bmatrix}
  \in\mathds C^{M\times (K + l)} \,,
\end{align*}
which fulfills the stated conditions.
\end{proof} 

\begin{theorem}\label{thm:unweightes_bss} Let $(\bm y^i)_{i=1}^M$ be a sequence of vectors in $\mathds C^{m}$ and $K\in \{0, \dots, M\}$.
Then, for any $b>1$, a set of indices $J \subset [M]$ (without duplicates) can be constructed (in polynomial time) such that $|J| \le \lceil b(K+m) \rceil$ and
\begin{align}\label{prefactorab}
  \frac 1M \sum_{i=1}^M \left|\left\langle \bm a, \bm y^i \right\rangle\right|^2
  \le \frac{(\sqrt b + 1)^2}{(\sqrt b - 1)^2} \frac{m}{K} \frac{1}{m}
  \sum_{i\in J}
  \left|\left\langle \bm a, \bm y^i \right\rangle\right|^2
  \quad\text{for all}\quad
  \bm a\in \field{C}^{m}\,.
\end{align}
\end{theorem} 

\begin{proof} We construct the vectors $\bm{\tilde y}^1, \dots, \bm{\tilde y}^M$ according to Lemma~\ref{discreteconstruction}.
    They form a tight frame in $\field{C}^{m'}$ with $m'\in\{K, \dots, K+m\}$ and $\|\bm{\tilde y}^i\|_2^2 \ge \frac{K}{M}$ for all $i\in[M]$.
    We can thus apply Lemma~\ref{l:nonweighted_bss} (\texttt{BSS$^\perp$}, which in effect is here \texttt{BSS)} with $B=1$. We obtain a subset 
    $J \subset [M]$ with $|J| \le \lceil bm' \rceil\le  \lceil b(K+m)\rceil$ (without duplicates) such that
    \begin{align*}
      \frac 1M \sum_{i=1}^M \left|\left\langle \bm a, \bm y^i \right\rangle\right|^2
      \le \frac{(\sqrt b + 1)^2}{(\sqrt b - 1)^2} \frac{1}{K}
      \sum_{i\in J}
      \left|\left\langle \bm a, \bm y^i \right\rangle\right|^2
      \quad\text{for all}\quad
      \bm a\in \field{C}^{m}\,
    \end{align*}
    which finishes the proof.
\end{proof} 

A result in terms of the `real' oversampling factor $b'$ in Theorem~\ref{thm:unweightes_bss}, determined by $\lceil b'm\rceil = \lceil b(K+m)\rceil$, is given in Corollary~\ref{frame:smallb}.

\begin{corollary}\label{frame:smallb} Let $\bm y^1, \dots, \bm y^M \in\field{C}^{m}$ be vectors with $m\in\field{N}$. Further, take $b' > 1+\frac{1}{m}$ and assume $M\ge\lceil b'm\rceil$.
We then obtain indices $J'\subset[M]$ with $|J'| \le \lceil b' m\rceil$ such that
\begin{align*}
  \frac 1M \sum_{i=1}^M \left|\left\langle \bm a, \bm y^i \right\rangle\right|^2
  \le 89\frac{(b'+1)^2}{(b'-1)^3}
  \frac 1m
  \sum_{i\in J'}
  \left|\left\langle \bm a, \bm y^i \right\rangle\right|^2
  \quad\text{for all}\quad
  \bm a\in \field{C}^{m} \,.
\end{align*}
\end{corollary} 

\begin{proof} The idea is to apply Theorem~\ref{thm:unweightes_bss} for specifically chosen $K \in\mathds N$ and $b > 1$, such that $\lceil b(m+K)\rceil\le\lceil b'm\rceil $ for the given $b'$ and the prefactor in~\eqref{prefactorab} becomes small.
    Theorem~\ref{thm:unweightes_bss} yields the prefactor
    \begin{equation*}
        \frac{(\sqrt b + 1)^2}{(\sqrt b - 1)^2} \frac{m}{K}
        = \frac{(\sqrt b + 1)^4}{(b-1)^2}\frac{m}{K}
        \le 4 \frac{(b + 1)^2}{(b-1)^2}\frac{m}{K}
        \eqqcolon C(K)\,.
    \end{equation*}
    Choosing $b$ and $K$ such that $b' = b\frac{m+K}{m}$ gives
    \begin{align*}
        b=b' / (1 + K/m) \,,\quad b+1 = (b'+1 + K/m)/(1 + K/m) \,,\quad   b-1 = (b'-1 - K/m)/(1 + K/m) \,,
    \end{align*}
    and hence
    \begin{equation}\label{CKconstant}
        C(K) = 4 \Big(\frac{b'+1+\frac{K}{m}}{b'-1-\frac{K}{m}}\Big)^2\frac{m}{K} \,.
    \end{equation}
    
    We now choose $K^\star = \lceil \frac{(b'-1)m}{8} \rceil\in [M]$. Assuming $b' \ge 1+4/m$, we can then bound
    \begin{equation*}
        \frac{b'-1}{8} \le \frac{K^\star}{m}
        \le \frac{b'-1}{4} \,.
    \end{equation*}
    Further, since $b'-1-K^\star/m > 0$, we arrive at the estimate
    \begin{align}\label{CKestimate1}
        C(K^\star)
        \le 4\Big(\frac{b'+1+ (b'-1)/4}{b'-1- (b'-1)/4}\Big)^2\frac{8}{b'-1} 
        = \frac{32}{b'-1} \Big(\frac{5b'/4 + 3/4)}{3(b'-1)/4}\Big)^2 \le 32\Big(\frac{5}{3}\Big)^2 \frac{(b'+1)^2}{(b'-1)^3} \le 89  \frac{(b'+1)^2}{(b'-1)^3} \,.
    \end{align}
    Next, we consider the cases $b'=1+2/m$ and $b'=1+3/m$ separately, where in both $K^\star=1$. 
    The associated $b$ are given by $b=1+1/(m+1)$ and $b=1+2/(m+1)$. Further
    $1/m=(b'-1)/2$ and $1/m=(b'-1)/3$. Inserting these values into \eqref{CKconstant}, 
    we obtain estimates for $C(K^\star)$ as in~\eqref{CKestimate1}. The prefactors, being $72$ and $48$, 
    are even smaller than $89$. 
    Finally, to extend the estimate~\eqref{CKestimate1} to the whole range $b'>1+\tfrac{1}{m}$, note that the right-hand side of~\eqref{CKestimate1} is increasing for $b'\searrow 1$. Taking into account $\lceil b'm \rceil=m+k+1$ for each $k\in\field{N}$ and $b'\in(1+\tfrac{k}{m},1+\tfrac{k+1}{m}]$, we are finished. 
\end{proof}

Building on the proof of Corollary~\ref{frame:smallb}, we now formulate Algorithm~\ref{alg:plainBSS} (\texttt{PlainBSS}).
Like Algorithm~\ref{algo1:BSS} (\texttt{BSS}) and Algorithm~\ref{alg2:BSS} (\texttt{BSS$^\perp$}), it is polynomial in time.

\noindent
\begin{algorithm}[H]
\caption{\texttt{PlainBSS}}\label{alg:plainBSS}
  \begin{tabularx}{\textwidth}{lXr}
    \textbf{Input:} & Vectors $\bm y^1,\dots,\bm y^M\in\field{C}^m$ with $m\in\mathds N$ and $M\ge m+2$; \\
    & Oversampling factor $b'$ s.t.\ $m+2\le\lceil b'm\rceil\le M$; Stability factor $\Delta\ge 0$. & \\\hline
    \textbf{Output:} & 
        Indices $J \subset [M]$ such that $|J| \le \lceil b'm\rceil$ and\\
        &$\frac 1M \sum_{i=1}^{M} |\langle\bm a, \bm y^i\rangle|^2 \le 89\frac{(b'+1)^2}{(b'-1)^3} \frac {1+\Delta}{m} \sum_{i\in J}|\langle \bm a, \bm y^i\rangle|^2$ .
  \end{tabularx}
\end{algorithm}
\vspace*{-14pt}
\begin{algorithmic}[1]
\STATE{
Compute $K^\star$ and $b$ from $b'$ as in the proof of Corollary~\ref{frame:smallb}.
}
\STATE{
    Construct vectors $\bm{\tilde y}^1, \dots, \bm{\tilde y}^M\in \field{C}^{m'}$ with $K=K^\star$ according to Lemma~\ref{discreteconstruction}, where the initial $\bm Y\in\mathds C^{M\times m}$ is the matrix~\eqref{eq:analysisoperator} with rows $(\bm y^1)\herm,\dots,(\bm y^M)\herm$ .
}
\STATE{
    Apply Algorithm~\ref{algo1:BSS} (\texttt{BSS}) to $\bm{\tilde y}^1, \dots, \bm{\tilde y}^M$ with oversampling factor $b$ and stability factor $\Delta$ to obtain weights $s_1,\dots,s_M$.
}
\RETURN{
    indices $J \coloneqq \{i : s_i \neq 0\}$.
}
\end{algorithmic}
\vspace*{-10pt}
\noindent
\rule{\textwidth}{0.8pt}\\[-12pt]
\rule{\textwidth}{0.8pt}
\vspace{1ex}

\noindent
For a better runtime, it might sometimes be advantageous to
combine \texttt{BSS} subsampling with a preceding random subsampling step. Theorem~\ref{thm:frame_sub} could be used, for instance, to quickly reduce the number of vectors to $\mathcal{O}(m\log(m))$ in case of very large $M$.
In the following corollary such a two-step procedure is used to construct a 
unit-norm frame with very few (close to~$m$) elements and well-behaved frame bounds.
Here it is crucial that the \texttt{BSS} algorithm returns no duplicates, which is used in the proof.

\begin{corollary} Assume that the vectors $\bm y^1, \dots, \bm y^M \in\mathds C^m, m\in\field{N}$, form a tight frame
  and let $b' > 1+\frac{1}{m}$.
  Further choose $p,t \in (0,1)$ and draw
  \begin{align*}
    n := \left\lceil \frac{3}{t^2}m\log\left(\frac{2m}{p}\right) \right\rceil 
  \end{align*}
  indices $J\subset[M]$ (with duplicates) i.i.d.\ according to the discrete probability density $\varrho_i = \|\bm y^i\|_2^2 /\|\bm Y\|_F^2$. In case $n>\lceil b'm\rceil$, those can further be subsampled
  using \texttt{BSS} (with oversampling factor $b'$) giving $J' \subset J$ with $|J'| \le \lceil b'm \rceil$
  and a unit-norm frame $(\bm y^i/\|\bm y^i\|_2)_{i\in J'}$ satisfying
  \begin{align*}
      \frac{(1-t)(b'-1)^3}{89(b'+1)^2}\|\bm a\|_2^2 
      \le \sum_{i\in J'} \left|\left\langle \bm a, \frac{\bm y^i}{\|\bm y^i\|_2} \right\rangle\right|^2 
      \leq (1+t)\left\lceil\frac{3\log(2m/p)}{t^2}\right\rceil \|\bm a\|_2^2
  \end{align*}
  for all $\bm a\in \field{C}^{m}$ with probability exceeding $1-p$. Otherwise, when $n\le\lceil b'm\rceil$,
  the frame $(\bm y^i/\|\bm y^i\|_2)_{i\in J'}$ with $J'=J$ already satisfies $|J'| \le \lceil b'm \rceil$ and~\eqref{eq:lsakjd} for all $\bm a\in \field{C}^{m}$ with probability exceeding $1-p$.
\end{corollary} 

\begin{proof} By \eqref{chain} and $(\by^i)_{i=1}^{M}$ forming a tight frame, we have $\|\bm Y\|_F^2 = mA$.
  By Theorem~\ref{thm:frame_sub} we first obtain a subframe with $n=|J|$ elements such that
  \begin{align}\label{eq:lsakjd}
      \frac{1-t}{m} \|\bm a\|_2^2
      \le \frac{1}{n} \sum_{i\in J} \left|\left\langle\bm a, \frac{\bm y^i}{\|\bm y^i\|_2} \right\rangle\right|^2
      \le \frac{1+t}{m} \|\bm a\|_2^2 \,.
  \end{align}
  Next, if we apply Algorithm~\ref{alg:plainBSS} (\texttt{PlainBSS}) to this subframe, we obtain $J'\subset J$ 
  with  $|J'|\le\lceil b'm \rceil$ such that
  \begin{align*}
      \frac{1}{n} \sum_{i\in J} \left|\left\langle\bm a, \frac{\bm y^i}{\|\bm y^i\|_2} \right\rangle\right|^2
      \le 89 \frac{(b'+1)^2}{(b'-1)^3} \frac 1m \sum_{i\in J'} \left|\left\langle\bm a, \frac{\bm y^i}{\|\bm y^i\|_2}\right\rangle\right|^2 \,,
  \end{align*}
  which is used in the lower frame bound.
For the upper frame bound we use that $J'$ has no duplicates, wherefore
  \begin{align*}
    \frac 1m \sum_{i\in J'} \left|\left\langle\bm a, \frac{\bm y^i}{\|\bm y^i\|_2}\right\rangle\right|^2
    \le \left\lceil\frac{3\log(2m/p)}{t^2}\right\rceil \frac{1}{n} \sum_{i\in J} \left|\left\langle\bm a, \frac{\bm y^i}{\|\bm y^i\|_2}\right\rangle\right|^2 \,.
  \end{align*}
  Here, the relation of $n$ and $m$ was used.
  Last, we use the upper frame bound \eqref{eq:lsakjd} and obtain the assertion.
\end{proof} %

\section{Numerical results}
\label{sec:numerics}
In this section we test the unweighted \texttt{BSS}, \texttt{BSS$^\perp$}, and \texttt{PlainBSS} (Algorithms \ref{algo1:BSS}, \ref{alg2:BSS}, and \ref{alg:plainBSS}) in practice. Note that there are further recent attempts to reduce the sampling budget in least squares approximations in practice, see \cite{HaNoPe22}.
A survey on different probabilistic sampling strategies for sparse recovery of multivariate functions can be found in~\cite{AdCaDexMo22} (here especially Sec.~1.4 provides many further references). In addition, let us mention~\cite{AdBru22}, where Adcock and Brugiapaglia give theoretical and empirical evidence of the near-optimal performance of simple Monte Carlo sampling for the recovery of smooth functions in high dimensions.

For the first three experiments, we use the rows of a $d$-dimensional Fourier matrix as initial frame, i.e.,
\begin{align}\label{eq:fourierframe}
    \bm y^i = \left[ \frac{1}{\sqrt M} \exp(2\pi\mathrm i\langle \bm k, \bm x^i\rangle) \right]_{\bm{k}\in I} 
    \quad\text{for}\quad i\in[M]\,,
\end{align}
where $I \subset \mathds Z^d$ are $|I| = m$ frequencies determining the dimension of the frame elements and the nodes $\bm X = (\bm x^1, \dots, \bm x^M)\subset\field{C}^d$ determine their number.
In the experiments, we will have a look at different choices for these frequencies $I$ and nodes $\bm X$.
Note that construction~\eqref{eq:fourierframe} gives an equal-norm frame.

\paragraph{Experiment~1}

We choose dimension $d=2$ and, in the frequency domain, we use a so-called dyadic hyperbolic cross
\begin{align*}
    I = H_R^d = \bigcup_{\substack{\bm l\in\mathds N_0^d\\ \|\bm l\|_1 = R}} \hat G_{\bm l}
    \quad\text{with}\quad
    \hat G_{\bm l} = \bigtimes_{j=1}^{d} \hat G_{l_j}
    \quad\text{and}\quad
    \hat G_l  = \mathds Z\cap(-2^{l-1}, 2^{l-1}] \,,
\end{align*}
which occurs naturally when approximating in Sobolev spaces with mixed smoothness, cf.\ \cite{DuTeUl19}.
Here, we use $R = 6$, which results in $256$ frequencies.
In spatial domain, the canonical candidate are sparse grids:
\begin{align*}
    S_R^d = \bigcup_{\substack{j\in\mathds N_0^d\\ \|\bm l\|_1 = R}} G_{\bm l}
    \quad\text{with}\quad
    G_{\bm l} = \bigtimes_{j=1}^{d} G_{l_j}
    \quad\text{and}\quad
    G_l = 2^{-l}(\mathds Z\cap[0,2^l))\,.
\end{align*}
Sparse grids have the minimal amount of nodes $n=m$ and reconstruct every frequency $\bm k\in H_R^d$, i.e., $A>0$.
Precise estimates on the frame bounds of these matrices are found in \cite[Thm.~3.1]{KK11}.

To test the \texttt{BSS} algorithm we use an initial $65\times 65$ equispaced grid
\begin{align*}
    \bm X = \left\{
    \frac{i}{\sqrt[d]{M}}
    : \bm i\in\{0,\dots,\sqrt[d]{M}-1\}^d
    \right\} \,,
\end{align*}
which has $M=4225$ nodes and is exact ($A=B=1$) for the $M$ frequencies $\bm k\in \{-(\sqrt[d]{M}-1)/2, \dots, (\sqrt[d]{M}-1)/2\}^d$, cf.\ \cite[Sec.~4.4.3]{PlPoStTa18}, in particular for the given dyadic hyperbolic cross.
These initial frequencies and nodes can be seen in the first three graphs of Figure~\ref{fig:d2_sparse_grid}.

\begin{figure}
    \centering
    \begingroup
  \makeatletter
  \providecommand\color[2][]{\GenericError{(gnuplot) \space\space\space\@spaces}{Package color not loaded in conjunction with
      terminal option `colourtext'}{See the gnuplot documentation for explanation.}{Either use 'blacktext' in gnuplot or load the package
      color.sty in LaTeX.}\renewcommand\color[2][]{}}\providecommand\includegraphics[2][]{\GenericError{(gnuplot) \space\space\space\@spaces}{Package graphicx or graphics not loaded}{See the gnuplot documentation for explanation.}{The gnuplot epslatex terminal needs graphicx.sty or graphics.sty.}\renewcommand\includegraphics[2][]{}}\providecommand\rotatebox[2]{#2}\@ifundefined{ifGPcolor}{\newif\ifGPcolor
    \GPcolortrue
  }{}\@ifundefined{ifGPblacktext}{\newif\ifGPblacktext
    \GPblacktexttrue
  }{}\let\gplgaddtomacro\g@addto@macro
\gdef\gplbacktext{}\gdef\gplfronttext{}\makeatother
  \ifGPblacktext
\def\colorrgb#1{}\def\colorgray#1{}\else
\ifGPcolor
      \def\colorrgb#1{\color[rgb]{#1}}\def\colorgray#1{\color[gray]{#1}}\expandafter\def\csname LTw\endcsname{\color{white}}\expandafter\def\csname LTb\endcsname{\color{black}}\expandafter\def\csname LTa\endcsname{\color{black}}\expandafter\def\csname LT0\endcsname{\color[rgb]{1,0,0}}\expandafter\def\csname LT1\endcsname{\color[rgb]{0,1,0}}\expandafter\def\csname LT2\endcsname{\color[rgb]{0,0,1}}\expandafter\def\csname LT3\endcsname{\color[rgb]{1,0,1}}\expandafter\def\csname LT4\endcsname{\color[rgb]{0,1,1}}\expandafter\def\csname LT5\endcsname{\color[rgb]{1,1,0}}\expandafter\def\csname LT6\endcsname{\color[rgb]{0,0,0}}\expandafter\def\csname LT7\endcsname{\color[rgb]{1,0.3,0}}\expandafter\def\csname LT8\endcsname{\color[rgb]{0.5,0.5,0.5}}\else
\def\colorrgb#1{\color{black}}\def\colorgray#1{\color[gray]{#1}}\expandafter\def\csname LTw\endcsname{\color{white}}\expandafter\def\csname LTb\endcsname{\color{black}}\expandafter\def\csname LTa\endcsname{\color{black}}\expandafter\def\csname LT0\endcsname{\color{black}}\expandafter\def\csname LT1\endcsname{\color{black}}\expandafter\def\csname LT2\endcsname{\color{black}}\expandafter\def\csname LT3\endcsname{\color{black}}\expandafter\def\csname LT4\endcsname{\color{black}}\expandafter\def\csname LT5\endcsname{\color{black}}\expandafter\def\csname LT6\endcsname{\color{black}}\expandafter\def\csname LT7\endcsname{\color{black}}\expandafter\def\csname LT8\endcsname{\color{black}}\fi
  \fi
    \setlength{\unitlength}{0.0500bp}\ifx\gptboxheight\undefined \newlength{\gptboxheight}\newlength{\gptboxwidth}\newsavebox{\gptboxtext}\fi \setlength{\fboxrule}{0.5pt}\setlength{\fboxsep}{1pt}\definecolor{tbcol}{rgb}{1,1,1}\begin{picture}(8220.00,2820.00)\gplgaddtomacro\gplbacktext{\csname LTb\endcsname \put(820,2504){\makebox(0,0){\scriptsize frequencies $I$}}\csname LTb\endcsname \put(820,806){\makebox(0,0){\scriptsize $m = 256$}}}\gplgaddtomacro\gplfronttext{}\gplgaddtomacro\gplbacktext{\csname LTb\endcsname \put(2460,2504){\makebox(0,0){\scriptsize sparse grid}}\csname LTb\endcsname \put(2460,806){\makebox(0,0){\scriptsize $M=256$ ($b=1$)}}\csname LTb\endcsname \put(2460,593){\makebox(0,0){\scriptsize $A=0.04336$}}\csname LTb\endcsname \put(2460,381){\makebox(0,0){\scriptsize $B=16$}}}\gplgaddtomacro\gplfronttext{}\gplgaddtomacro\gplbacktext{\csname LTb\endcsname \put(4100,2504){\makebox(0,0){\scriptsize initial nodes}}\csname LTb\endcsname \put(4100,806){\makebox(0,0){\scriptsize $M=4225$ ($b\approx 16.5$)}}\csname LTb\endcsname \put(4100,593){\makebox(0,0){\scriptsize $A=1$}}\csname LTb\endcsname \put(4100,381){\makebox(0,0){\scriptsize $B=1$}}}\gplgaddtomacro\gplfronttext{}\gplgaddtomacro\gplbacktext{\csname LTb\endcsname \put(5740,2504){\makebox(0,0){\scriptsize BSS subsampling}}\csname LTb\endcsname \put(5740,806){\makebox(0,0){\scriptsize $n=384$ ($b\approx 1.5$)}}\csname LTb\endcsname \put(5740,593){\makebox(0,0){\scriptsize $A=0.06672$}}\csname LTb\endcsname \put(5740,381){\makebox(0,0){\scriptsize $B=2.58239$}}}\gplgaddtomacro\gplfronttext{}\gplgaddtomacro\gplbacktext{\csname LTb\endcsname \put(7380,2504){\makebox(0,0){\scriptsize random subsampling}}\csname LTb\endcsname \put(7380,806){\makebox(0,0){\scriptsize $n=384$ ($b\approx 1.5$)}}\csname LTb\endcsname \put(7380,593){\makebox(0,0){\scriptsize $A=0.02379$}}\csname LTb\endcsname \put(7380,381){\makebox(0,0){\scriptsize $B=3.18402$}}}\gplgaddtomacro\gplfronttext{}\gplbacktext
    \put(0,0){\includegraphics[width={411.00bp},height={141.00bp}]{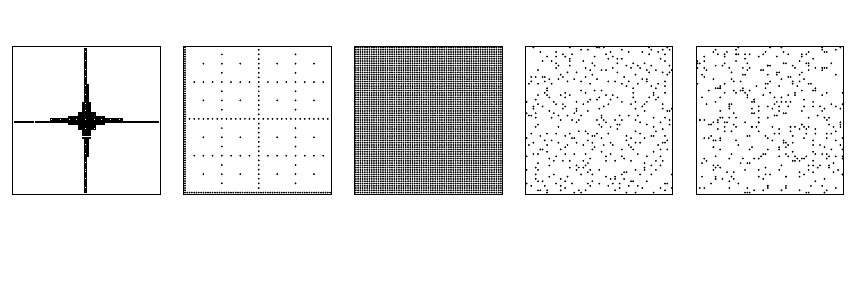}}\gplfronttext
  \end{picture}\endgroup
     \caption{Two-dimensional experiment with sparse grid.}
    \label{fig:d2_sparse_grid}
\end{figure}

On the resulting frame constructed according to \eqref{eq:fourierframe} we apply the unweighted \texttt{BSS} algorithm (discarding the weights $s_i$) with a target oversampling of $b=1.5$ to obtain the subset $J$ and compute the new frame bounds.
For comparison, we draw a random subset (with replacement) of the same size and compute the frame bounds as well.
Note, that we do not have theoretical bounds for these few random nodes.
The results are depicted in the two rightmost graphs of Figure~\ref{fig:d2_sparse_grid}.

Since $\|\bm y^i\|_2^2 = m$, we obtain by Lemma~\ref{l:nonweighted_bss} the theoretical lower frame bound $A = (\sqrt b-1)^2/(\sqrt b+1)^2 = 0.01021$ (cf.\ Lemma~\ref{l:nonweighted_bss}) where we observe $A = 0.06672$ in the experiment.
This is better by a factor of $4$ when compared to random subsampling, where we obtain a lower frame bound of $A = 0.02379$.
Furthermore, the \texttt{BSS} algorithm gives a smaller upper frame constant than random subsampling, but this is not covered by our theory.
The lower frame bound  of the \texttt{BSS} subsampled nodes is bigger than the lower frame bound of the sparse grid.
Even using the next biggest sparse grid with $n = 576$ nodes this still holds, as the frame bounds are $A = 0.06126$ and $B = 14.44698$.

Following \cite{KK11} the frame bounds worsen for the sparse grids in higher dimensions.
We conducted the same experiment in five dimensions with dyadic hyperbolic cross with $m=1002$ frequencies with the following outcome:

\begin{center}
\begin{small}
\begin{tabular}{cc|cc|cc}
& & $b$ & $n$ & $A$ & $B$\\
\hline
\multirow{3}*{sparse grids} & $S_5^5$ & 1.00 & 1002 & 0.00009 & 89.5249 \\
& $S_5^6$ & 2.96 & 2972 & 0.00063 & 74.5446 \\
& $S_5^7$ & 8.46 & 8472 & 0.00158 & 63.5213 \\
\hline
\multirow{3}*{Frolov nodes} &  & 1.02 & 1021 & 0.00008 & 3.13560 \\
& & 2.05 & 2051 & 0.08128 & 2.14287 \\
& & 4.08 & 4093 & 0.37502 & 1.79493 \\
\hline
\multirow{7}*{\texttt{BSS}} & & 1.01 & 1013 & 0.00012 & 3.69835 \\
& & 1.50 & 1503 & 0.04333 & 2.99637 \\
& & 2.00 & 2004 & 0.10659 & 2.61729 \\
& & 2.50 & 2505 & 0.16325 & 2.39153 \\
& & 2.96 & 2966 & 0.20790 & 2.24841 \\
& & 3.50 & 3507 & 0.25682 & 2.10744 \\
& & 4.08 & 4089 & 0.30101 & 2.00187
\end{tabular}
\end{small}
\end{center}

We cannot set $b=1$ with the \texttt{BSS} algorithm, but already for $b = 1.01$ we achieve a slightly better lower frame bound $A$ than for the sparse grid.
When $b$ increases is where the \texttt{BSS} algorithm shows its advantage as the frame bounds become progressively better.

\paragraph{Experiment~2}
As the components of the frame elements $\bm y^i$ are continuous, we have similar frame elements for close nodes $\bm x^i$ and $\bm x^j$.
For the next experiment, we again are in dimension $d=2$ and choose the full grid of frequencies $I = [-6,6]\cap\mathds Z^2$ with $m=169$ frequencies for which the full grid of $13\times 13 = 169$ nodes is barely exact.
For the nodes we use two $13\times 13$ point grids where one is slightly moved by $[0.01, 0.01]\transp$, which is depicted in the two leftmost plots of Figure~\ref{fig:d2_grid}.
This setting is a union of two tight frames, itself a tight frame, where each element has a close duplicate which occur as pairs.
A reasonable subsampling technique would pick at least one out of each pair.
We set a target oversampling factor of $b = 1.1$ and apply the unweighted \texttt{BSS} algorithm and random subsampling for comparison.
The results are depicted in the two rightmost graphs of Figure~\ref{fig:d2_grid}.

\begin{figure}
    \centering
    \begingroup
  \makeatletter
  \providecommand\color[2][]{\GenericError{(gnuplot) \space\space\space\@spaces}{Package color not loaded in conjunction with
      terminal option `colourtext'}{See the gnuplot documentation for explanation.}{Either use 'blacktext' in gnuplot or load the package
      color.sty in LaTeX.}\renewcommand\color[2][]{}}\providecommand\includegraphics[2][]{\GenericError{(gnuplot) \space\space\space\@spaces}{Package graphicx or graphics not loaded}{See the gnuplot documentation for explanation.}{The gnuplot epslatex terminal needs graphicx.sty or graphics.sty.}\renewcommand\includegraphics[2][]{}}\providecommand\rotatebox[2]{#2}\@ifundefined{ifGPcolor}{\newif\ifGPcolor
    \GPcolortrue
  }{}\@ifundefined{ifGPblacktext}{\newif\ifGPblacktext
    \GPblacktexttrue
  }{}\let\gplgaddtomacro\g@addto@macro
\gdef\gplbacktext{}\gdef\gplfronttext{}\makeatother
  \ifGPblacktext
\def\colorrgb#1{}\def\colorgray#1{}\else
\ifGPcolor
      \def\colorrgb#1{\color[rgb]{#1}}\def\colorgray#1{\color[gray]{#1}}\expandafter\def\csname LTw\endcsname{\color{white}}\expandafter\def\csname LTb\endcsname{\color{black}}\expandafter\def\csname LTa\endcsname{\color{black}}\expandafter\def\csname LT0\endcsname{\color[rgb]{1,0,0}}\expandafter\def\csname LT1\endcsname{\color[rgb]{0,1,0}}\expandafter\def\csname LT2\endcsname{\color[rgb]{0,0,1}}\expandafter\def\csname LT3\endcsname{\color[rgb]{1,0,1}}\expandafter\def\csname LT4\endcsname{\color[rgb]{0,1,1}}\expandafter\def\csname LT5\endcsname{\color[rgb]{1,1,0}}\expandafter\def\csname LT6\endcsname{\color[rgb]{0,0,0}}\expandafter\def\csname LT7\endcsname{\color[rgb]{1,0.3,0}}\expandafter\def\csname LT8\endcsname{\color[rgb]{0.5,0.5,0.5}}\else
\def\colorrgb#1{\color{black}}\def\colorgray#1{\color[gray]{#1}}\expandafter\def\csname LTw\endcsname{\color{white}}\expandafter\def\csname LTb\endcsname{\color{black}}\expandafter\def\csname LTa\endcsname{\color{black}}\expandafter\def\csname LT0\endcsname{\color{black}}\expandafter\def\csname LT1\endcsname{\color{black}}\expandafter\def\csname LT2\endcsname{\color{black}}\expandafter\def\csname LT3\endcsname{\color{black}}\expandafter\def\csname LT4\endcsname{\color{black}}\expandafter\def\csname LT5\endcsname{\color{black}}\expandafter\def\csname LT6\endcsname{\color{black}}\expandafter\def\csname LT7\endcsname{\color{black}}\expandafter\def\csname LT8\endcsname{\color{black}}\fi
  \fi
    \setlength{\unitlength}{0.0500bp}\ifx\gptboxheight\undefined \newlength{\gptboxheight}\newlength{\gptboxwidth}\newsavebox{\gptboxtext}\fi \setlength{\fboxrule}{0.5pt}\setlength{\fboxsep}{1pt}\definecolor{tbcol}{rgb}{1,1,1}\begin{picture}(8220.00,2820.00)\gplgaddtomacro\gplbacktext{\csname LTb\endcsname \put(1025,2661){\makebox(0,0){\strut{}frequencies $I$}}\csname LTb\endcsname \put(1025,648){\makebox(0,0){\strut{}$m = 169$}}}\gplgaddtomacro\gplfronttext{}\gplgaddtomacro\gplbacktext{\csname LTb\endcsname \put(3075,2661){\makebox(0,0){\strut{}initial nodes}}\csname LTb\endcsname \put(3075,648){\makebox(0,0){\strut{}$M = 338$ ($b = 2$)}}\csname LTb\endcsname \put(3075,397){\makebox(0,0){\strut{}$A=1$}}\csname LTb\endcsname \put(3075,145){\makebox(0,0){\strut{}$B=1$}}}\gplgaddtomacro\gplfronttext{}\gplgaddtomacro\gplbacktext{\csname LTb\endcsname \put(5125,2661){\makebox(0,0){\strut{}BSS subsampling}}\csname LTb\endcsname \put(5125,648){\makebox(0,0){\strut{}$n = 182$ ($b \approx 1.1$)}}\csname LTb\endcsname \put(5125,397){\makebox(0,0){\strut{}$A=0.49471$}}\csname LTb\endcsname \put(5125,145){\makebox(0,0){\strut{}$B=1.81720$}}}\gplgaddtomacro\gplfronttext{}\gplgaddtomacro\gplbacktext{\csname LTb\endcsname \put(7175,2661){\makebox(0,0){\strut{}random subsampling}}\csname LTb\endcsname \put(7175,648){\makebox(0,0){\strut{}$n = 182$ ($b \approx 1.1$)}}\csname LTb\endcsname \put(7175,397){\makebox(0,0){\strut{}$A=0$}}\csname LTb\endcsname \put(7175,145){\makebox(0,0){\strut{}$B=5.38234$}}}\gplgaddtomacro\gplfronttext{}\gplbacktext
    \put(0,0){\includegraphics[width={411.00bp},height={141.00bp}]{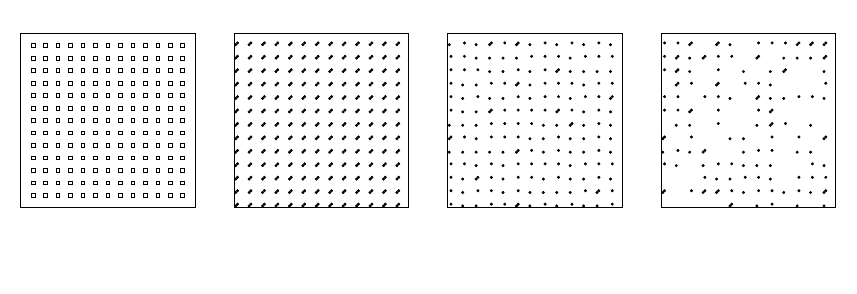}}\gplfronttext
  \end{picture}\endgroup
     \caption{Two-dimensional experiment with frequencies on the grid}
    \label{fig:d2_grid}
\end{figure}

As in the first experiment, we have the theoretical lower frame bound $A = (\sqrt b-1)^2/(\sqrt b+1)^2 = 0.00057$ (cf.\ Lemma~\ref{l:nonweighted_bss}) where we observe $A = 0.49471$ in the experiment.
For random subsampling we do not pick one frame element of each pair creating holes which spoil the lower frame bound.
In fact, the subsampling is not even a frame anymore as $A = 0$.

\paragraph{Experiment~3}
As our algorithms do not depend on the dimension, for the next experiment, we choose $d = 25$.
In frequency domain we choose $m=500$ random frequencies in $[-1000, 1000]^{25}\cap\mathds Z^{25}$.
In time domain we use two different choice:
\begin{itemize}
\item
    We use a full grid with $M = 2001^{25} > 10^{80}$ nodes, which is exact for all possible frequencies.
\item
    We use $M = \lceil 6m\log(m)\rceil = 18\,644$ random nodes.
    In Lemma~\ref{lem:aux1Finite}, we show that this gives frame bounds $A = 1/2$ and $B = 3/2$ with high probability.
\end{itemize}

For ten different choices of $b \in (1, 2]$ we use the unweighted \texttt{BSS} algorithm for the grid and unweighted \texttt{BSS$^\perp$} for the random nodes.
We compute the new frame bounds and count the inner iterations (i.i.) of the \texttt{BSS} algorithm in line~6.
Further, we compute the theoretical frame bounds $1/B\cdot (\sqrt b-1)^2/(\sqrt b+1)^2$ from Lemma~\ref{l:nonweighted_bss}.
The results are shown in the table below and Figure~\ref{fig:d25}.

\begin{center}
\begin{small}
\begin{tabular}{cc|cccc|cccc}
&& \multicolumn{4}{c|}{\textbf{Grid nodes}} & \multicolumn{4}{c}{\textbf{Random nodes}} \\
&& \multicolumn{4}{c|}{$M=2001^{25}$ ($b \approx 7\cdot 10^{79}$)} & \multicolumn{4}{c}{$M = 18\,644$ ($b\approx 37$)} \\
&& \multicolumn{4}{c|}{$A=B=1$} & \multicolumn{4}{c}{$A = 0.70, B = 1.34$} \\
\hline
$b$ & $n$ & $A$ & bound & $B$ & i.i. & $A$ & bound & $B$ & i.i. \\
\hline
1.02 & 510 & $3.72\cdot 10^{-4}$ & $2.45 \cdot 10^{-5}$ & 3.81 & 1.5 & $2.70\cdot 10^{-4}$ & $1.83\cdot 10^{-5}$ & 3.84 & 1.4\\
1.12 & 564 & $5.59\cdot 10^{-3}$ & $8.02 \cdot 10^{-4}$ & 3.63 & 1.6 & $4.68\cdot 10^{-3}$ & $5.99\cdot 10^{-4}$ & 3.67 & 1.4\\
1.23 & 618 & $1.46\cdot 10^{-2}$ & $2.67 \cdot 10^{-3}$ & 3.46 & 1.5 & $1.26\cdot 10^{-2}$ & $2.00\cdot 10^{-3}$ & 3.53 & 1.4\\
1.34 & 673 & $2.63\cdot 10^{-2}$ & $5.33 \cdot 10^{-3}$ & 3.35 & 1.6 & $2.28\cdot 10^{-2}$ & $3.98\cdot 10^{-3}$ & 3.39 & 1.4\\
1.45 & 727 & $3.92\cdot 10^{-2}$ & $8.58 \cdot 10^{-3}$ & 3.22 & 1.5 & $3.56\cdot 10^{-2}$ & $6.40\cdot 10^{-3}$ & 3.24 & 1.4\\
1.56 & 782 & $5.14\cdot 10^{-2}$ & $1.23 \cdot 10^{-2}$ & 3.11 & 1.5 & $4.89\cdot 10^{-2}$ & $9.15\cdot 10^{-3}$ & 3.14 & 1.4\\
1.67 & 836 & $6.63\cdot 10^{-2}$ & $1.63 \cdot 10^{-2}$ & 3.01 & 1.6 & $5.77\cdot 10^{-2}$ & $1.21\cdot 10^{-2}$ & 3.04 & 1.4\\
1.78 & 891 & $7.90\cdot 10^{-2}$ & $2.05 \cdot 10^{-2}$ & 2.94 & 1.5 & $7.02\cdot 10^{-2}$ & $1.53\cdot 10^{-2}$ & 3.01 & 1.4\\
1.89 & 940 & $9.02\cdot 10^{-2}$ & $2.49 \cdot 10^{-2}$ & 2.90 & 1.6 & $7.96\cdot 10^{-2}$ & $1.86\cdot 10^{-2}$ & 2.91 & 1.4\\
2.00 & 1000& $1.02\cdot 10^{-1}$ & $2.94 \cdot 10^{-2}$ & 2.82 & 1.6 & $9.55\cdot 10^{-2}$ & $2.20\cdot 10^{-2}$ & 2.83 & 1.4
\end{tabular}
\end{small}
\end{center}

\begin{figure}
    \centering
    \begingroup
  \makeatletter
  \providecommand\color[2][]{\GenericError{(gnuplot) \space\space\space\@spaces}{Package color not loaded in conjunction with
      terminal option `colourtext'}{See the gnuplot documentation for explanation.}{Either use 'blacktext' in gnuplot or load the package
      color.sty in LaTeX.}\renewcommand\color[2][]{}}\providecommand\includegraphics[2][]{\GenericError{(gnuplot) \space\space\space\@spaces}{Package graphicx or graphics not loaded}{See the gnuplot documentation for explanation.}{The gnuplot epslatex terminal needs graphicx.sty or graphics.sty.}\renewcommand\includegraphics[2][]{}}\providecommand\rotatebox[2]{#2}\@ifundefined{ifGPcolor}{\newif\ifGPcolor
    \GPcolortrue
  }{}\@ifundefined{ifGPblacktext}{\newif\ifGPblacktext
    \GPblacktexttrue
  }{}\let\gplgaddtomacro\g@addto@macro
\gdef\gplbacktext{}\gdef\gplfronttext{}\makeatother
  \ifGPblacktext
\def\colorrgb#1{}\def\colorgray#1{}\else
\ifGPcolor
      \def\colorrgb#1{\color[rgb]{#1}}\def\colorgray#1{\color[gray]{#1}}\expandafter\def\csname LTw\endcsname{\color{white}}\expandafter\def\csname LTb\endcsname{\color{black}}\expandafter\def\csname LTa\endcsname{\color{black}}\expandafter\def\csname LT0\endcsname{\color[rgb]{1,0,0}}\expandafter\def\csname LT1\endcsname{\color[rgb]{0,1,0}}\expandafter\def\csname LT2\endcsname{\color[rgb]{0,0,1}}\expandafter\def\csname LT3\endcsname{\color[rgb]{1,0,1}}\expandafter\def\csname LT4\endcsname{\color[rgb]{0,1,1}}\expandafter\def\csname LT5\endcsname{\color[rgb]{1,1,0}}\expandafter\def\csname LT6\endcsname{\color[rgb]{0,0,0}}\expandafter\def\csname LT7\endcsname{\color[rgb]{1,0.3,0}}\expandafter\def\csname LT8\endcsname{\color[rgb]{0.5,0.5,0.5}}\else
\def\colorrgb#1{\color{black}}\def\colorgray#1{\color[gray]{#1}}\expandafter\def\csname LTw\endcsname{\color{white}}\expandafter\def\csname LTb\endcsname{\color{black}}\expandafter\def\csname LTa\endcsname{\color{black}}\expandafter\def\csname LT0\endcsname{\color{black}}\expandafter\def\csname LT1\endcsname{\color{black}}\expandafter\def\csname LT2\endcsname{\color{black}}\expandafter\def\csname LT3\endcsname{\color{black}}\expandafter\def\csname LT4\endcsname{\color{black}}\expandafter\def\csname LT5\endcsname{\color{black}}\expandafter\def\csname LT6\endcsname{\color{black}}\expandafter\def\csname LT7\endcsname{\color{black}}\expandafter\def\csname LT8\endcsname{\color{black}}\fi
  \fi
    \setlength{\unitlength}{0.0500bp}\ifx\gptboxheight\undefined \newlength{\gptboxheight}\newlength{\gptboxwidth}\newsavebox{\gptboxtext}\fi \setlength{\fboxrule}{0.5pt}\setlength{\fboxsep}{1pt}\definecolor{tbcol}{rgb}{1,1,1}\begin{picture}(2820.00,2540.00)\gplgaddtomacro\gplbacktext{\csname LTb\endcsname \put(560,816){\makebox(0,0)[r]{\strut{}$10^{-4}$}}\csname LTb\endcsname \put(560,1464){\makebox(0,0)[r]{\strut{}$10^{-2}$}}\csname LTb\endcsname \put(560,2111){\makebox(0,0)[r]{\strut{}$10^{0}$}}\csname LTb\endcsname \put(672,612){\makebox(0,0){\strut{}$10^{-2}$}}\csname LTb\endcsname \put(1623,612){\makebox(0,0){\strut{}$10^{-1}$}}\csname LTb\endcsname \put(2575,612){\makebox(0,0){\strut{}$10^{0}$}}\csname LTb\endcsname \put(1624,2305){\makebox(0,0){\strut{}lower frame bound}}}\gplgaddtomacro\gplfronttext{\csname LTb\endcsname \put(1623,306){\makebox(0,0){\strut{}$b-1$}}}\gplbacktext
    \put(0,0){\includegraphics[width={141.00bp},height={127.00bp}]{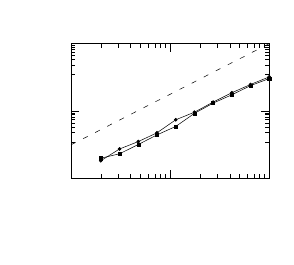}}\gplfronttext
  \end{picture}\endgroup
     \caption{25-dimensional experiment. Solid line with circles: lower frame bound $A$ for the initial nodes being the full Grid. Solid line with squares: lower frame bound $A$ for the initial nodes being drawn randomly. Dashed: $(b-1)^{3/2}$.}
    \label{fig:d25}
\end{figure}

The message of this experiment is twofold:
\begin{itemize}
\item 
    The rate of $A$ for $b\to 1$ is cubic in our theoretical results, cf.\ Lemma~\ref{l:nonweighted_bss} and Theorem~\ref{thm:unweightes_bss}.
    In this experiment we observe the rate of $3/2$ which is even smaller that the bound for the weighted \texttt{BSS} algorithm, cf.\ Theorem~\ref{BSS}.
\item
    As the number of nodes in the grid is larger than the estimated number of atoms in the observable universe, we would expect a longer runtime for this example.
    The only difference in the computational effort could originate from the iterations in the inner loop of the \texttt{BSS} algorithm.
    From our theory we obtain $M$ iterations in the worst case whereas we observe $1.5$ iterations on average in both experiments.
\end{itemize}

\paragraph{Experiment~4} Here we deal with two-dimensional hyperbolic Chui-Wang wavelets $\psi_{\bm j,\bm k}$, which are 
compactly supported and piecewise linear and $L_2([0,1]^d)$-normalized, see for instance \cite{LPU21} for the precise construction. We define the index sets
\begin{equation}\label{eq:J_n}
\mathcal{J}_n  = \{(\bm j,\bm k)\in \field{N}_{-1}^d \times \field{Z}^d : \bm j\geq -\bm 1, |\bm j|_1\leq N, \bm k \in I_{\bm j}\}
\end{equation}
and 
\begin{equation*}
I_{\bm j}=\prod\limits_{i=1}^d\begin{cases} \{0,1,\ldots 2^{j_i}-1\}&\text{ for } j_i\geq 0,\\
\{0\} &\text{ for } j_i=-1.
\end{cases}
\end{equation*}
The projection on the $\bm j$-component of this index set is displayed in the first picture in Figure \ref{fig:wavelets} with $N = 3$. Drawing sufficiently many ($M$) nodes i.i.d.\ and uniformly at random ($M = \mathcal{O}( |\mathcal{J}_n| \log( |\mathcal{J}_n|))$) it has been shown in \cite{LPU21} that the corresponding frame $([\psi_{\bm j, \bm k}(\bm x^i)]_{\bm j,\bm k})_{i=1}^{M}$ has reasonably good frame bounds (see the second picture in Figure \ref{fig:wavelets}. In the previous experiments we only dealt with equal-norm frames. This is not given anymore in this particular frame such that we are forced to apply \texttt{PlainBSS} to extract a reasonable subframe with $b \approx 1.5$.
The resulting nodes can be seen in the third picture of Figure~\ref{fig:wavelets}.

\begin{figure}
    \centering
    \begingroup
  \makeatletter
  \providecommand\color[2][]{\GenericError{(gnuplot) \space\space\space\@spaces}{Package color not loaded in conjunction with
      terminal option `colourtext'}{See the gnuplot documentation for explanation.}{Either use 'blacktext' in gnuplot or load the package
      color.sty in LaTeX.}\renewcommand\color[2][]{}}\providecommand\includegraphics[2][]{\GenericError{(gnuplot) \space\space\space\@spaces}{Package graphicx or graphics not loaded}{See the gnuplot documentation for explanation.}{The gnuplot epslatex terminal needs graphicx.sty or graphics.sty.}\renewcommand\includegraphics[2][]{}}\providecommand\rotatebox[2]{#2}\@ifundefined{ifGPcolor}{\newif\ifGPcolor
    \GPcolortrue
  }{}\@ifundefined{ifGPblacktext}{\newif\ifGPblacktext
    \GPblacktexttrue
  }{}\let\gplgaddtomacro\g@addto@macro
\gdef\gplbacktext{}\gdef\gplfronttext{}\makeatother
  \ifGPblacktext
\def\colorrgb#1{}\def\colorgray#1{}\else
\ifGPcolor
      \def\colorrgb#1{\color[rgb]{#1}}\def\colorgray#1{\color[gray]{#1}}\expandafter\def\csname LTw\endcsname{\color{white}}\expandafter\def\csname LTb\endcsname{\color{black}}\expandafter\def\csname LTa\endcsname{\color{black}}\expandafter\def\csname LT0\endcsname{\color[rgb]{1,0,0}}\expandafter\def\csname LT1\endcsname{\color[rgb]{0,1,0}}\expandafter\def\csname LT2\endcsname{\color[rgb]{0,0,1}}\expandafter\def\csname LT3\endcsname{\color[rgb]{1,0,1}}\expandafter\def\csname LT4\endcsname{\color[rgb]{0,1,1}}\expandafter\def\csname LT5\endcsname{\color[rgb]{1,1,0}}\expandafter\def\csname LT6\endcsname{\color[rgb]{0,0,0}}\expandafter\def\csname LT7\endcsname{\color[rgb]{1,0.3,0}}\expandafter\def\csname LT8\endcsname{\color[rgb]{0.5,0.5,0.5}}\else
\def\colorrgb#1{\color{black}}\def\colorgray#1{\color[gray]{#1}}\expandafter\def\csname LTw\endcsname{\color{white}}\expandafter\def\csname LTb\endcsname{\color{black}}\expandafter\def\csname LTa\endcsname{\color{black}}\expandafter\def\csname LT0\endcsname{\color{black}}\expandafter\def\csname LT1\endcsname{\color{black}}\expandafter\def\csname LT2\endcsname{\color{black}}\expandafter\def\csname LT3\endcsname{\color{black}}\expandafter\def\csname LT4\endcsname{\color{black}}\expandafter\def\csname LT5\endcsname{\color{black}}\expandafter\def\csname LT6\endcsname{\color{black}}\expandafter\def\csname LT7\endcsname{\color{black}}\expandafter\def\csname LT8\endcsname{\color{black}}\fi
  \fi
    \setlength{\unitlength}{0.0500bp}\ifx\gptboxheight\undefined \newlength{\gptboxheight}\newlength{\gptboxwidth}\newsavebox{\gptboxtext}\fi \setlength{\fboxrule}{0.5pt}\setlength{\fboxsep}{1pt}\definecolor{tbcol}{rgb}{1,1,1}\begin{picture}(8220.00,2820.00)\gplgaddtomacro\gplbacktext{\csname LTb\endcsname \put(1025,2661){\makebox(0,0){\strut{}$I$}}\csname LTb\endcsname \put(1025,648){\makebox(0,0){\strut{}$m = 192$}}}\gplgaddtomacro\gplfronttext{}\gplgaddtomacro\gplbacktext{\csname LTb\endcsname \put(3075,2661){\makebox(0,0){\strut{}initial nodes}}\csname LTb\endcsname \put(3075,648){\makebox(0,0){\strut{}$M=2400$ ($b\approx 12.5$)}}\csname LTb\endcsname \put(3075,397){\makebox(0,0){\strut{}$A=0.01708$}}\csname LTb\endcsname \put(3075,145){\makebox(0,0){\strut{}$B=1.01502$}}}\gplgaddtomacro\gplfronttext{}\gplgaddtomacro\gplbacktext{\csname LTb\endcsname \put(5125,2661){\makebox(0,0){\strut{}BSS subsampling}}\csname LTb\endcsname \put(5125,648){\makebox(0,0){\strut{}$n=288$ ($b\approx 1.5$)}}\csname LTb\endcsname \put(5125,397){\makebox(0,0){\strut{}$A=0.00321$}}\csname LTb\endcsname \put(5125,145){\makebox(0,0){\strut{}$B=1.05583$}}}\gplgaddtomacro\gplfronttext{}\gplgaddtomacro\gplbacktext{\csname LTb\endcsname \put(7175,2661){\makebox(0,0){\strut{}random subsampling}}\csname LTb\endcsname \put(7175,648){\makebox(0,0){\strut{}$n=288$ ($b\approx 1.5$)}}\csname LTb\endcsname \put(7175,397){\makebox(0,0){\strut{}$A=0.00030$}}\csname LTb\endcsname \put(7175,145){\makebox(0,0){\strut{}$B=1.13434$}}}\gplgaddtomacro\gplfronttext{}\gplbacktext
    \put(0,0){\includegraphics[width={411.00bp},height={141.00bp}]{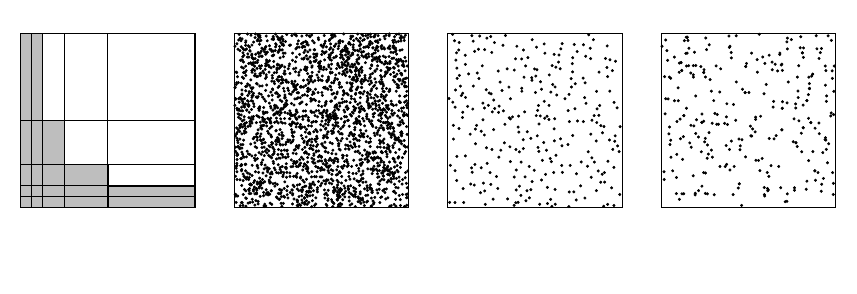}}\gplfronttext
  \end{picture}\endgroup
     \caption{Two-dimensional hyperbolic  wavelet transform}
    \label{fig:wavelets}
\end{figure}

The lower frame bound of the subsampled nodes can be estimated by Corollary~\ref{frame:smallb}: $A \le \frac{0.01708(b-1)^3}{89(b+1)^2} \approx 3.84\cdot 10^{-6}$ for $b=1.5$.
In practice we obtain a subsampled frame bound of $A = 3.21\cdot 10^{-3}$, which indicates that the theoretical constants may be improved.
Further, in the sanity check, \texttt{PlainBSS} is better by a factor of $10$ when compared to random subsampling (last picture of Figure~\ref{fig:wavelets}) and the upper frame bound does not differ much.
Overall, this experiment demands for the tricky construction of Lemma~\ref{discreteconstruction} and shows its stable applicability.

\section{Applications and discussion}\label{Discussion}

Finally, we apply the subsampling results from the previous sections
to the problem of $L_2$(-stable) recovery of multivariate complex-valued functions $f\colon D\to\mathds C$. Those are assumed to be given
on some measure space $(D,\nu)$ and the considered task shall be to recover $f$
from sampling values
\begin{align}\label{sampvalues}
\bm{f}_{n} := (f(x^1),\ldots,f(x^n)) \in\field C^n
\end{align}
taken at certain sampling nodes
\begin{align}\label{sampnodes}
\bX_n := (x^1,\ldots,x^n) \in D^n \,.
\end{align}
In order to give sense to the point evaluation $f(x^i)$ (i.e.\ to ensure that it represents a continuous functional), we model $f$ to belong either to a reproducing kernel Hilbert space (RKHS) $H(K)$ on $D$ or to $\ell_\infty(D)$, the space of bounded functions on $D$. 
Since~\eqref{sampvalues} and~\eqref{sampnodes} is
usually insufficient information for an exact reconstruction of $f$,
we merely seek to find good approximants $\tilde{f}$ of $f$. The approximation shall take place in $L_2(D,\nu)$, the space
of square-integrable complex-valued functions with
\begin{align*}
  \langle f,g\rangle_{L_2}
  = \int_D f(x)\overline{g(x)} \,\mathrm d\nu(x)
  \quad\text{and}\quad
  \|f\|_{L_2}
  = \left( \int_D |f(x)|^2 \,\mathrm d\nu(x)\right)^{1/2} \,.
\end{align*}

A possible way to recover $f$ from the given data~\eqref{sampvalues} and~\eqref{sampnodes}
is to apply a weighted least squares reconstruction operator
\begin{align}\label{WLSQ-Operator}
 S^{\bX_n}_{V_m,w_m} f  := \argmin\limits_{g\in V_m} \sum_{i=1}^{n} w_m(x^i)|g(x^i) - f(x^i)|^2 \,, 
\end{align}
for certain weights $w_m(x^i)$ and some $m$-dimensional reconstruction space $V_m \subset H(K) \subset L_2(D,\nu)$ (or, alternatively $V_m \subset \ell_\infty(D)$)\,.
Such operators perform well in many scenarios, depending on the utilized node set $\bX_n$, the space $V_m$, and the weights $w_m(x^i)$.
Of particular interest are plain least squares operators, where $w_m(x^i)=1$ for $i\in [n]$.
For those we will use the simpler notation $ S^{\bX_n}_{V_m}$.

To practically determine $S^{\bX_n}_{V_m,w_m} f$,
it is useful to employ a basis $(\eta_1,\ldots,\eta_m)$ of $V_m$. The coefficients $\bc\in\field{C}^m$
of $S^{\bX_n}_{V_m,w_m} f$ in this basis can then be obtained by solving
\begin{equation}\label{concreteWLSQ}
\bc = \argmin\limits_{\bc\in\field{C}^{m}} \| \bL_{n,m} \cdot \bc - \bm{f}_{n}\|_{\ell_{2,w}}^2 =  \argmin\limits_{\bc\in\field{C}^{m}} \| \bW_{n}(\bL_{n,m} \cdot \bc - \bm{f}_{n}) \|_{\ell_{2}}^2 \,,
\end{equation}
where $\bL_{n,m}:=(\eta_k(x^i))_{k=1\ldots m}^{i=1\ldots n}\in \field{C}^{n\times m}$ and $\bW_n:=\diag(\sqrt{w_m(x^1)}),\ldots,\sqrt{w_m(x^n)})\in\field{C}^{n\times n}$.
For this to make sense, the $\eta_k$ are assumed to be proper functions, not equivalence classes, in $L_2(D,\nu)$ (see e.g.~the explanation after~\eqref{eq:relation_lr}).

If the matrix $\widetilde{\bL}_{n,m}:=\bW_{n}\bL_{n,m}$
has full column rank, the solution of~\eqref{concreteWLSQ} is unique and can be
expressed by
$$
\bc = (\widetilde{\bL}_{n,m})^\dagger \bW_{n}\mathbf{f}_{n}  \,,
$$
where $(\widetilde{\bL}_{n,m})^\dagger$ is the Moore-Penrose pseudo-inverse of $\widetilde{\bL}_{n,m}$ and has the explicit form
\begin{align}\label{MPpseudoForm}
(\widetilde{\bL}_{n,m})^\dagger =  ((\widetilde{\bL}_{n,m})^{\ast}\widetilde{\bL}_{n,m})^{-1} (\widetilde{\bL}_{n,m})^{\ast}  \,.
\end{align}

\subsection{Sampling recovery in spaces of finite measure} \label{ssec:Appl_finite}

We first consider the case of reconstructing functions $f$ from $\ell_\infty(D)$.
We hereby assume that $L_2(D,\nu)$ is equipped with a finite measure $\nu$. In this case,
$\ell_\infty(D) \hookrightarrow L_2(D,\nu)$ and there is a constant $D_\nu>0$ such that
\begin{align}\label{LinfembedL2}
\|f\|_{L_2(D,\nu)} \le D_\nu  \|f\|_{\ell_\infty(D)} \quad\text{for all }f\in \ell_{\infty}(D) \,.
\end{align}

Our primary goal is now to derive unweighted (left) Marcinkiewicz-Zygmund inequalities for $m$-dimensional subspaces of $\ell_\infty(D)$. We let $m\in\field{N}$ and $V_m:=\spn(\eta_k)_{k=1}^m$. The spanning set $(\eta_k)_{k=1}^m$ shall
be a fixed orthonormal basis of $V_m$. Then we define a new sampling measure $d\mu$ as  $d\mu:=\varphi^{Vm}(\cdot)d\nu$ with
\begin{align}\label{densityFinite}
    \varphi^{V_m}(x) := \frac 12 + \frac 12 \frac{\sum_{k=1}^{m} |\eta_k(x)|^2}{m}
\end{align}
and draw random sampling nodes according to this measure, which is independent of the chosen orthonormal basis since $\varphi^{V_m}$ in~\eqref{densityFinite}
is unique up to $\nu$-null sets. This density first appeared in \cite{PU21}.

For a sufficiently large number of sampling nodes we have the following result.

\begin{lemma}\label{lem:aux1Finite}
	Let $p, t \in (0,1)$ and let $\widetilde{\bX}_M=(\tilde{x}^i)_{i=1}^{M}\in D^M$ be $M$ nodes drawn independently (with duplicates)
    according to the probability measure $\mu$ on $D$ given by~\eqref{densityFinite}. In case
	\begin{align}\label{cond:Maux1}
	M \ge \frac{4}{t^2}m\log\left(\frac{m}{p}\right)
	\end{align}
	it holds
	\begin{align*}
	(1-t)\|\bm a\|_2^2
	\leq \frac{1}{M}
	\| \widetilde{\bL}_{M,m} \bm a \|_2^2 \quad\text{for all ${\bm a}\in \field{C}^{m}$}
	\end{align*}
	with probability exceeding $1-p$, where
\begin{align*}
\widetilde{\bL}_{M,m} := \left(\begin{array}{llll}
{[}\eta_1/\sqrt{\varphi^{V_m}}{]}(\tilde{x}^1)& \cdots & {[}\eta_{m}/\sqrt{\varphi^{V_m}}{]}(\tilde{x}^1)\\
\qquad\vdots && \qquad\vdots\\
{[}\eta_1/\sqrt{\varphi^{V_m}}{]}(\tilde{x}^M)& \cdots & {[}\eta_{m}/\sqrt{\varphi^{V_m}}{]}(\tilde{x}^M)
\end{array}\right) \,.
\end{align*}
\end{lemma}
\begin{proof}
Let $\bu^i$, $i\in[M]$, denote the rows of $\widetilde{\bL}_{M,m}$ and define $\bA_i:=\frac{1}{M}\bu^i\otimes\bu^i$.
Then we have $\lamax(\bA_i)= \frac{\|\bu^i\|_2^2}{M} \le \frac{2m}{M}$ due to $\varphi^{V_m}\ge \frac{1}{2m} \sum_{k=1}^{m} |\eta_k|^2$ (see~\eqref{densityFinite}). 
The matrix
\[
\bH_m:= \frac{1}{M} \widetilde{\bL}_{M,m}^\ast \widetilde{\bL}_{M,m} = \frac{1}{M} \sum_{i=1}^M \bu^i\otimes\bu^i = \sum_{i=1}^M \bA_i
\]
is further Hermitian positive semi-definite and fulfills $\Ept (\bH_m) = \bI$, where $\bI$ is the identity matrix in $\field{C}^{m\times m}$. The latter follows from the orthogonality of the function system $(\eta_i/\sqrt{\varphi^{V_m}})_{i=1}^m$ in $L_2(D,\mu)$.

Lemma~\ref{matrixchernoff}, applied with $\mu_{\min}=\mu_{\max}=1$ and $R=2m/M$, now states that
\begin{align*}
\lamin(\bH_m) \le 1-t
\end{align*}
with probability not more than $m \exp(- Mt^2/(4m))$. If we choose $M$ according to~\eqref{cond:Maux1}, this then yields
\[
\frac{1}{M}\big\|\widetilde{\bL}_{M,m} \bw\big\|_2^2 = \bw^\ast \bH_m \bw \ge (1-t) \|\bw\|_2^2
\]
with probability exceeding $1-p$.
\end{proof}

In the following, let $b>1+\frac{1}{m}$ be a fixed parameter. Further, let $\widetilde{\bX}_M=(\tilde{x}^i)_{i=1}^{M}$ be a node sequence 
sampled according to Lemma~\ref{lem:aux1Finite} fulfilling $M\ge\lceil bm\rceil$.  
Applying the \texttt{plainBSS} algorithm to $\widetilde{\bX}_M$ (i.e.\ the rows of $\widetilde{\bL}_{M,m}$) with respect to $b$ yields
an index set $J\subset[M]$ with $|J|\le\lceil bm \rceil$. Selecting the corresponding nodes
in $\widetilde{\bX}_M$, we obtain a subsequence
$\bX_n=(x^i)_{i=1}^{n}\subset\widetilde{\bX}_M$ with $n\le\lceil bm \rceil$.

\begin{theorem}\label{theorem:finite} Let $(D,\nu)$ be a finite measure space and $p,t\in (0,1)$. Let further $V_m$ be an $m$-dimensional subspace of $\ell_\infty(D)$  for fixed $m\in\field{N}$.
  Let further $\widetilde{\bX}_M=(\tilde{x}^i)_{i=1}^M \in D^M$ and $\bX_n=(x^i)_{i=1}^n \in D^n$ denote the node sets constructed as above for $b>1+\frac{1}{m}$, with $M = M_{p,t}$ satisfying~\eqref{cond:Maux1} and
  $n\le\lceil bm\rceil\le M$.
  Then with probability exceeding $1-p$ for all $f\in V_m$
  \begin{align}\label{f30}
    \|f\|_{L_2(D,\nu)}^2
    \le \frac{2}{M(1-t)} \sum_{i=1}^M |f(\bm \tilde{x}^i)|^2
    \le \frac{178(b+1)^2}{(b-1)^3(1-t)} \frac {1}{m} \sum_{i=1}^n |f(x^i)|^2\,.
  \end{align}
\end{theorem} \begin{proof} Let $\ba\in\field{C}^m$ be the coefficient vector of $f$ with respect to $(\eta_i)_{i=1}^m$. The first inequality follows from Lemma~\ref{lem:aux1Finite} and the fact that $\varphi^{V_m}\ge 1/2$.
We have
\begin{align*}
\|f\|_{L_2(D,\nu)}^2 = \|\bm a\|_{2}^2 \le  \frac{\| \widetilde{\bL}_{M,m} \bm a \|_2^2}{M(1-t)}
	  = \frac{1}{M(1-t)}  \sum_{i=1}^M \frac{|f(\tilde{x}^i)|^2}{\varphi^{V_m}(\tilde{x}^i)}  \le \frac{2}{M(1-t)}  \sum_{i=1}^M |f(\tilde{x}^i)|^2 \,.
\end{align*}
An application of Corollary~\ref{frame:smallb} proves the second inequality.
\end{proof} 

From this, we can directly derive a recovery result for functions $f\in \ell_{\infty}(D)$. Earlier versions of this result can be found in \cite{CoMi16} and \cite{Temlyakov20}. The relation between the $L_2$ recovery error and the $\ell_\infty$ best approximation has been first established in \cite{CoMi16}. The main contribution here is that we prove the existence of a plain least squares recovery operator. In fact, the following theorem is a consequence of \cite[Thm.\ 2.1]{Temlyakov20} together with our Theorem \ref{theorem:finite} above. For the convenience of the reader we give a proof. 
\begin{theorem}\label{thm:sampl} Let $V_m\subset \ell_\infty(D)$ with dimension $m\in\mathds N$ and $\bX_n=(x^i)_{i=1}^n$
be as above, with $n\le\lceil bm \rceil$ and $b>1+\frac{1}{m}$,
fulfilling \eqref{f30} (with high probability).
For any $f\in \ell_\infty(D)$ the plain least squares operator $S_{V_m}^{{\bm X_n}}$ recovers $f$ in $L_2(D,\nu)$ with the following error
\begin{equation*}\|f-S_{V_m}^{{\bm X_n}}f\|_{L_2(D,\nu)}^2 \leq C_\nu\frac{b^3}{(b-1)^{3}}e(f,V_m)_{\ell_{\infty}(D)}^2  \,,
\end{equation*}
with a constant $C_\nu>0$ that only depends on $\nu$, where
$$
    e(f,V_m)_{\ell_{\infty}(D)} := \inf\limits_{g\in V_m} \|f-g\|_{\ell_{\infty}(D)}\,.
$$
\end{theorem}

\begin{proof} By the triangle inequality, we obtain for any $g \in V_m$
$$
    \|f-S_{V_m}^{{\bm X_n}}f\|^2_{L_2(D,\nu)} = \|f-g\|^2_{L_2(D,\nu)}+\|g-S_{V_m}^{{\bm X_n}}f\|^2_{L_2(D,\nu)} \,.
$$
The fact that $\nu$ is a finite measure gives $\|f-g\|_{L_2(D,\nu)} \leq D_{\nu} \|f-g\|_{\ell_\infty(D)}$ (see~\eqref{LinfembedL2}). The second summand equals
$\|S_{V_m}^{{\bm X_n}}(g-f)\|_{L_2(D,\nu)}$, which can be estimated by Theorem~\ref{theorem:finite}, namely
\begin{equation*}
  \begin{split}
    \|S_{V_m}^{{\bm X_n}}&(g-f)\|^2_{L_2(D,\nu)} \leq \frac{Cb^2}{(b-1)^3(1-t)} \frac{1}{m}  \sum\limits_{i=1}^{n} |(S_{V_m}^{{\bm X_n}}(f-g))({x}^i)|^2\\
    &\leq \frac{2Cb^2}{(b-1)^3(1-t)} \frac{1}{m} \sum\limits_{i=1}^{n} |(S_{V_m}^{{\bm X_n}}(f-g))({x}^i) - (f-g)({x}^i)|^2  +  |(f-g)({x}^i)|^2             \\
    &\leq \frac{4Cb^2}{(b-1)^3(1-t)} \frac{1}{m} \sum\limits_{i=1}^{n} |(f-g)({x}^i)|^2
    \leq \frac{4\tilde{C}b^3}{(b-1)^3(1-t)}\|f-g\|^2_{\ell_{\infty}}\,.
  \end{split}
\end{equation*}
From the second to the third line, we hereby used
\begin{align*}
\sum\limits_{i=1}^{n} |(S_{V_m}^{{\bm X_n}}(f-g))({x}^i) - (f-g)({x}^i)|^2   \le \sum\limits_{i=1}^{n} |(f-g)({x}^i)|^2 \,.
\end{align*}
Choosing $g \in V_m$ such that $\|f-g\|_{\ell_{\infty}} \leq 2e(f,V_m)_{\ell_{\infty}}$ yields the result.
\end{proof}

At last, let us define the following quantity for a function class $F\subset \ell_\infty(D)$,
\begin{equation} \label{sampling_numbers}
    	g_{n,m}^{\mathrm{ls}}(F,L_2(D,\nu)) := \inf\limits_{\substack{V_m \subset \ell_\infty(D)\\
    	\mathrm{dim}V_m = m}} \inf \limits_{{\bm X}_n = ({x}^1,\cdots,{x}^n) \in D^n} \sup\limits_{f \in F}
    	\|f-S_{V_m}^{{\bm X}_n}f\|_{L_2(D,\nu)}\,.
\end{equation}
It measures the error of an optimal plain least squares algorithm of the above type,
using $n$ nodes and an $m$-dimensional reconstruction space. Our last result of this subsection, Corollary~\ref{cor:lastresult1},
compares this quantity with the Kolmogorov number
\begin{align}\label{f100}
d_m(F,\ell_{\infty}(D)) := \inf_{\substack{W\subset \ell_{\infty}(D) \\ \dim(W)=m}} \sup_{f\in F} \min_{w\in W} \|f - w\|_{\ell_{\infty}(D)}
\end{align}
of the class $F$. It is a direct consequence of Theorem \ref{thm:sampl}. The constant $C_{\nu}>0$ 
only depends on $\nu$.

\begin{corollary}
\label{cor:lastresult1}
Let $F$ be a class of functions in $\ell_\infty(D)$. Then for $m\in\field{N}$ and $b>1+\frac{1}{m}$
$$
    g_{\lceil bm \rceil,m}^{\mathrm{ls}}(F,L_2(D,\nu)) \leq C_\nu\frac{b^{3/2}}{(b-1)^{3/2}}d_m(F,\ell_{\infty}(D))\,.
$$
\end{corollary}

This estimate improves on a recent result by Temlyakov~\cite{Temlyakov20}, where the quantity $g_{\lceil bm \rceil,m}^{\mathrm{ls}}(F,L_2(D,\nu))$ is related to a modified Kolmogorov width which includes an additional restriction on the admissible subspaces in \eqref{f100}. 

\subsection{Sampling recovery in reproducing kernel Hilbert spaces}
\label{ssec:Appl_infinite}
We next consider functions from a RKHS $H(K)$ with a finite trace kernel $K$. The RKHS is
assumed to be compactly embedded into $L_2(D,\nu)$ via a Hilbert-Schmidt embedding $\Id_{K,\nu}$. The measure $\nu$ does not have to be a finite measure as in Subsection~\ref{ssec:Appl_finite}. In this setting, the number of non-zero singular values of $\Id_{K,\nu}$ is countable and, under the additional assumption that the subspace $\Id_{K,\nu}(H(K))$ of $L_2(D,\nu)$ is infinite-dimensional, also infinite. We thus have
a sequence $(\sigma_k)_{k=1}^{\infty}$ of strictly positive singular numbers, which we order in descending order.
The associated left and right singular functions shall be denoted by $(\eta_k)_{k=1}^{\infty}$ and $(e_k)_{k=1}^{\infty}$ (see \ref{ssec:A_RKHS} for more details).
We follow the course of \cite{DKU22,KrUl19,KaVoUl21,MoUl20,NaSchUl20}, where this setting was considered as well.

A natural reconstruction space in this scenario is $V_{m}:=\spn(\eta_k)_{k=1}^{m}$ spanned by
the first $m$ left singular functions associated to the $m$ largest singular numbers of $\Id_{K,\nu}$.
Appropriate nodes and weights for $S^{\bX_n}_{V_{m},w_{m}}$ are constructed in a two-step procedure,
similar to the node generation in Subsection~\ref{ssec:Appl_finite}.
The initial node set $\widetilde{\bX}_M$ is drawn according to a probability measure $d\varrho_m:=\varrho_m(\cdot)d\nu$ with
\begin{equation}\label{density}
\varrho_m(x) := \frac{1}{2} \bigg(\frac{1}{m} \sum_{k = 1}^{m} |\eta_k(x)|^2 + \frac{K(x,x) - \sum_{k = 1}^{m}|e_k(x)|^2}{\int_D K(x,x) d\nu(x) - \sum_{k = 1}^{m}\sigma^2_k} \bigg) 
\end{equation}
as density function. The corresponding weight function is
\begin{align*}
w_m:D\to[0,\infty) \quad,\quad w_m(x) := \begin{cases} \varrho_m(x)^{-1/2} \quad &,\,\varrho_m(x)\neq0 \,, \\
0  &,\,\varrho_m(x)=0\,.\end{cases}
\end{align*}
The set $\bX_n$ is then again obtained from $\widetilde{\bX}_M$ by \texttt{PlainBSS}. Note
that~\eqref{density} is well-defined for all $m\in\field{N}$ due to the positivity of the 
singular numbers.

\paragraph{Generation of sampling nodes} 

\paragraph{Step 1 (Initial nodes)}
Let $m\in\field{N}$ and $b>1+\frac{1}{m}$ and fix parameters $p,t\in(0,1)$. Then, with
\begin{align}\label{choiceMstep1}
M= \max\left\{ \left\lceil \frac{4}{t^2}m\log\Big( \frac{m}{p} \Big) \right\rceil , \lceil bm \rceil \right\}  \,,
\end{align}
an initial random sampling set
\begin{align}\label{sampnodes1}
\widetilde{\bX}_M := (\tilde{x}^1,\ldots,\tilde{x}^M) \in D^M
\end{align}
is drawn, each node independently according to the measure $d\varrho_{m}$ with density~\eqref{density}.

For the associated weights almost surely $w_{m}(\tilde{x}^i)>0$.
Furthermore, with probability exceeding $1-p$ the rows of the matrix $\frac{1}{\sqrt{M}}\widetilde{\bL}_{M,m}$ where (cf.~\cite{NaSchUl20} replacing $n$ with $M$)
\begin{align*}\widetilde{\bL}_{M,m} = \left(\begin{array}{llll}
{[}w_{m} \eta_1{]}(\tilde{x}^1)& \cdots & {[}w_{m} \eta_{m}{]}(\tilde{x}^1)\\
\qquad\vdots && \qquad\vdots\\
{[}w_{m} \eta_1{]}(\tilde{x}^M)& \cdots & {[}w_{m} \eta_{m}{]}(\tilde{x}^M)
\end{array}\right) \,,
\end{align*}
represent a finite frame with lower frame bound $(1-t)$.
This can be formulated as
\begin{align}\label{fraspec}
(1-t)\|\ba\|_2^2 \leq \frac{1}{M}\big\|\widetilde{\bL}_{M,m} \ba\big\|_2^2
\quad\text{for all}\quad \ba \in \field C^{m}
\end{align}
and follows from Lemma~\ref{lem:aux1} below.

\begin{lemma}\label{lem:aux1}
	Let $p, t \in (0,1)$ and let $\widetilde{\bX}_M=(\tilde{x}^1,\ldots,\tilde{x}^M)\in D^M$ be $M$ nodes drawn independently (with duplicates)
    according to the measure $d\varrho_m$ given by~\eqref{density}. In case
	\begin{align*}
	M \ge \frac{4}{t^2}m\log\left(\frac{m}{p}\right)
	\end{align*}
	\eqref{fraspec} holds
	with probability exceeding $1-p$.
\end{lemma}
\begin{proof}
Based on Lemma~\ref{matrixchernoff}, analogous to the proof of Lemma~\ref{lem:aux1Finite}.
\end{proof}

For any $\tilde{p}\in(0,1)$ we also have
\begin{align}\label{second_est}
\Big\|\frac{1}{M}(\PhiYt)^{\ast}\PhiYt\Big\|_{2\to 2}  \leq  2\sigma_{m+1}^2 + \frac{42}{M}   \log\Big(2^\frac{3}{4}\frac{M}{\tilde{p}}\Big) \sum_{j=m+1}^{\infty} \sigma_j^2
\end{align}
with probability exceeding $1-\tilde{p}$ for the infinite matrix  $\widetilde{\mathbf{\Phi}}_{M,m}$ given by
$$
\widetilde{\mathbf{\Phi}}_{M,m}:=\left(\begin{array}{lll}
[w_m e_{m+1}](\tilde{x}^1) & [w_m e_{m+2}](\tilde{x}^1) & \ldots \\
\qquad\vdots & \qquad\vdots &  \\
{[}w_m e_{m+1}{]}(\tilde{x}^M) & {[}w_m e_{m+2}{]}(\tilde{x}^M) & \ldots
\end{array}\right) \,.
$$
This is a consequence of the following lemma, itself a corollary
of~\cite[Prop.~3.8]{MoUl20} (Proposition~\ref{C:prop:main1} in the Appendix). 

\begin{lemma}\label{lem:aux2}
	Let $\bu^i$, $i\in[M]$, be i.i.d.\ random sequences from \( \ell_2(\field{N})\) with $M\in\field{N}_{\ge 3}$.
	Let further $R>0$ such that \(\| \bu^i \|_2 \leq R\) almost surely and
	$\field E(\bu^i \otimes  \bu^i)  =  \mathbf{\Lambda}$ for each $i\in[M]$.
	 Then for each $\tilde{p}\in(0,1)$, with probability exceeding $1-\tilde{p}$,
\[
\Big\| \frac{1}{M} \sum_{i=1}^M  \, \bu^i \otimes  \bu^i \Big\|_{2 \to 2}  
\leq 2 \|\mathbf{\Lambda}\|_{2 \to 2}  + \frac{21 R^2}{M}  \log\Big( \frac{2^\frac{3}{4}M}{\tilde{p}}\Big) \,.
\]
\end{lemma}
\begin{proof}
In~\cite[Prop.~3.8]{MoUl20} (Proposition~\ref{C:prop:main1}) we can choose $r>1$ such that $\tilde{p}=2^\frac{3}{4}M^{1-r}$. Then $\log(2^\frac{3}{4}M/\tilde{p})=r\log(M)$ and, noting $8\kappa^2\le 21$, we obtain
\[
\Big\| \frac{1}{M} \sum_{i=1}^M  \, \bu^i \otimes  \bu^i - \mathbf{\Lambda} \Big\|_{2 \to 2}  \leq \|\mathbf{\Lambda}\|_{2 \to 2}  + \frac{21 R^2}{M}  \log\Big( \frac{2^\frac{3}{4}M}{\tilde{p}}\Big)
\]
with probability exceeding $1-\tilde{p}$. The triangle inequality finally yields the result.
\end{proof}

To prove~\eqref{second_est} with Lemma~\ref{lem:aux2} the rows of $\PhiYt$ are interpreted as sequences $\bu^i\in\ell_2(\field{N})$, where $\|\bu^i\|_2^2 \le 2 \sum_{k\ge m+1} \sigma_k^2$ for each $i\in[M]$. Those then satisfy 
\[
\field E(\bu^i \otimes  \bu^i)  = \diag(\sigma_{m+1}^2,\sigma_{m+2}^2,\ldots)=: \mathbf{\Lambda}_m \,,
\]
with $\|\mathbf{\Lambda}_m \|_{2\to 2}=\sigma_{m+1}^2$,
and applying Lemma~\ref{lem:aux2} with $R^2=2 \sum_{k\ge m+1} \sigma_k^2$  yields~\eqref{second_est}.

\paragraph{Step 2 (Subsampling)} The \texttt{PlainBSS} algorithm is used to determine a set of indices $J\subset[M]$ and a
subset of nodes $\bX_n\subset \widetilde{\bX}_M$ of cardinality $|\bX_n|=|J|\le\lceil bm \rceil$, where $b>1+\frac{1}{m}$ as in Step~1.
We then build a submatrix $\widetilde{\bL}_{J,m}$ of $\widetilde{\bL}_{M,m}$ by selecting the corresponding rows of $\widetilde{\bL}_{M,m}$.
From~\eqref{fraspec} and Corollary~\ref{frame:smallb} we get, with probability exceeding $1-p$, that
\begin{align*}
(1-t)\|\ba\|_2^2
\le \frac{1}{M}
	\| \widetilde{\bL}_{M,m} \ba \|_2^2   \le  \frac{89 (b+1)^2}{(b-1)^3}
  \frac 1m  \|\widetilde{\bL}_{J,m} \ba\|_2^2 \quad\text{for all }\ba\in\field{C}^m \,.
\end{align*}
In particular, $\widetilde{\bL}_{J,m}$ then has full rank and the norm of the Moore-Penrose pseudo-inverse $(\widetilde{\bL}_{J,m})^\dagger=((\widetilde{\bL}_{J,m})^\ast\widetilde{\bL}_{J,m})^{-1}(\widetilde{\bL}_{J,m})^\ast$
(see~\eqref{MPpseudoForm}) fulfills the estimate
\begin{align}\label{opnormest}
	\|(\widetilde{\bL}_{J,m})^\dagger\|^2_{2\to 2} \le   \frac{89 (b+1)^2}{(b-1)^3} \frac{1}{1-t} \frac 1m \,.
\end{align}

\paragraph{Performance analysis} The sampling reconstruction operator $S^{\bX_n}_{V_{m},w_{m}}$ defined in~\eqref{WLSQ-Operator}, with nodes $\bX_n$ constructed according to the previous paragraph, yields a near-optimal reconstruction performance (cf.~\cite{KrUl19,KaVoUl21,MoUl20,NaSchUl20}), with the advantage of a precise control of the oversampling factor $b$
as well as a polynomial-time semi-constructive node generation procedure. A modified subsampling procedure was presented recently by Dolbeault, Krieg, and M.~Ullrich~\cite{DKU22}, leading to an optimal reconstruction rate (without the log-term). It is a refinement of the
non-constructive Weaver subsampling from~\cite{NaSchUl20} and cannot be made constructive by our approach, since for that the upper frame bounds need to be preserved as well.

\begin{theorem}\label{thm:approx} Let $H(K)$ be a reproducing kernel Hilbert space on a measurable domain $(D,\nu)$ with a positive semi-definite $\nu$-measurable kernel
	$K:D\times D\to\field{C}$ of finite trace
	$$
	\int_D K(x,x) d\nu(x) < \infty \,.
	$$
	Further, assume that $H(K)$ is separable and that the
	canonical Hilbert-Schmidt embedding 
	\[
	\Id_{K,\nu}:H(K) \to L_2(D,\nu)
	\]
	has infinite rank. In this setting, fix $m\in\field{N}_{\ge10}$,
	$p\in(0,\frac{2}{3})$, $t=\frac{2}{3}$, and construct a node 
	set
	\[
	\bX_n\subset D \quad\text{with}\quad |\bX_n|\le \lceil bm\rceil 
	\]
	for $b>1+\frac{1}{m}$
	according to the procedure described in the previous paragraph (applying \texttt{PlainBSS} on $M\ge\max\{\frac{4}{t^2}m\log(\frac m p),\lceil bm\rceil\}$ randomly drawn nodes).
    The reconstruction $S^{\bX_n}_{V_m,w_m}$ given by~\eqref{WLSQ-Operator} then performs as
    \begin{equation}\label{samplingnum}
	\sup\limits_{\|f\|_{H(K)}\le 1} \|f-S^{\bX_n}_{V_m,w_m}f\|^2_{L_2(D,\nu)} \leq 
	\frac{4827}{\min\{b-1,1\}} \frac{ (b+1)^2}{ (b-1)^2} \log\Big(\frac{m}{p}\Big)  \bigg(\sigma_{m+1}^2 + \frac{7}{m}\sum\limits_{k=m+1}^{\infty} \sigma_k^2 \bigg)
	\end{equation}
	with probability exceeding $1-\frac{3}{2}p$, where $(\sigma_k)_{k=1}^{\infty}$ is the sequence of decreasingly ordered, square-summable, non-trivial singular numbers of $\Id_{K,\nu}$.
\end{theorem}
\begin{proof}
	First recall that, with certain probabilities,
	the initial nodes 
	associated to $\bX_n$,
	\begin{align}\label{OrigNodes}
	\widetilde{\bX}_M=(\tilde{x}^1,\ldots,\tilde{x}^M) \,,
	\end{align}
    fulfill property~\eqref{fraspec} for the matrix $\widetilde{\bL}_{M,m}$ 
	and~\eqref{second_est} for $\PhiYt$.
To prove~\eqref{samplingnum}, we now take $f\in H(K)$ with $\|f\|_{H(K)}\le 1$ and denote with
	$P_{m}:L_2(D,\nu)\to V_{m}$ the orthogonal projection onto the reconstruction space $V_m$ of $S^{\bX_n}_{V_{m},w_{m}}$.
	Since the operator $S^{\bX_n}_{V_{m},w_{m}}$
	acts as identity on $V_{m}$, whence $S^{\bX_n}_{V_{m},w_{m}} P_{m}f=P_{m}f$,
	we have
	\begin{align}\label{triangle}
	\|f-S^{\bX_n}_{V_{m},w_{m}} f\|^2_{L_2}
	&= \|f-P_{m}f\|^2_{L_2} + \|S^{\bX_n}_{V_{m},w_{m}}(f-P_{m}f)\|^2_{L_2} \notag \\
    &\le \sigma^2_{m+1} + \|(\widetilde{\bL}_{J,m})^\dagger\|^2_{2\to 2} \sum\limits_{i=1}^{n} w_{m}(x^i)^2|f(x^i)-P_{m}f(x^i)|^2  \,.
    \end{align}

	Clearly, we also have
	\begin{align}\label{DeltaEst1}	
	\sum\limits_{i=1}^{n}w_{m}(x^i)^2|f(x^i)-P_{m}f(x^i)|^2 \le  \sum\limits_{i=1}^M w_{m}(\tilde{x}^i)^2|f(\tilde{x}^i) - P_{m}f(\tilde{x}^i)|^2 \,,
	\end{align}
	where $(\tilde{x}^i)_{i=1}^M$ are the initial nodes~\eqref{OrigNodes}.
	For $f\in \cN(\Id_{K,\nu})$, the null-space of $\Id_{K,\nu}$, the right-hand side of
	\eqref{DeltaEst1} vanishes almost surely, due to the separability of $H(K)$.
	For general $f$, following the proof in \cite[Thm.\ 5.1]{MoUl20},
	almost surely
	\begin{align}\label{DeltaEst2}	
	\sum\limits_{i=1}^M w_{m}(\tilde{x}^i)^2|f(\tilde{x}^i)-P_{m}f(\tilde{x}^i)|^2 \leq \|(\PhiYt)^{\ast}\PhiYt\|_{2\to 2}
	\end{align}
	and according to~\eqref{second_est}
	\begin{align}\label{DeltaEst3}	
	 \Big\|(\PhiYt)^{\ast}\PhiYt\Big\|_{2\to 2}  \leq  2M\sigma_{m+1}^2 + 42   \log\Big(2^\frac{3}{4}\frac{M}{\tilde{p}}\Big) \sum_{k=m+1}^{\infty} \sigma_k^2  \,.
	\end{align}

	Altogether, the estimates~\eqref{triangle}-\eqref{DeltaEst3}
    together with the norm estimate~\eqref{opnormest} for $(\widetilde{\bL}_{J,m})^\dagger$
    yield
    \begin{align}\label{eq:nexttolast}
	\|f-S^{\bX_n}_{V_{m},w_{m}} f\|^2_{L_2}
	\le  \sigma_{m+1}^2 + \frac{89 (b+1)^2}{(b-1)^3} \frac{1}{1-t} \frac{1}{m} \Big( 2M\sigma_{m+1}^2 +   42\log\Big(2^\frac{3}{4}\frac{M}{\tilde{p}}\Big)\sum\limits_{k=m+1}^\infty \sigma_k^2 \Big) \,.
	\end{align}

    Let us proceed to bound the second summand in~\eqref{eq:nexttolast}, which can be rewritten as
     \begin{align}\label{eq:nexttolast2}
    \frac{(b+1)^2}{(b-1)^3} \frac{178}{t^2(1-t)} \Big( \frac{Mt^2}{m}\sigma_{m+1}^2 +  \frac{21t^2}{m} \log\Big(2^\frac{3}{4}\frac{M}{\tilde{p}}\Big)\sum\limits_{k=m+1}^\infty \sigma_k^2 \Big) \,.
    \end{align}

    Plugging in \eqref{choiceMstep1} for $M$ with, at the moment, still arbitrary $p,t\in(0,1)$, we get
    \begin{align}\label{eq:est:1}
    \frac{M-1}{m}\le \max\Big\{\frac{4}{t^2} \log\Big( \frac{m}{p} \Big) , b \Big\} \quad\text{and thus}\quad \frac{Mt^2}{m} \le  \frac{4M}{M-1} \max\Big\{ \log\Big( \frac{m}{p} \Big), \frac{bt^2}{4} \Big\} \,.
    \end{align}

    Using $\log(\frac{m}{p})\le \frac{m}{\mathrm{e}p}$ as well as $\log(xy)\le x\log(y)$ for $x\ge1$ and $y\ge\mathrm{e}$, we also have
    \begin{align}\label{eq:est:2}
    \log\Big(2^\frac{3}{4}\frac{M}{\tilde{p}}\Big) \le \log\Big(2^\frac{3}{4} \frac{Mm}{\tilde{p}(M-1)}  \max\Big\{ \log\Big( \frac{m}{p} \Big), \frac{bt^2}{4} \Big\} \frac{4}{t^2} \Big)
    \le  \frac{M}{M-1} \log\Big(2^\frac{3}{4} \frac{4m}{\tilde{p}t^2} \max\Big\{ \frac{m}{\mathrm{e}p}, \frac{bt^2}{4} \Big\}  \Big)  \,.
    \end{align}

     For $t=\frac{2}{3}$ the denominator $t^2(1-t)$ in~\eqref{eq:nexttolast2} becomes maximal. We get $\frac{178}{t^2(1-t)}=1201.5$.
    Together with~\eqref{eq:est:1} we obtain
    \[
    \frac{(b+1)^2}{(b-1)^3} \frac{178}{t^2(1-t)} \frac{Mt^2}{m} \sigma_{m+1}^2 \le \frac{4806 (b+1)^2}{(b-1)^3} \frac{M}{M-1}  \max\Big\{ \log\Big( \frac{m}{p} \Big), \frac{b}{9} \Big\}  \sigma_{m+1}^2 \]
    for the first part in~\eqref{eq:nexttolast2}.
For the second part, we have $21 t^2=\frac{28}{3}$ and
    due to~\eqref{eq:est:2}
    \begin{align*}
    21 t^2 \log\Big(2^\frac{3}{4}\frac{M}{\tilde{p}}\Big) \le \frac{56}{3} \frac{M}{M-1}  \log\Big( 2^\frac{3}{8}  \sqrt{\frac{9}{\mathrm{e}}} \sqrt{\frac{m}{\tilde{p}}} \max\Big\{ \sqrt{\frac{m}{p}}, \sqrt{\frac{\mathrm{e} b}{9}} \Big\} \Big) \,.
    \end{align*}
    Putting $C:= 2^{\frac{3}{8}}  \sqrt{\frac{9}{\mathrm{e}}}\approx 2.359...$ , this yields
    \begin{align*}
	 \frac{(b+1)^2}{(b-1)^3} \frac{178}{t^2(1-t)}    \log\Big(2^\frac{3}{4}\frac{M}{\tilde{p}}\Big) 21 t^2 
     \le \frac{4806 (b+1)^2}{(b-1)^3} \frac{M}{M-1}  \log\Big( C  \sqrt{\frac{m}{\tilde{p}}} \max\Big\{ \sqrt{\frac{m}{p}} , \sqrt{\frac{\mathrm{e}b}{9}} \Big\} \Big) \frac{14}{3}   \,.
	\end{align*}

    Altogether, we can thus bound~\eqref{eq:nexttolast2} by
    \begin{align}\label{eq:nexttolastnew}
    \frac{4806(b+1)^2}{(b-1)^3} \frac{M}{M-1}  \bigg( \max\Big\{ \log\Big( \frac{m}{p} \Big), \frac{b}{9} \Big\} \sigma_{m+1}^2 +   \log\Big( C  \sqrt{\frac{m}{\tilde{p}}} \max\Big\{ \sqrt{\frac{m}{p}} , \sqrt{\frac{\mathrm{e}b}{9}} \Big\}\Big) \frac{14}{3m} \sum\limits_{k=m+1}^\infty \sigma_k^2 \bigg) \,.
    \end{align}

 We now choose $p=2\tilde{p}\le\frac{2}{3}$.
 For this choice $M\ge \lceil 9 \cdot 10 \cdot\log(\frac{3\cdot 10}{2})\rceil=244$ is always fulfilled,
 taking into account $t=\frac{2}{3}$, and thus $\frac{M}{M-1} \le  \frac{244}{243}$. We hence arrive at
 \begin{align*}
    \frac{4826(b+1)^2}{(b-1)^3} \bigg( \max\Big\{ \log\Big( \frac{m}{p} \Big), \frac{b}{9} \Big\} \sigma_{m+1}^2 +   \log\Big( C\sqrt{2} \sqrt{\frac{m}{p}} \max\Big\{ \sqrt{\frac{m}{p}} , \sqrt{\frac{\mathrm{e}b}{9}} \Big\}\Big) \frac{14}{3m} \sum\limits_{k=m+1}^\infty \sigma_k^2 \bigg) \,.
    \end{align*}
 Further,
 \[
   \sqrt{\frac{m}{p}} \max\Big\{ \sqrt{\frac{m}{p}} , \sqrt{\frac{\mathrm{e}b}{9}} \Big\} \le  \max\Big\{ \frac{m}{p} , \frac{\mathrm{e}b}{9} \Big\} \le  \max\Big\{ \frac{m}{p} , \frac{b}{3} \Big\} \,.
 \]
 
 Let us also estimate
 \begin{align*}
 \log\Big(C\sqrt{2} \max\Big\{\frac{m}{p}, \frac{b}{3} \Big\}\Big) = \Big( \frac{\log(C\sqrt{2})}{\log( \max\{m/p, b/3 \})} + 1 \Big)  \log\Big( \max\Big\{\frac{m}{p}, \frac{b}{3} \Big\}\Big) \,,
 \end{align*}
 where, in view of $m\in\field{N}_{\ge 10}$ and $p\le\frac{2}{3}$,
 \begin{align*}
 \frac{\log(C\sqrt{2})}{\log( \max\{m/p, b/3 \})} \le \frac{\log(C\sqrt{2})}{\log(m/p)} \le \frac{\log(\sqrt{2}2^{\frac{3}{8}}  \sqrt{\frac{9}{\mathrm{e}}})}{\log(15)} =: F
 \approx 0.445... \,.
 \end{align*}

When we plug this into~\eqref{eq:nexttolastnew}, we finally arrive at the bound
\begin{align*}
\frac{4826 (b+1)^2}{(b-1)^3}  \Big( \max\Big\{  \log\Big( \frac{m}{p}  \Big) , \frac{b}{9} \Big\} \sigma_{m+1}^2 +  \log\Big( \max\Big\{   \frac{m}{p}  , \frac{b}{3} \Big\}  \Big)    \frac{14(F+1)}{3m} \sum\limits_{k=m+1}^\infty \sigma_k^2 \Big)
\end{align*}
for~\eqref{eq:nexttolast2}, with $\frac{14}{3} (F+1) \approx 6.744...< 7$.

This proves~\eqref{samplingnum} with probability at least $1-p-\tilde{p}=1-\frac{3}{2}p$.
\end{proof}

\begin{remark}
In the proof of Theorem~\ref{thm:approx} we chose $p=2\tilde{p}\le\frac{2}{3}$.
Of course, other choices are possible here.
We would like to detail one particular scenario.
First, recall Step~1 of the node generation process of $\bX_n$.
Here, condition~\eqref{fraspec} of $\widetilde{\bL}_{M,m}$ could in fact be deterministically checked after the probabilistic generation of the initial nodes $\widetilde{\bX}_M$ in~\eqref{sampnodes1}.
Redrawing these nodes until this condition is fulfilled, which happens with high probability in polynomial time, we thus generate $\widetilde{\bX}_M$ which satisfies~\eqref{fraspec} for $t=\frac{2}{3}$ and where $M$ is as in~\eqref{choiceMstep1} for $p\sim 1$.
In practice, we could take for example (cf.~\eqref{choiceMstep1})
\[
M= \max\{ \left\lceil 9 m \log(m) \right\rceil + 1 , \lceil bm \rceil \}  \,.
\]
For this $M$ the success probability of~\eqref{fraspec} at each draw is strictly positive. Utilizing the reconstruction operator $S^{\bX_n}_{V_m,w_m}$
with a \texttt{BSS}-downsampled node set $\bX_n$ derived from such a (repeatedly redrawn) set of nodes $\widetilde{\bX}_M$ then yields
\begin{equation*}\sup\limits_{\|f\|_{H(K)}\le 1} \|f-S^{\bX_n}_{V_m,w_m}f\|^2_{L_2(D,\nu)}
	\leq \frac{4831}{\min\{b-1,1\}} \frac{(b+1)^2}{(b-1)^2} \bigg(  \log(m)  \sigma_{m+1}^2 + \log\Big( \frac{m}{\sqrt{\tilde{p}}} \Big) \frac{7}{m}\sum\limits_{k=m+1}^{\infty} \sigma_k^2 \bigg)
	\end{equation*}
with probability exceeding $1-\tilde{p}$ for each $\tilde{p}\in(0,1)$.
The proof is analogous to the proof of Theorem~\ref{thm:approx}. 
\end{remark}

\section*{Acknowledgement}
The authors would like to thank Moritz Moeller, who helped to implement the
\texttt{BSS} algorithm and its modifications. They would also like to thank Daniel Potts, Vladimir N. Temlyakov, and Andr\'e Uschmajew for fruitful discussions and David Krieg, Stefan Kunis, and Mario Ullrich for their comments and questions during the (online) school/conference `Sampling Recovery and Related Problems' in May 2021. They are also very thankful for the valuable comments by Matthieu Dolbeault, which, in particular, resulted in an improved version of Lemma~\ref{discreteconstruction}.
Next to that, Felix Bartel would like to thank the Deutscher Akademischer Austauschdienst (DAAD) for funding his research scholarship.

\bibliographystyle{abbrv}

\newpage
\appendix

\section{Matrix theory}
\label{ssec:A_Matrix}

A basic tool in Subsection~\ref{ssec:PotBar} is the matrix determinant lemma (cf.~\cite[Lem.~2.2]{BaSpSr09}). 
The complex version reads as follows.

\begin{manuallemma}{A.1}[Matrix determinant lemma]\label{APPlem:MatDet}
If $\bA\in\field{C}^{m\times m}$ is nonsingular and $\bv\in\field{C}^m$ is a vector, then
\[
\det(\bA + \bv\bv^\ast\big) = \det(\bA)( 1 + \bv^\ast \bA^{-1}\bv) \,.
\]
\end{manuallemma}

\noindent
Lemma~\ref{APPlem:MatDet} is a direct consequence of the Sherman-Morrison formula (see e.g.~\cite{Harville97})
\[
\big(\bA + \bv\bv^\ast\big)^{-1} = \bA^{-1} - \frac{\bA^{-1}\bv\bv^\ast \bA^{-1}}{1+ \bv^\ast \bA^{-1}\bv} \,,
\]
which holds under the same assumptions as in Lemma~\ref{APPlem:MatDet}.

Another basic result, needed in the proof of Lemma~\ref{lem3.5}, is the following statement.
\begin{manuallemma}{A.2}\label{lem:auxBSS}
Let $(\bm y^i)_{i=1}^M \subset \field{C}^{m}$ be a frame with frame bounds $0<A\le B<\infty$.
Let further $\bM\in\field{C}^{m\times m}$ be a positive semi-definite Hermitian matrix.
Then
\begin{align*}
A\tr(\bM) \le  \sum_{i=1}^{M} (\by^i)^\ast \bM \by^i  \le B\tr(\bM) \,.
\end{align*}
\end{manuallemma}
\begin{proof}
First, observe that for an arbitrary matrix $\bM\in\field{C}^{m\times m}$
\begin{align*}
\sum_{i=1}^{M} (\by^i)^\ast \bM \by^i &= \sum_{i=1}^{M} \tr\Big( (\by^i)^\ast \bM \by^i \Big) = \sum_{i=1}^{M} \tr\Big( \bM \by^i(\by^i)^\ast \Big) = \tr\Big( \bM \Big(\sum_{i=1}^{M} \by^i(\by^i)^\ast\Big) \Big) \,.
\end{align*}
Since $\bY:=\sum_{i=1}^{M} \by^i(\by^i)^\ast$ is positive-definite Hermitian with $\sigma(\bY)\subset[A,B]$, there exists a unitary matrix $\bU$ such that
$\bD := \bU\bY\bU^{-1}$
is diagonal with entries in the range $[A,B]$. We can hence conclude
\begin{align*}
\sum_{i=1}^{M} (\by^i)^\ast \bM \by^i = \tr\Big( \bM\bY \Big) = \tr\Big( \bU\bM\bY\bU^{-1} \Big) = \tr\Big( \bU\bM\bU^{-1}\bD \Big) \,.
\end{align*}
Under the assumption that $\bM$ is positive semi-definite Hermitian
the transformation $\bU\bM\bU^{-1}$ is also positive semi-definite Hermitian. In particular, its diagonal entries are all nonnegative real numbers.
As a consequence, we obtain the assertion
\begin{align*}
A\tr\Big(\bM\Big) = A\tr\Big(\bU\bM\bU^{-1}\Big) \le \tr\Big( \bM\bY \Big) \le B \tr\Big(\bU\bM\bU^{-1}\Big) = B \tr\Big(\bM\Big) \,.  
\end{align*}
\end{proof}

\subsection{Concentration results for random matrices}
\label{ssec:A_prob}

Here we give some concentration inequalities which enable us to control the spectrum of
sums of Hermitian positive semi-definite rank-$1$ matrices. In the finite case, these sums take the form
\begin{align}\label{A:sum}
\frac{1}{n} \sum_{i=1}^n  \bu^i \otimes \bu^i \quad\text{for a sequence of random vectors }\bu^i\in\field{C}^m,\,i\in[n] \,.
\end{align}
Recall the notation $[n]=\{1,\ldots,n\}$.
The basic assumption is always that the vectors $\bu^i$ are drawn i.i.d.\
according to some probability distribution and that a uniform bound $\|\bu^i\|_2\le M$ is satisfied almost surely for every $i\in[n]$.
Putting $\bA_i:=\frac{1}{n} \bu^i \otimes \bu^i$, we obtain a sequence $({\bA_i})_{i=1}^n$ of i.i.d.\ Hermitian positive semi-definite random matrices
which satisfy $\lamax(\bA_i)\le \frac{M^2}{n}$ almost surely.
In this situation, a matrix Chernoff inequality proved by Tropp~\cite[Thm.\ 1.1]{Tr11} can be applied. We derive the following form.

\begin{manuallemma}{A.3}[Matrix Chernoff, cf.~{\cite[Thm.~1.1]{Tr11}}]\label{matrixchernoff} For a sequence $(\bm A_i)_{i=1}^n\subset \field{C}^{m \times m}$ of independent, Hermitian, positive semi-definite random matrices satisfying $\lambda_{\max}(\bm A_i) \le R$ almost surely it holds
  \begin{align*}
    \Prob\bigg(
      \lamin\Big( \sum_{i=1}^{n} \bm A_i \Big) \le (1-t)\mu_{\min}
    \bigg)
    &\le {m}\exp\Big(-\frac{\mu_{\min}}{R} (t+(1-t)\log(1-t)) \Big) \\
    &\le {m}\exp\Big(-\frac{\mu_{\min}t^2}{2R} \Big)
  \end{align*}
  and
  \begin{align*}
    \Prob\bigg(
      \lamax\Big( \sum_{i=1}^{n} \bm A_i \Big) \ge (1+t)\mu_{\max}
    \bigg)
    &\le {m}\exp\Big(-\frac{\mu_{\max}}{R} (-t+(1+t)\log(1+t)) \Big) \\
    &\le {m}\exp\Big(-\frac{\mu_{\max}t^2}{3R} \Big) 
  \end{align*}
  for $t\in[0,1]$, where $\mu_{\min} \coloneqq \lambda_{\min} (\sum_{i=1}^{n} \mathds E \bm A_i)$ and $\mu_{\max} \coloneqq \lambda_{\max} (\sum_{i=1}^{n} \mathds E \bm A_i)$.
\end{manuallemma} 

\begin{proof} The first estimates are provided by \cite[Thm.~1.1]{Tr11}. Based on the Taylor expansion
\begin{align*}
(1+t)\log(1+t) = t + \sum_{k=2}^{\infty} \frac{(-1)^k}{k(k-1)} t^k  \,,
\end{align*}
which holds true for $t\in[-1,1]$, we can further derive the inequalities
\begin{flalign*}
 && t+(1-t)\log(1-t) &= \sum_{k=2}^{\infty} \frac{1}{k(k-1)} t^k \ge \frac{t^2}{2}  && \\
\text{and} &&  -t+(1+t)\log(1+t) &=  \sum_{k=2}^{\infty} \frac{(-1)^k}{k(k-1)} t^k \ge \frac{t^2}{2} - \frac{t^3}{6} \ge  \frac{t^2}{3} &&
\end{flalign*}
for the range $t\in[0,1]$.
\end{proof}

A concentration inequality for the case when the vectors $\bu^i$ in~\eqref{A:sum} are infinite dimensional is given in~\cite[Thm.~1.1]{MoUl20}. Here it is assumed that the $\bu^i$
are i.i.d.\ random sequences from $\ell_2(\field{N})$. Let us recite this result.

\begin{manualtheorem}{A.4}[{\cite[Thm.~1.1]{MoUl20}}]\label{C:thm2} Let \(\bu^i , i\in [n], \) be i.i.d.\ random sequences from \( \ell_2(\field{N})\). Let further $n \geq 3$, $M>0$ such that \(\| \bu^i \|_2 \leq M\) almost surely and \( \Ept(\bu^i \otimes  \bu^i)={\mathbf\Lambda}\) for $i\in [n]$ with $\|\boldsymbol{\Lambda}\|_{2\to 2} \leq 1$. Then
\[ \Prob \Big( \Big\| \frac{1}{n} \sum_{i=1}^n  \, \bu^i \otimes  \bu^i - {\mathbf\Lambda} \Big\|_{2 \to 2} \geq t \Big) \leq 2^\frac{3}{4} n\exp\Big(-\frac{t^2n}{21M^2}\Big)\,.\]
\end{manualtheorem}

A useful rephrasing of Theorem~\ref{C:thm2} is given by~\cite[Prop.~3.8]{MoUl20}. It is as follows.

\begin{manualprop}{A.5}[{\cite[Prop.~3.8]{MoUl20}}]\label{C:prop:main1}
Let \(\bu^i , i\in [n] \), be i.i.d.\ random sequences from \( \ell_2(\field{N})\). Let further $n \geq 3$, $r > 1$, $M>0$ such that \(\| \bu^i \|_2 \leq M\) almost surely and
\( \Ept(\bu^i \otimes  \bu^i)  =  {\mathbf\Lambda}\)  for all \(i\in [n]\). Then
\[ \Prob \Big( \Big\| \frac{1}{n} \sum_{i=1}^n  \, \bu^i \otimes  \bu^i - {\mathbf\Lambda} \Big\|_{2 \to 2} \geq F \Big) \leq 2^\frac{3}{4} \, n^{1-r}\,,\]
where \(F := \max\Big\{ \frac{8r \log n}{n} M^2 \nu^2 , \| {\mathbf\Lambda} \|_{2 \to 2} \Big\}\) and \( \nu = \frac{1+\sqrt{5}}{2}\) .
\end{manualprop}

\section{The considered RKHS setting}
\label{ssec:A_RKHS}

In Subsection~\ref{ssec:Appl_infinite} we consider functions from a RKHS $H(K)$ on a
non-empty measure space $(D,\nu)$. In the sequel, this setting is analyzed in more detail.
First note that the space $H(K)$, as a RKHS, consists of proper point-wise defined functions. By definition, it
is associated with a positive semi-definite Hermitian kernel $K:D \times D\to\field{C}$
such that the reproducing property
\begin{align}\label{eq:rep_property}
		f(x) = \langle f, K(\cdot,x) \rangle_{H(K)}
\end{align}
holds true for all $f\in H(K)$ and $x \in D$. 
Due to this property, sampling is not only a well-defined but even continuous operation in $H(K)$
(see e.g.~\cite{HeBo04}).

For our analysis of $L_2$-approximation, an embedding relation between $H(K)$ and $L_2(D,\nu)$ is
crucial. It is guaranteed by the presumed finite trace of $K$, namely
\begin{equation*}
	\tr(K):= \int_{D} K(x,x) \,d\nu(x) < \infty \,.
\end{equation*}
Under such a condition,
see \cite{HeBo04} and \cite[Lem.~2.3]{StSc12},
for every $f\in H(K)$
\begin{align*}
\|f\|^2_{2}
=   \int_D |\langle f,K(\cdot,x)\rangle|^2 \,d\nu(x) \le
\int_D \|f\|_{H(K)}^2 \|K(\cdot,x)\|_{H(K)}^2 \,d\nu(x) =  \|f\|_{H(K)}^2 \cdot \tr(K) \,.
\end{align*}
As a consequence, there is a compact embedding
\begin{equation*}\Id_{K,\nu}:H(K) \hookrightarrow L_2(D,\nu) \,,
\end{equation*}
which can be shown to be even Hilbert-Schmidt (see~\cite[Lem.~2.3]{StSc12}).
Here, it is important to note that, in contrast to $H(K)$, the elements of $L_2(D,\nu)$ are not functions 
but $\nu$-equivalence classes, with two functions $f,\tilde{f}:D\to \field{C}$ considered $\nu$-equivalent if $f(x)=\tilde{f}(x)$ for $\nu$-almost every $x\in D$. $\Id_{K,\nu}$ may hence not be injective.
In fact, $\Id_{K,\nu}$ is the restriction of $\Id_\nu$ to $H(K)$, where $\Id_\nu$ is the map
that assigns to every $f:D\to \field{C}$ the corresponding $\nu$-equivalence class.
Depending on the measure $\nu$, the null-space $\cN(\Id_{K,\nu})$ can thus be non-trivial. 
The choice $\nu=0$ illustrates this, where $\cN(\Id_{K,\nu})=H(K)$.

Without loss of generality, to simplify the considerations in Subsection~\ref{ssec:Appl_infinite}, it is further assumed that the subspace $\Id_{K,\nu}(H(K))$ of $L_2(D,\nu)$ is infinite dimensional (i.e.\ that $\Id_{K,\nu}$ has infinite rank). Under this condition, the sequence $(\sigma_k)_{k\in \field{N}}$ of strictly positive singular numbers associated to $\Id_{K,\nu}$ is countably infinite. For this, note that $\Id_{K,\nu}(H(K))$ is a separable subspace of $L_2(D,\nu)$ due to the compactness of $\Id_{K,\nu}$.

We now fix orthonormal systems $(\eta_k)_{k\in \field{N}} \subset L_2(D,\nu)$ and $(e_k)_{k\in \field{N}}\subset H(K)$ of associated left and right singular functions
such that
\begin{align}\label{eq:relation_lr}
e_k = \sigma_k \eta_k  \quad\text{for all}\quad k\in \field{N} \,.
\end{align}
Whereas each $e_k$ represents a point-wise function on $D$, the left singular functions $\eta_k$, as elements of $L_2(D,\nu)$, refer to $\nu$-equivalence classes. However,
we choose for each $\eta_k$ the specific representative $e_k/\sigma_k$ so that both systems $(\eta_k)_{k\in \field{N}}$ and $(e_k)_{k\in \field{N}}$ are comprised of proper functions satisfying~\eqref{eq:relation_lr} in a point-wise sense.

Let us next ask for basis properties of these systems. Clearly, $(\eta_k)_{k\in \field{N}}$ is a basis for $\Id_{K,\nu}(H(K))$.
The system $(e_k)_{k\in \field{N}}$, on the other hand, is usually not a basis for $H(K)$ since it only corresponds to the non-trivial singular numbers of $\Id_{K,\nu}$.
Under additional restrictions ensuring $\cN(\Id_{K,\nu})=\{0\}$ it would be, e.g.\ if the kernel $K$ is continuous and bounded (i.e.\ a Mercer kernel).
Generally, $H(K)$ decomposes in the form
\begin{align*}H(K) = \overline{\spn}\{e_1, e_2, \ldots \} \oplus_{\text{orth}}^{H} \cN(\Id_{K,\nu}) \,.
\end{align*}

A useful representation of $K$ in terms of the functions $(e_k)_{k\in \field{N}}$ can be obtained as follows.
For each $y\in D$, first expand the function $K(\cdot,y)\in H(K)$ in the form
\begin{align*}
K(\cdot,y) = \sum_{k\in \field{N}} c_k e_k +  r_{y}
\end{align*}
with an associated function $r_{y}\in\cN(\Id_{K,\nu})$ and convergence of the sum in $H(K)$.
For the coefficients calculate with~\eqref{eq:rep_property}
\begin{align*}
c_k=\langle K(\cdot,y), e_k \rangle_{H(K)} = \overline{e_k(y)} \,.
\end{align*}
Since convergence in $H(K)$ entails point-wise convergence, we obtain the representation
\begin{align*}K(x,y) = \sum_{k\in \field{N}} e_k(x) \overline{e_k(y)} +  r_{y}(x) \,.
\end{align*}

In case that $H(K)$ is separable, which is assumed in Theorem \ref{thm:approx}, it is shown in \cite{HeBo04} and \cite[Cor.~3.2]{StSc12} that $r_x(x)$ vanishes $\nu$-almost everywhere. We hence have $K(x,x) = \sum_{k\in \field{N}} |e_k(x)|^2 $ for $\nu$-almost every $x\in D$. This is a crucial ingredient in the proof of Theorem \ref{thm:approx}, see also \cite[Sec.~4]{MoUl20} and 
\cite[Sec.~2]{KaVoUl21}.

\end{document}